\documentclass[10pt]{amsart}
\usepackage{amsmath}
\usepackage{amsfonts}
\usepackage{amssymb}
\usepackage{graphicx}
\usepackage{amsthm,graphicx,color,yfonts}
\usepackage{pdfsync}
\usepackage{epstopdf}%

\usepackage[colorlinks=true]{hyperref}
\hypersetup{linkcolor=black,citecolor=black,filecolor=black,urlcolor=black} 

\providecommand{\U}[1]{\protect\rule{.1in}{.1in}}

\newcommand{\wt}{\widetilde}

\newcommand{\R} {\mathbb R}
\newcommand{\cuad}{{\sqcap\kern-.68em\sqcup}}

\newcommand{\dist}{{\rm dist}\, }

\newcommand{\ve}{\varepsilon}

\newcommand{\be}{\begin{equation}}
\newcommand{\ee}{\end{equation}}

\newcommand{\equ}[1]{(\ref{#1})}

\long\def\red#1{{\color{black}#1}}

\numberwithin{equation}{section}

\newtheorem{theorem}{Theorem}[section]
\newtheorem{proposition}{Proposition}[section]
\newtheorem{corollary}{Corollary}[section]
\newtheorem{lemma}{Lemma}[section]
\newtheorem{definition}{Definition}[section]

\newtheorem{remark}{Remark}[section]

\title[The classification of four end solutions]{The classification of four end solutions to the Allen-Cahn equation on  the plane}

\author{Micha{\l} Kowalczyk}
\address{Micha{\l} Kowalczyk, Departamento de Ingenier\'{\i}a  Matem\'atica and Centro de Modelamiento Matem\'atico (UMI 2807 CNRS), Universidad de Chile, Casilla 170 Correo 3, Santiago, Chile.}
\email {kowalczy@dim.uchile.cl}

\author{Yong Liu}
\address{\noindent  Y. Liu - Departamento de
Ingenier\'{\i}a  Matem\'atica and Centro de Modelamiento Matem\'atico (UMI 2807 CNRS), Universidad de Chile,
Casilla 170 Correo 3, Santiago,
Chile.}
\email{yliu@dim.uchile.cl}

\author{Frank Pacard}
\address{Frank Pacard, Centre de Math\'ematiques Laurent Schwartz,  \'Ecole Polytechnique,  91128 Palaiseau, France et Institut Universitaire de France}
\email{frank.pacard@math.polytechnique.fr}

\begin{document}

\maketitle

\begin{abstract}
An entire  solution of the Allen-Cahn equation $\Delta u=f(u)$, where $f$ has exactly three zeros at $\pm 1$ and $0$, is balanced and odd, e.g. $f(u)=u(u^2-1)$, is called a $2k$-ended solution if its nodal set is asymptotic to $2k$ half lines, and if along each of these half lines the function $u$ looks like the one dimensional, heteroclinic solution.
In this paper we consider the family of four ended solutions  whose ends are almost parallel at $\infty$. We show that this family can be parametrized by the family of  solutions of the two component Toda system. As a result we obtain the uniqueness of  four ended solutions with almost parallel ends. Combining this result with the classification of connected components in the moduli space of the four ended solutions we  can classify all such solutions.  
Thus we show  that four end solutions form, up to rigid motions, a one parameter family. This family contains the saddle solution, for which the angle between the nodal lines is $\frac{\pi}{2}$ as well as solutions for which the angle between the asymptotic half lines of the nodal set is arbitrary small (almost parallel nodal sets). 
\end{abstract}

\section{Introduction}
\subsection{Some entire solutions to the Allen-Cahn equation in $\R^2$}

This paper deals with the problem of classification of the family of four end solutions  (precise definition will follow) to the Allen-Cahn equation:
\begin{equation}
\Delta u=F^{\prime}\left(  u\right)  \text{ in }\mathbb{R}^{2}.\label{AC}%
\end{equation}
The function  $F$ is a smooth,  double well potential, which means that  we assume the
following conditions for $F:$ $F$ is even, nonnegative and has only two zeros
at $\pm1$. Moreover $F^{\prime\prime}\left(  1\right)  \neq0$ and $F^{\prime
}\left(  t\right)  \neq0,t\in\left(  0,1\right)$. We also suppose $F''(0)\neq 0$. 


It is know that (\ref{AC}) has a solution whose nodal set is a straight line, it will be called {\it a planar solution}. It is simply obtained by taking the unique, odd, heteroclinic solution connecting $-1$ to $1$:
\begin{align}
\label{heteroclinic}H''=F'(H), \quad H(\pm \infty)=\pm 1, \quad H(0)=0,
\end{align}
and letting $u(x,y)=H(ax+by+c)$ for some constants $a,b, c$ such that $a^2+b^2=1$. We note that if, say $a>0$ then $\partial_x u=aH'>0$.  De Giorgi conjecture says that if $u$ is any smooth and bounded solution of (\ref{AC}) such that $\partial_{\mathtt e} u>0$ for certain fixed direction ${\mathtt e}$ then in fact $u$ must be a planar solution. Indeed this conjecture holds in $\R^N$, $N\leq 8$  (\cite{MR1637919} when $N=2$, \cite{MR1775735} when $N=3$, and \cite{MR2480601}, for $4\leq N\leq 8$ under additional limit condition), while a counterexample can be given when $N\geq 9$ \cite{dkp_dg}.  It is worth mentioning that the De Giorgi conjecture is a direct analog of the famous  Bernstein conjecture in the theory of minimal surfaces.

 In order to proceed with  the statement of our results we  will define the family  of  four ended solutions of (\ref{AC}), which  is a particular example of a more general family of $2k$ ended solutions \cite{dkp-2009}.   Intuitively,  a  four ended solution $u$ is characterized by the fact that its nodal set $N(u)$ is asymptotic at infinity to four half lines, and along each of this half lines it looks locally like the heteroclinic solution.  To describe this precisely we introduce  the set $\Lambda_4$ of   {oriented and ordered four affine lines} in $\mathbb R^2$. Thus $\Lambda_4$ consists of $4$- tuples $(\lambda_1, \dots, \lambda_4)$ such that  each  $\lambda_j$ can be uniquely written as
\[
\lambda  :=   r_j  \, {\tt e}_j^\perp + \mathbb R  \, {\tt e}_j ,
\]
for some $r_j\in \mathbb R$ and some unit vector ${\tt e}_j =(\cos\theta_j, \sin\theta_j)\in S^1$, which defines the orientation of the line.  Recall that we denote by $\perp$ the rotation of angle  $\pi/2$ in $\mathbb R^2$.  
Observe that the affine lines are oriented and hence we do not identify the line corresponding to $(r_j, \theta_j)$ and the line corresponding to $(-r_j, \theta_j +\pi)$. 
Additionally we require these lines are ordered, which means:
%
%
\[
\theta_1 < \theta_2 < \theta_3 < \theta_{4} <  2\, \pi + \theta_1 .
\]
For future purpose  we denote by
\begin{equation}
\label{min angle}
 \theta_\lambda : = \frac{1}{2} \, \min \{ \theta_{2}- \theta_1 ,  \theta_{3} - \theta_{2} ,\theta_4-\theta_3, 2 \, \pi + \theta_1 - \theta_{4} \} ,
\end{equation}
the half of the minimum of the angles between any  two consecutive oriented affine lines $\lambda_1, \ldots, \lambda_4$.

Assume that we are given a $4$-tuple of oriented affine lines  $\lambda = (\lambda_1, \ldots, \lambda_4)$ . It is easy to check that for all $R > 0 $ large enough and for all $j=1, \ldots, 4$, there exists $s_j \in \mathbb R$  such that~:
\begin{itemize}
\item[(i)] The point ${\tt x}_j: =  r_j \,  {\tt e}^\perp_{j} + s_j \, {\tt e}_j$ belongs to the circle $\partial B_R$, with $R > 0$. \\
\item[(ii)] The half lines
\begin{equation}
\lambda_j^+ := {\tt x}_j + \mathbb R^+ \,  {\tt e}_j ,
\label{eq:halfline}
\end{equation}
are disjoint and included in $\mathbb R^2\setminus  B_R$.\\
\item[(iii)] The minimum of the distance between two distinct half lines $\lambda^+_i$ and $\lambda_j^+$ is larger than $4$.
\end{itemize}

The set of half affine lines $\lambda_1^+, \ldots, \lambda_{4}^+$ together with the  circle $\partial B_R$ induce a decomposition of $\mathbb R^2$ into $5$  slightly overlapping connected components
\[
\mathbb R^2  = \Omega_0 \cup \Omega_1 \cup \ldots \cup \Omega_{4} ,
\]
where
\[
\Omega_0 : =  B_{R+1},
\]
and where, for $j=1, \ldots, 4$,
\begin{align}
\label{decomp 1}
\Omega_j : = \left\{ {\tt x} \in \mathbb R^2 \, : \,   |{\tt x}| > R-1 \quad \mbox{and} \quad   {\rm dist} ({\tt x}, \lambda_j^+)  <   {\rm dist} ({\tt x}, \lambda_i^+) +2 ,  \quad \forall i \neq j \,\right\} ,
\end{align}
where $ {\rm dist} ({\tt x}, \lambda_j^+)$ denotes the distance to $\lambda_j^+$. Observe that, for all $j=1, \ldots, 4$, the set  $\Omega_j$ contains the half line $\lambda_j^+$.

We define ${\mathbb I}_0, {\mathbb I}_1,  \ldots , {\mathbb I}_4$,  a smooth partition of unity of  $\mathbb R^2$ which is subordinate to the above decomposition of $\mathbb R^2$. Hence
\[
\sum_{j=0}^4  \mathbb I_j \equiv 1,
\]
and the support of $\mathbb I_j$ is included in $\Omega_j$, for $j=0, \ldots, 4$. Without loss of generality, we can also assume that  ${\mathbb I}_0\equiv 1$ in
\[
\Omega'_0 := B_{R -1},
\]
and ${\mathbb I}_j\equiv 1$ in
\[
\Omega'_j := \left\{ {\tt x} \in \mathbb R^2 \, : \,   |{\tt x}| > R+1 \quad \mbox{and} \quad   {\rm dist} ({\tt x}, \lambda_j^+)  <   {\rm dist} ({\tt x}, \lambda_i^+) - 2 ,  \quad \forall i \neq j \,\right\}  ,
\]
for $j=1, \ldots, 4$. Finally, we assume that
\[
\|{\mathbb I}_{j}\|_{\mathcal C^2(\mathbb R^2)} \leq C  .
\]

We now take $\lambda  = (\lambda_1, \ldots, \lambda_4 ) \in \Lambda_4$
with $\lambda_j^+ = {\tt x}_j  + \mathbb R^+ \, {\tt e}_j$ and we define
\begin{align}
u_\lambda ({\tt x}): =  \sum_{j=1}^{4}  (-1)^j \, {\mathbb I}_j ({\tt x})  \, H (({\tt x}-{\tt x}_j)\cdot {\tt e}^\perp_j).
\label{def w}
\end{align}

Observe that, by construction, the function $u_\lambda$ is, away from a compact, asymptotic to copies of planar solutions whose nodal set are the half affine lines $\lambda_1^+, \ldots,  \lambda_{4}^+$. A simple computation shows that $u_\lambda$ is not far from being a solution of (\ref{AC}) in the sense that $\Delta \, u_\lambda - F'(u_\lambda)$ is a function which decays exponentially to $0$ at infinity (this uses the fact that $\theta_\lambda >0$).

In this paper we are interested in four ended solutions of (\ref{AC})  which means that they are asymptotic to a function $u_\lambda$ for some choice of $\lambda  \in  \Lambda_4$.  More precisely, we have the~:
\begin{definition}
Let $\mathcal S_4$ denote the set of functions $u$ which are defined in $\mathbb R^2$ and which satisfy
\begin{equation}
u  - u_\lambda  \in   W^{2,2} \, (\mathbb R^2) ,
\label{ass w}
\end{equation}
for some $\lambda \in \Lambda_4$. We also define the {\rm decomposition operator} $\mathcal J$ by
\[
\begin{array}{rcccllll}
\mathcal J  : & \mathcal S_{4}  & \longrightarrow &   W^{2,2}  (\mathbb R^2) \times \Lambda_4 \\[3mm]
                    & u &\longmapsto &  \left( u- u_\lambda , \lambda  \right)  .
\end{array}
\]
The topology on $\mathcal S_4$ is the one for which the operator $\mathcal J$ is continuous (the target space being endowed with the product topology). We define the set of four ended solutions of the Allen-Cahn equation $\mathcal M_4$ to be the set of solutions $u$ of (\ref{AC}) which belong to $\mathcal S_4$.
\label{de:001}
\end{definition}


The set $\mathcal M_4$ is nonempty. Indeed,  it is known that $\left(  \ref{AC}\right)  $ has a saddle solution $U,$ which
is bounded and symmetric:
\[
U\left(  x,y\right)  =U\left(  x,-y\right)  =U\left(  -x,y\right)  .
\]
Moreover, the nodal set of $U$ is exactly the lines $y=\pm x.$ Along these two
lines, $U$ converges exponentially fast to the \textquotedblleft
heteroclinic\textquotedblright\ solution. In addition in \cite{MR2557944} it is shown that there exists a small number $\varepsilon_0$ such that for all $0<\theta$, with $\tan\theta< \ve_0$ there exists a four ended solution with corresponding  angles  of the half lines $\lambda^+_j$, $j=1,\dots, 4$ given by
\[
\theta_1=\theta, \quad\theta_2=\pi-\theta, \quad\theta_3=\theta+\pi, \quad \theta_4=2\pi-\theta.
\] 
Observe that the fact that  $\theta$ is small implies that the ends of this solution are almost parallel and their slopes given by $\pm\ve$ are small as well. Clearly, by symmetry, it is easy to see that there exist also solutions with parallel ends whose angles are given by:
\[
\theta_1=\pi/2-\theta, \quad\theta_2=\pi/2+\theta, \quad\theta_3=-\theta+3\pi/2, \quad \theta_4=3\pi/2+\theta.
\] 
In this case we would have $\tan\theta>\frac{1}{\ve_0}$. 

Clearly, any four ended solution can be translated and rotated, yielding another four ended solution.  In fact, by a result of Gui \cite{2011arXiv1102.4022G} we know that any $u\in {\mathcal M}_4$ is, modulo rigid motions, and a multiplication of a solution by $-1$, even in its variables, and monotonic in $x$ in the set $x>0$, and in $y$ in the set $y<0$ i.e.:
\begin{align}
\label{even}
u(x,y)=u(-x,y)=u(x,-y), \quad u_x(x,y)>0,\quad  x>0, \quad u_y(x,y)>0, \quad y<0.
\end{align}
Thus, when studying four ended solutions,   it is natural to consider the set ${\mathcal M}_4^{even}\subset {\mathcal M}_4$, consisting precisely of functions satisfying (\ref{even}).  With each such function $u$ we may associate in a unique way the angle that the component of its nodal set in the first quadrant makes with the $x$-axis. Thus we can define  {\it{the angle  map}}:
\begin{equation}
\label{def thetau}
\begin{aligned}
&\theta\colon {\mathcal M}_4^{even}\to (0,\frac{\pi}{2}),\\
&u\longmapsto \theta.
\end{aligned}
\end{equation}
In principle the value of the angle map is not enough to identify in a unique way a solution to (\ref{AC}) in ${\mathcal M}_4^{even}$.  However  for solutions with almost parallel ends we have the following:
\begin{theorem}\label{teo uniqueness}
There exists a small number $\ve_0$ such that for any two solutions $u_1, u_2\in {\mathcal M}_4^{even}$ satysfying  $\tan\theta(u_1)=\tan(\theta_2)=m$,  and either $m<\ve_0$ or $m>\frac{1}{\ve_0}$, we have necessarily $u_1\equiv u_2$.
\end{theorem} 

This result gives in some sense  classification of the subfamily of  the family of four end solutions which contains solutions with  almost parallel ends. It says that this subfamily consists   precisely of the solutions  constructed in \cite{MR2557944}. Let us explain the importance of this statement from the point of view of  classification of all four end solutions. We will appeal to the following theorem  proven in \cite{partI}:
\begin{theorem}\label{part I}
Let $M$ be any connected component of $ {\mathcal M}_4^{even}$ . Then the angle map $\theta\colon M\to (0, \frac{\pi}{2})$ is surjective.
\end{theorem} 
Consider for example the connected component $M_0\subset {\mathcal M}_{4}^{even}$ which contains the saddle solution $U$. Theorem \ref{part I} implies that $U$ can be deformed along $M_0$  to a solution with the value of the angle map arbitrarily close to $0$ or to $\frac{\pi}{2}$, thus yielding a solution in the subfamily of the solutions with almost parallel ends. But these solutions are uniquely determined by the value of the angle map, which follows from the uniqueness statement in Theorem \ref{teo uniqueness}. As a result we obtain the following classification theorem:
\begin{theorem}\label{classification}
Any solution  $u\in {\mathcal M}_{4}^{even}$ belongs to $M_0$ and  is a continuous deformation of the saddle solution $U$.
\end{theorem}

To appreciate further this result let us mention that in \cite{partI} we prove a more general statement regarding arbitrary connected components in the moduli space of solutions ${\mathcal M}_4^{even}$. To explain this we consider the following map 
\begin{align*}
P\colon {\mathcal M}_4^{even}&\to (0, \frac{\pi}{2})\times (-1,1)\subset \R^2,\\
u&\longmapsto (\theta(u), u(0)).
\end{align*}
Then, according to Theorem 1 in \cite{partI}, the image of any connected component $M\subset {\mathcal M}_4^{even}$ under this map $P(M)$ is an embedded, smooth curve in $(0, \frac{\pi}{2})\times (-1,1)$. Thus, Theorem 
\ref{classification} implies that the set $P({\mathcal M}_4^{even})$ consists of a single embedded curve in $(0, \frac{\pi}{2})\times (-1,1)$.  

We observe that according to  the conjecture of De Giorgi  in two dimensions any  bounded solution $u$ which is monotonic in one direction must be one dimensional and equal to $u({\tt x})=H({\tt a}\cdot {\tt x}+b)$. In the language of multiple end solutions, this solution has {\it two} (heteroclinic, planar) ends. Theorem \ref{classification} gives on the other hand the classification of the family of solutions with {\it four} planar ends.  Since the number of ends of a solution to (\ref{AC}) must be even, the family of four ended solutions is the natural object to study. In this context, one may wonder if it is possible to classify solutions to (\ref{AC}) assuming for instance that the nodal sets of $u_x$, and $u_y$ have just one component. This question is beyond the scope of this paper, however since partial derivatives of four ended solution satisfy this assumption it seems reasonable to conjecture that a result similar to Theorem \ref{classification} should hold in this more general setting. We should mention here that it is in principle possible  to study the problem of classification of solutions assuming for example that their Morse index is $1$. This is natural since the Morse index of $u$ and the number of the nodal domains of $u_x$ and $u_y$ are related. We recall here that the heteroclinic is stable, and from  \cite{MR2110438} we know that in dimension $N=2$ stability of a solution  implies that it is necessarily a one dimensional solution (for the related minimality conjecture, see for example  \cite{MR2480601} and  \cite{wei_pacard} and the reference therein). We expect that in fact the family of four ended solutions should should contain all   multiple end solutions with Morse index $1$ \cite{partIII} (the Morse index of the saddle solution is $1$ \cite{MR1363002}). 

Let us now explain the analogy of Theorem \ref{classification} with some aspects of the theory minimal surfaces in $\R^3$. In 1834, Scherk discovered an example of singly-periodic, embedded, minimal surface in $\R^3$ which, in a complement of a vertical cylinder, is asymptotic to $4$ half planes with angle $\frac{\pi}{2}$ between them. This surface, after a rigid motion, has two planes of symmetry, say  $\{x_2=0\}$ plane and  $\{x_1=0\}$, and it is periodic, with period $1$ in the $x_3$ direction. If $\theta$ is  the angle between the asymptotic end of the Scherk surface contained in $\{x_1>0, x_2>0\}$ and the $\{x_2=0\}$ plane  by then $\theta=\frac{\pi}{4}$. This is the so called second Scherk's surface and it will denoted here by $S_{\frac{\pi}{4}}$. In 1988 Karcher \cite{MR958255} found Scherk surfaces other than the original example in the sense that the corresponding angle between their asymptotic planes and the $\{x_2=0\}$ plane can be any $\theta\in (0, \frac{\pi}{2})$. The one parameter family $\{S_\theta\}_{\{0<\theta<\frac{\pi}{2}\}}$ of these surfaces is the family of  Scherk singly periodic minimal surfaces. Thus, accepting that the saddle solution of the Allen-Cahn equation $U$ corresponds to the Scherk surface $S_{\frac{\pi}{4}}$  Theorem \ref{part I} can be understood as an analog of the result of Karcher. We note that, unlike in the case of the Allen-Cahn equation, the Scherk family is given explicitly, for example it can be represented as the zero level set of the function:
\begin{align*}
F_\theta(x_1,x_2, x_3)=\cos^2\theta\cosh\big(\frac{x_1}{\cos\theta}\big)-\sin^2\theta\cosh\big(\frac{x_2}{\cos\theta}\big)-\cos x_3.
\end{align*}
From this it follows immediately that the angle map in this context  $S_\theta \mapsto \theta$ is a diffeomorphism. A corresponding result for the family ${\mathcal M}_4^{even}$ is of course more difficult since no explicit formula is available in this case.

We will explore further  the analogy of our result with the theory of minimal surfaces in $\R^3$, now in the context of the classification of the four ended solutions in Theorem \ref{classification}. The corresponding problem can be stated as follows: if  $S$ is an embedded, singly periodic,  minimal surface with $4$ Scherk ends, what can be said about this surface ? It is proven by Meeks and Wolf \cite{MR2276776} that $S$ must be one of the Scherk surfaces $S_\theta$ described above (similar result is proven in \cite{MR2262839} assuming additionally that the genus of $S$ in the quotient $\R^3/{\mathbb Z}$ is $0$). The key results to prove this general statement are  in fact the counterparts of  Theorem \ref{teo uniqueness} and  Theorem \ref{part I}.

We now sketch the basic elements in the proofs of Theorem \ref{teo uniqueness}. First of all let us explain the existence result in \cite{MR2557944}. The point of departure of the construction is the following Toda system 
\begin{equation}
\begin{aligned}
{{c_0}} \, q_1''  &=  - e^{\sqrt {F''(1)} ( q_{1}- q_{2}) },\\
{{c_0}}\, q_2''  &= e^{\,\sqrt {F''(1)} (q_{1}- q_2)},
\end{aligned}
\label{toda 02}
\end{equation}
for which $q_1<0<q_2$ and $q_1(x)=-q_2(x)$, as well as $q_j(x)=q_j(-x)$, $j=1,2$. Here $c_0$ is a fixed constant depending on $F$ only (when $F(u)=\frac{1}{4}(1-u^2)^2$, then  $c_0=\frac{\sqrt{2}}{24}$). Any solution of this system is asymptotically linear, namely:
\[
q_{j}(x)=(-1)^j (m |x|+b)+{\mathcal O}(e^{\, -2\sqrt{F''(1)}m |x|}), \quad x\to \infty,
\]
where $m>0$ is the slope of the asymptotic straight line in the first quadrant.  On the other hand, given  that we only consider solutions whose trajectories are symmetric with respect to the $x$-axis, the value of the slope $m$ determines the unique solution of (\ref{toda 02}).  When the asymptotic lines become  parallel then $m\to  0$ or $m\to \infty$. By symmetry it suffices to consider the case $m\to 0$ and in this paper we will denote small slopes by $m=\ve$ and the corresponding solutions by $q_{\ve, j}$.  Note that if by $q_{1,j}$ we denote a solution with slope $m=1$ then 
\[
q_{\ve,j}(x)=q_{1, j}(\ve x)+\frac{(-1)^{j}}{\sqrt{F''(1)}}\log\frac{1}{\ve} .
\]
Then, the existence result in \cite{MR2557944} implies  that given a small $\ve$, there exists a four ended solution $u$  to (\ref{AC}) whose nodal set $N(u)$ is close to the trajectories of the Toda system given by the graphs $y=q_{\ve, j}(x)$. Although we do not use directly this result in the present paper but the idea of relating solutions of the Toda system and the four ended solutions of (\ref{AC}) that  comes from  \cite{MR2557944} is very important. In fact, what we want to achieve is to parametrize the manifold of four ended  solutions with almost parallel ends using corresponding solutions of the Toda system as parameters. To do this in Section \ref{nodal asymp} we obtain  a very precise control of  the nodal sets of the four ended solutions. The key observation is that in every quadrant the nodal  set  $N(u)$ of any four ended solution is a bigraph, and if we assume that the slope of its asymptotic lines is small then it is a graph of a smooth function, both in the lower and in the upper half plane. We have then 
\[
N(u)=\{(x,y)\in \R^2\mid y=f_{\ve, j}(x), \quad j=1,2, \quad f_{\ve, 1}(x)<0, \quad f_{\ve,2}(x)=-f_{\ve, 1}(x)\}, 
\]
for any $u\in {\mathcal M}_{4}^{even}$, with $\ve=\tan\theta(u)$.  Our main result in Section \ref{nodal asymp} says that  for each $\ve$ small
\[
f_{\ve, 1}(x)-q_{\ve, 1}(x)={\mathcal O}(\ve^\alpha e^{\, -\tau\ve|x|}), 
\]
with some positive constants $\alpha, \tau$. Next, we define   (Section \ref{toda aprox}) a suitable approximate four ended solution based on the solution of the Toda system with slope $\ve$. To explain this by $\widetilde\Gamma_{\ve, 1}$ we denote the graph of the function $y=q_{\ve, 1}(x)$, which is contained in the lower half plane. In a suitable neighborhood of the curve $\widetilde\Gamma_{\ve, 1}$ we introduce Fermi coordinates ${\tt x}=(x, y)\mapsto (x_1, y_1)$, where $y_1$ denotes  the signed distance to $\widetilde\Gamma_{\ve, 1}$, and $x_1$ is the $x$ coordinate of the projection of the point ${\tt x}$ onto $\widetilde\Gamma_{\ve, 1}$. With this notation we write locally the solution $u$, with $\ve=\tan \theta(u)$ in the form
\[
u({\tt x})=H(y_1-h_{\ve}(x_1))+\phi.
\]
This definition is suitably adjusted to yield a globally defined  function. 
Then it  is proven in Section \ref{toda aprox} that   $h_\ve\colon \R\to \R$ and $\phi\colon\R^2\to \R$ are  small functions, of order ${\mathcal O}(\ve^\alpha)$ in some weighted norms. 

Finally in Section \ref{last step} we consider two solution $u_1$, and $u_2$ such that $\tan\theta(u_1)=\tan\theta(u_2)=\ve$. We use the results of the previous section to prove that the function $u_1-u_2$ is Lipschitz and its norm can be controlled by the norms of $h_{1, \ve}-h_{2, \ve}$ and $\phi_1-\phi_2$ multiplied by a small constant. But on the other hand, using an argument similar to Lyapunov-Schmidt reduction we show that the difference  $h_{1,\ve}-h_{2, \ve}$ is controlled $u_1-u_2$, which eventually implies that $h_{1, \ve}=h_{2, \ve}$, and thus yields uniqueness. 

\section{Preliminaries}\label{sec prelim}

In this section we collect some facts about the Allen-Cahn equation which will
be used later on.
\subsection{Refined asymptotics theorem for  four ended solutions}\label{summary moduli}

Let $H(x)$ be the  heteroclinic solution on the Allen-Cahn equation and let us denote $\alpha_0=\sqrt{F''(1)}$. It is known that we have asymptotically:
\[
H(x)=1-Ae^{\,-\alpha_0 x}+{\mathcal O}(e^{\,-2\alpha_0 x}), \quad H'(x)=A\alpha_0 e^{\,-\alpha_0 x}+{\mathcal O}(e^{\,-2\alpha_0 x}),\quad x\to \infty,
\]
with similar estimates when $x \to-\infty$.

We consider  the linearized operator
\begin{align*}
L_0\phi=-\phi''+F''(H)\phi.
\end{align*}
It is known that the principal eigenvalue of this operator $\mu_0=0$ and the corresponding eigenfunction is $H'$. In general, the operator  $L_0$ has, possibly infinite, discrete spectrum $0<\mu_1<\dots\leq \alpha_0^2$, and the continuous spectrum which is $[\alpha_0^2,\infty)$, $\alpha_0=\sqrt{F''(1)}$. It may also happen that $L_0$ has just one eigenvalue, $\mu_0=0$ and the continuous spectrum, in which case we will set $\mu_1=\alpha_0^2$.

Next, we recall some facts about the moduli space theory developed in
\cite{dkp-2009}. We will mostly use this theory in  the case of four end solutions, thus we will restrict the presentation to this situation only.
We keep the notations introduced above. Thus we let
\[
\lambda = (\lambda_1, \ldots, \lambda_{4} ) \in \Lambda_{4},
\]
we write $ \lambda_j^+ = {\tt x}_j + \mathbb R^+ \, {\tt e}_j$ as in (\ref{eq:halfline}). We denote by $\Omega_0, \ldots , \Omega_{4}$ the decomposition of $\mathbb R^2$ associated to this $4$  half affine lines and  $\mathbb I_0, \ldots, \mathbb I_{4}$ the partition of unity subordinate to this partition. Given  $\gamma, \delta \in \mathbb R$, we define a {\em weight function} $\Gamma_{\gamma, \delta}$ by
\begin{equation}
\Gamma_{\gamma, \delta} ( {\tt x} ) : = {\mathbb I}_0 ({\tt x}) + \sum_{j=1}^{4} \,  {\mathbb I}_j  ({\tt x})  \, e^{{\gamma} \, ({\tt x} - {\tt x}_j) \cdot {\tt e}_j } \, \left( \cosh ( ({\tt x} - {\tt x}_j) \cdot {\tt e}_j^\perp ) \right)^{{\delta}} ,
\label{weight}
\end{equation}
so that, by construction, $\gamma$ is the rate of decay or blow up along the half lines $\lambda_j^+$ and $\delta$ is the rate of decay or blow up in the direction orthogonal to  $\lambda_j^+$.

With this definition in mind, we define the weighted Lebesgue space
\begin{equation}
L_{\gamma, \delta}^2 (\mathbb R^2) : =  \Gamma_{\gamma, \delta}  \, L^2 (\mathbb R^2) ,
\label{l2weight}
\end{equation}
and the weighted Sobolev space
\begin{equation}
W_{\gamma, \delta}^{2,2} (\mathbb R^2) : =  \Gamma_{\gamma, \delta} \, W^{2,2} (\mathbb R^2).
\label{S2weight}
\end{equation}
Observe that, even though this does not appear in the notations, the partition of unity, the weight  function and the induced weighted spaces  all depend on the choice of $\lambda \in \Lambda_{4}$.

Our first result shows that, if $u$ is a solution of (\ref{AC}) which is close to  $u_\lambda$ (in $W^{2,2}$ topology) then $ u -u_\lambda$ tends to $0$ exponentially fast at infinity.
\begin{proposition}[Refined Asymptotics]\label{refined asymp}
Assume that $u \in \mathcal S_{4}$ is a solution of (\ref{AC})  and define  $\lambda  \in \Lambda_{4}$, so that
\[
u - u_\lambda \in W^{2,2}  (\mathbb R^2) .
\]
Then, there exist $ \delta \in (0, \alpha_0 )$, $\alpha_0=\sqrt{F''(1)}$,  and  $\gamma > 0$ such that
\begin{equation}
u - u_\lambda \in W^{2,2}_{-\gamma,  -\delta}  (\mathbb R^2) .
\label{eq:refined-2}
\end{equation}
More precisely, $ \delta > 0$ and $ \gamma > 0$ can be chosen so that
\begin{align}
\label{bargammadelta}
\gamma \in (0, \sqrt{\mu_1})  , \qquad  \gamma^2 +  \delta^2 < \alpha_0^2 \qquad \mbox{and}  \qquad \alpha_0 >  \delta + \gamma  \, \cot \theta_\lambda  ,
\end{align}
where $ \theta_\lambda$ is equal to the half of the  minimum of the angles between two consecutive oriented affine lines $\lambda_1, \ldots, \lambda_{4}$ (see~(\ref{min angle})) and $\mu_1$ is the second eigenvalue of the operator $L_0$ (or $\mu_1=\alpha_0^2$ if $0$ is the only eigenvalue).
\end{proposition}
We observe here that it is well known that for any solution of (\ref{AC}) the following is true: if by $N(u)$ we will denote the nodal set of $u$ and by ${\mathrm d}(N(u), {\tt x})$ the distance of ${\tt x}$ to $N(u)$ then
\begin{align}
\label{exp est 1}
|u({\tt x})^2-1|+|\nabla u({\tt x})|+|D^2 u({\tt x})|\leq C e^{\,-\beta{\mathrm d}(N(u), {\tt x})},
\end{align}
where $\beta>0$. This type of estimate is relatively easy to obtain using a comparison argument. On the other hand, the estimate (\ref{eq:refined-2}) is non trivial.

\subsection{The balancing formulas}
We will now describe briefly the balancing formulas for $4$ ended solutions in the form they were introduced in \cite{dkp-2009}. Assume that $u$ is a solution of (\ref{AC}) which is defined in $\mathbb R^2$. Assume that $X$ and $Y$  are two vector fields also defined in $\mathbb R^2$.  In coordinates, we can write
\[
X = \sum_j X^j \partial_{x_j},  \qquad \qquad Y =  \sum_j Y^j \partial_{x_j},
\]
and, if $f$ is a smooth function, we use the following notations
\[
X(f) : = \sum_j X^j \, \partial_{x_j} f,   \qquad \qquad  \nabla f : = \sum_j \partial_{x_j} f  \, \partial_{x_j} ,
\]
\[
\mbox{\rm div} \, X : =  \sum_i \partial_{x_i} X^i ,
\]
and
\[
{\rm d}^* X : = \frac{1}{2} \sum_{i,j} ( \partial_{x_i} X^j  +  \partial_{x_j} X^i  ) \,   {\rm d}x_i  \otimes   {\rm d}x_j  ,
\]
so that
\[
{\rm d}^*  X  \,  (Y,Y) =  \sum_{i,j}  \partial_{x_i} X^j  \, Y^i \, Y^j  .
\]

We claim that~:
\begin{lemma} [Balancing formula]
The following identity holds
\[
\begin{array}{rlllll}
\displaystyle \mbox{\rm div}  \left(  \left( \frac12  | \nabla u |^2  +  F(u) \right)  X  -  X(u)  \nabla u  \right) = \displaystyle  \left( \frac12  | \nabla u|^2  + F(u) \right)  \mbox{\rm div} \, X -  {\rm d}^* X (\nabla u ,\nabla u) .
\end{array}
\]
\label{le:PI}
\end{lemma}
\begin{proof} This follows from direct computation.\end{proof}

Translations of $\mathbb R^2$ correspond to the constant  vector field
\[
X : = X_0
\]
where $X_0$ is a fixed vector,  while rotations correspond to the vector field
\[
X : = x \, \partial_y - y \, \partial_x .
\]
In either case, we have $\mbox{\rm div} \, X =0$ and ${\rm d}^* \, X =0$. Therefore, we  conclude that
\[
\mbox{\rm div} \, \left(  \left( \frac12 \, | \nabla u |^2  +  F(u) \right)  \, X  -  X(u) \, \nabla u  \right)  = 0  ,
\]
for these two vector fields. The divergence theorem implies that
\begin{equation}
\int_{\partial \Omega}  \left( \left( \frac12 \, | \nabla u |^2  +  F(u) \right)  \, X  -  X(u) \, \nabla u \right) \cdot \nu  \,  ds = 0  ,
\label{eq:lsr}
\end{equation}
where $\nu$ is the (outward pointing) unit normal vector field to $\partial \Omega$.

To see how this identity is applied let us fix a unit vector ${\tt e}\in \R^2$ and let $X={\tt e}$. For any $s\in \R$ we consider a straight line $L_s=\{{\tt x}\in \R^2\mid {\tt x}=s {\tt e}+t{\tt e}^\perp, t\in \R\}$. Then we get:
\begin{align}
\label{hamilt 1}
\int_{L_s} \big[\frac{1}{2}|\nabla u|^2-|\nabla u\cdot{\tt e}|^2+F(u)\big]\,dS=const.,
\end{align}
for any $4$ end solution $u$ of (\ref{AC}), as long as the direction of $L_s$ does not coincide with that of any end, i.e. ${\tt e}\neq {\tt e_j}$, $j=1, \dots, 4$. In a particular case ${\tt e}=(0,1)$ we get a {\it Hamiltonian  identity} \cite{MR2381198}:
\begin{align}
\label{hamilt 2}
\int_{y=s} \big[\frac{1}{2}(\partial_x u)^2-\frac{1}{2}(\partial_y u)^2 + F(u)\big]\,dx=const.  . 
\end{align}

\subsection{Summary of the existence result for small angles in \cite{MR2557944}}\label{exists small eps}

To state the existence  result in precise way, we assume that we are given an even symmetric solution of the {Toda system} (\ref{toda 02}) represented by a pair of functions $q_1(t)<0<q_2(t)$, where $q_1(t)=-q_2(t)$ as well as $q_1(t)=q_1(-t)$. In addition let us assume that the slope of $q_1$ at $\infty$ is $-1$. Then,  asymptotically we have:
\[
q_{j}(x)=(-1)^j (|x|+b)+{\mathcal O}(e^{\, -2\sqrt{F''(1)}|x|}), \quad x\to \infty,
\]
Given $\varepsilon>0$, we define the the vector valued function
${\bf q}_{\varepsilon}$, whose components are given by
\begin{equation}
q_{j, \varepsilon} (x)  : =  q_{j}(\varepsilon \, x )  
+\frac{(-1)^j}{\sqrt {F''(1)}}  \log \frac{1}{\varepsilon}   \, .
\label {falpha}
\end{equation}
It is easy to check that the $q_{j,\varepsilon}$ are again solutions of (\ref{toda 02}).

\medskip

Observe that, according to the description of asymptotics the functions $q_j$, the 
graphs of the functions $q_{j,\varepsilon}$ are asymptotic to oriented half lines with slope $\ve$ at  
infinity. In addition, for $\varepsilon >0$ small enough, these graphs are disjoint and 
in fact their mutual distance is given by $ \frac{2}{\sqrt{F''(1)}} \,  \log\frac{1}{\varepsilon} + \mathcal O (1)$
as $\varepsilon$ tends to $0$. 

\medskip

It will be convenient to agree that $\chi^+$ (resp. $\chi^-$) is a smooth 
cutoff function defined on $\mathbb R$ which is identically equal to $1$ 
for $x >1$ (resp. for $x < -1$) and identically equal to $0$ for $x < -1$  
(resp. for $x > 1$) and additionally $\chi^-+\chi^+\equiv 1$.  With these 
cutoff functions at hand, we define the $4$ dimensional space
\begin{equation}
D : =  {\rm Span} \, \{x \longmapsto  \chi^\pm (x) , \,  x\longmapsto  x \, \chi^\pm  (x) \} \, ,
\label{D}
\end{equation}
and, for all $\mu \in (0,1)$ and all $\tau \in \mathbb R$,  we define 
the space  ${\mathcal C}^{2, \mu}_\tau (\mathbb R)$ of $\mathcal C^{2 , \mu}$ 
functions $r$ which satisfy 
\[
\| r\|_{\mathcal C^{\ell , \mu}_\tau (\mathbb R)}  : =   \|Ê (\cosh x)^{\tau} \,  r \|_{\mathcal C^{\ell , \mu} (\mathbb R )}  \,   < \infty \, .
\]
Keeping in mind the above notations, we have the~:
\begin{theorem}\label{teo1}
For all  $\varepsilon>0$ sufficiently small, there exists an entire solution 
$u_\varepsilon$ of the Allen-Cahn equation $\equ{AC}$ whose nodal set 
is the union of $2$ disjoint curves $\wt \Gamma_{1, \varepsilon }, \wt
\Gamma_{2, \varepsilon }$ which  are the graphs of the functions
\[
x \longmapsto   q_{j,\varepsilon } (x)  + r_{j, \varepsilon} (\varepsilon \, x)  \, ,
\]
for some functions $r_{j, \varepsilon } \in \mathcal C^{2, \mu}_{\tau}
 (\mathbb R) \oplus D$ satisfying
\[
\| r_{j, \varepsilon} \|_{\mathcal C^{2, \mu}_{\tau } (\mathbb R) \oplus D} 
 \leq C \,  \varepsilon^{\alpha} \, . 
\] 
for some constants $C , \alpha , \tau , \mu >0$ independent of $\varepsilon >0$.
\end{theorem}

In other words, given a solution of the Toda system, we can find  
a one parameter family of $4$-ended solutions of \equ{AC}  which 
depend on a small parameter $\varepsilon > 0$. As $\varepsilon$ tends to $0$, the 
nodal sets of the solutions we construct become close to the graphs of the functions 
$q_{j ,\varepsilon }$.

\medskip

Going through the proof, one can be more precise about the description of the solution 
$u_\varepsilon$. If $\Gamma \subset \mathbb R^2$ is a curve in $\mathbb R^2$ which is the graph 
over the $x$-axis of some function, we denote by
$\mbox{Y} \, ( \cdot,  \Gamma ) $ the signed distance to $\Gamma$ 
which is positive  in the upper half of $\mathbb R^2 \setminus \Gamma$ 
and is negative in the lower half of $\mathbb R^2 \setminus \Gamma$. Then we have the~:
\begin{proposition}
The solution of \equ{AC} provided by Theorem \ref{teo1} satisfies
\[
\|  e^{ {\varepsilon} \, \hat \alpha \, |{\bf x}|} \,  (u_\varepsilon - u_\varepsilon^*)  \|_{L^\infty (\mathbb R^2)} 
\leq  C \,  \varepsilon^{\bar \alpha}Ê\, ,
\]
for some constants $C, \bar \alpha , \hat \alpha > 0$ independent of $\varepsilon$, where ${\bf  x}=(x,y)$ and 
\begin{equation}
u_\varepsilon^*  : = \sum_{\red{j=1}}^k (-1)^{j+1} \, H\big( {\rm Y} ( \cdot, 
\wt\Gamma_{j , \varepsilon}) \big) - \frac 12 ( (-1)^{k} +1 ) \,.
\label{eq:bua}
\end{equation}
\end{proposition}

\section{The nodal sets of solutions}\label{nodal asymp}

We know \cite{partI}  that the angle map on any connected component of the moduli space $\mathcal{M}_4$
of four end solutions is surjective, and that in particular it contains solutions whose nodal lines are almost parallel ($\theta(u)\approx 0$ or $\frac{\pi}{2}-\theta(u)\approx 0$). We recall also that, after a rigid motion,  any four ended solution is even symmetric \cite{2011arXiv1102.4022G} and thus we will always consider solutions in ${\mathcal M}^{even}_4$, which in particular satisfy (\ref{even}). To any solution $u\in {\mathcal M}^{even}_4$ we can associate the value of the angle map $\theta(u)$, which by definition is the asymptotic angle at $\infty$ between the nodal set of $u$ in the first quadrant  and the $x$-axis. Finally, by $N(u)$ we will denote in this paper the nodal set of $u\in{\mathcal M}^{even}_4$. Because of (\ref{even}), restricted to each quadrant, this set is a bigraph, and restricted to a half plane it is a graph. We are interested in solutions whose nodal lines are almost parallel at $\infty$ and, by symmetry, we can restrict our considerations to the case $\theta(u)\to 0$. In this case the nodal set consists of two components, one of  them is a graph of a smooth function in the lower  half plane and the other, which  is  symmetric with respect to the $x$-axis, is contained in the upper half plane. It is convenient to introduce a parameter $\ve=\tan\theta(u)$, which is small $\ve\to 0$, when $\theta(u)\to 0$.

\subsection{Basic properties of solutions with almost parallel ends}

It is expected that when the angle between the ends becomes smaller, i.e. $\theta(u)\to 0$ or 
$\theta(u)\to \frac{\pi}{2}$,  then the distance between  the nodal sets  becomes larger.   This is indeed the case and
could be seen from  lemma \ref{lema unique 1} to follow. With no loss of generality  we will restrict our considerations to the case $\theta(u)\to 0$. In the sequel by $Q_1$ we will denote the first quadrant in $\R^2$. 

\begin{lemma}\label{lema unique 1}
Suppose $u_{n}$ is a sequence of four end solutions such that $\theta({u_{n})%
}\rightarrow0,$ $p_{n}\in N\left(  u_{n}\right)  \cap\partial Q_{1}.$ Then
$\left\vert p_{n}\right\vert \rightarrow+\infty,$ as $n\rightarrow+\infty.$ Moreover $p_n\in \{x=0\}$.
\end{lemma}

\begin{proof}

To show that $p_n\to \infty$ we suppose to the contrary that $p_n\to p^*$, $|p^*|<\infty$ . We know that the sequence $\{u_n\}$ converges in ${\mathcal C}_{loc}^2(\R^2)$ to a solution $u^*$ of the Allen-Cahn equation. Since $|p^*|<\infty$, by Theorem 4.4 of  \cite{MR2381198} we know that if $u^*_x>0$, $x>0$, $u^*_y<0$, $y>0$, then $u^*$ must be a  solution to (\ref{AC})  whose nodal set in the first quadrant is asymptotically a straight line with positive slope equal to $\tan \theta^*\neq 0$. To show that $u^*_x>0$, $x>0$ and $u^*_y<0$, $y>0$, one may use  an maximum principle based  argument, similar to the one in the Claim 1 below, we leave the details to the reader.   Although a priori $u^*$  is not a four ended solution in the sense of Definition \ref{de:001}, it can also be proven, using the De Giorgi conjecture in dimension two and the Refined Asymptotic theorem (Proposition \ref{refined asymp}), that $u^*\in\mathcal{M}_4^{even}$ (we will not rely on this fact in the argument given below). 
For future purpose we recall  that the nodal set in the first quadrant is a graph i.e.  $N(u_n)\cap Q_1=\{(x,y)\mid y=f_n(x)\}$,  where $f_n$ is a smooth function.

To proceed we need the following:

\noindent
{\bf{Claim 1}.}
{\it Let $\{u_n\}$ be a sequence of solutions as described above and let $x_n\to \infty$, as $n\to \infty$. There exists a subsequence $\{x_{n_k}\}$ such that $f_{n_k}(x_{n_k})\to \infty$, as $k\to \infty$.  }
 
To prove this claim we argue by contradiction. We assume that for some constant $A_0>0$, we have  that 
\[
f_n(x_n)\leq A_0, \quad \forall n.
\]
Since $f_n'(x)=-\frac{u_{n,x}(x, f_n(x))}{u_{n,y}(x,f_n(x))}>0$ therefore, passing to the limit as $n\to \infty$, we have $u_n\to u^*$, at least along a subsequence in $\mathcal C^2_{loc}(\R^2)$,  where $u^*$ is  an even  solution  such that $u^*_x\geq 0$, in $x>0$, $u^*_y\leq 0$, in $y>0$. We will show that in fact $u^*_x>0$ when $x>0$ and  $u^*_y<0$, when $y>0$. If $u_x^*(x^*, y^*)=0$, for some $(x^*, y^*)$ then, using the maximum principle, we get $u^*_x\equiv 0$. The same is true for $u^*_y$. Consequently either $u^*$ is an even, monotone, one dimensional solution, which is impossible, or $u^*=u_0$, where $u_0=\pm 1$ or $0$. In the latter case we consider the quantity 
\[
E_R(u)= \frac{1}{R}\int_{B_R}\frac{1}{2}|\nabla u|^2+F(u), \quad R>0.
\]
This quantity is known to be monotone in $R$ and, since $|p^*|<\infty$ we   have that 
\[
c\leq E_R(u_n)\leq C,
\]
for certain constants $c, C>0$. But since $E_R(u_n)\to E_R(u^*)$ we see that $u^*$ can not be a constant. 

In summary $u^*$ is an even and monotone solution in $\R^2$, whose $0$ level set $N(u^*)$ is contained in the strip $|y|\leq A_0$. But by Theorem 4.4 of  \cite{MR2381198} this is impossible. This ends the proof of the claim.

Now we have the following:

\noindent
{\bf{Claim 2}.}  {\it For each sufficiently small $\delta$ there exists a $r>0$ such that for all $r<x_{1}<x_2$ we have $|f'_n(x_1)-f'_n(x_2)|<\delta$.} 

Accepting this  we get that $\lim_{n\to \infty}\tan \theta(u_n)=\tan\theta^*\neq 0$, which is a contradiction. It remains then to prove the claim. 

Arguing by contradiction we suppose that there exist sequence $r_{k}%
\rightarrow+\infty$ and $x_{1,n_{k}},x_{2,n_{k}}>r_{k},$ such that%
\begin{equation}
\label{contra 1}
\left\vert f_{n_{k}}^{\prime}\left(  x_{1,n_{k}}\right)
-f_{n_{k}}^{\prime}\left(  x_{2,n_{k}}\right)  \right\vert \geq {\delta}.
\end{equation}
Then, at least for one of the points $x_{i, n_k}$ we have (passing to a subsequence if necessary):
\begin{equation}
\label{lem one 1}
|f'(x_{i, n_k})|>\frac{\delta}{2}.
\end{equation}
We will call this point $x^*_{1, n_k}$. Next we chose $x^{*}_{2, n_k}$ such that $|f'_{n_k}(x^{*}_{2, n_k})|=\frac{\delta}{4}$ and
\[
\frac{\delta}{4}\leq |f'_{n_k}(x)|, \quad x\in [x^*_{1, n_k}, x^{*}_{2,n_k}].
\] 
Such point clearly exists when $k$ is sufficiently large since   $f'_{n_k}(x)\to \tan(\theta(u_{n_k}))$, as $x\to \infty$. 


Consider two lines $L_{1,n_{k}}$ and $L_{2,n_{k}}$ with slopes $-1$ passing
though the points $(x^*_{i,n_{k}},f_{n_{k}}(x^*_{i,n_{k}}))$, $i=1,2$,
respectively. Note that since the nodal lines $N(u_{n_{k}})$ are bigraphs,
which are eventually asymptotically straight lines with positive slopes, therefore the
lines $L_{i,n_{k}}$ must be transversal to $N(u_{n_{k}})$ at their points of
intersection. Next, consider the domain $\Omega_{n_{k}}\subset Q_{1}$ bounded
by the lines $L_{i,n_{k}}$, $i=1,2$,
and the vector field $X=\left(  0,1\right)  .$ The balancing formula
(\ref{eq:lsr}) tells us
\[
\int_{\partial\Omega_{n_{k}}}\left(  \left(  \frac{1}{2}\left\vert \nabla
u_{n_{k}}\right\vert ^{2}+F\left(  u_{n_{k}}\right)  \right)  X-X\left(
u_{n_{k}}\right)  \nabla u_{n_{k}}\right)  \cdot\nu dS=0.
\]
Now we estimate the various boundary integrals involved in the above integral. First, note that the integral over the segment $\partial\Omega_{n_k}\cap\{x=0\}$ is automatically $0$  by the choice of the vector field $X$ and the evenness of $u_{n_k}$. 

Next, we will show that as $r_{k}\rightarrow+\infty,$
\begin{equation}
\int_{\partial\Omega_{n_{k}}\cap\left\{  y=0\right\}  }\left(  \left(
\frac{1}{2}\left\vert \nabla u_{n_{k}}\right\vert ^{2}+F\left(  u_{n_{k}%
}\right)  \right)  X-X\left(  u_{n_{k}}\right)  \nabla u_{n_{k}}\right)
\cdot\nu dS\rightarrow0. \label{small}%
\end{equation}
To this end let $\bar x_{i, n_k}$, $i=1,2$ be two points such that $\partial \Omega_{n_k}\cap\{y=0\}=\{(x, 0)\mid x\in [\bar x_{1, n_k},  \bar x_{2, n_k}] \}$. Then, by elementary geometry we get,  with some constant $c>0$: 
\[
\mathrm{dist}\,(N(u_{n_k})\cap Q_1, {\tt x})\geq \min\left\{(x-x^*_{1,n_k}), c \left(\frac{\delta}{4}(x-x^*_{1,n_k})+f(x^*_{1,n_k})\right)\right\}, \quad {\tt x}=(x, 0), \quad x \in [\bar x_{1,n_k}, x^*_{2, n_k}].
\]
On the other hand, when $x\in [x^*_{2,n_k}, \bar x_{2, n_k}]$ then 
\begin{align*}
\mathrm{dist}\,(N(u_{n_k})\cap Q_1, {\tt x})&\geq c\left( \frac{\delta}{4}(x^*_{2,n_k}-x^*_{1,n_k})+f_{n_k}(x^*_{1,n_k})\right), \\
 ( \bar x_{2, n_k}-x^*_{2,n_k})&\leq C \left(\frac{\delta}{4}(x^*_{2,n_k}-x^*_{1,n_k})+f_{n_k}(x^*_{1,n_k})\right).
\end{align*}
Additionally we know that 
\begin{equation}
\left\vert u_{n_{k}}^{2}\left({\tt x}\right)  -1\right\vert +\left\vert \nabla
u_{n_{k}}\left( {\tt x}\right)  \right\vert \leq Ce^{-\beta\mathrm{dist}\,(N(u_{n_k})\cap Q_1, {\tt x})}, \quad{\tt x}=(x,0)  \in\partial\Omega_{n_{k}}.
\label{de}%
\end{equation}
Above, the constants $c>0$ and $C>0$ can be chosen independent on $n$. Using this, the Claim 1,  and the fact that $\bar x_{1, n_k}-x^*_{1,n_k}$  is proportional to $f_{n_k}(x^*_{1,n_k})$, we conclude (\ref{small}). 


Now we will estimate the integrals along the lines $\partial\Omega_{n_k}\cap L_{i, n_k}$. For this purpose it is convenient to denote 
\[
{\tt e}_{i,n_k}=(\cos\alpha_{i,n_k}, \sin\alpha_{i,n_k}), \quad {\tt e}^\perp_{i, n_k}=(\sin\alpha_{i,n_k}, -\cos\alpha_{i, n_k}), \quad \mbox{where}\ \alpha_{i,n_{k}}=\arctan f_{n_{k}}^{\prime}\left(  x^*_{i,n_{k}}\right),
i=1,2.
\]
Let us consider the sequence of functions $v_{i, k}(x,y)=u(x^*_{i, n_k}+x, f_{n_k}(x^*_{i, n_k})+y)$. By the De Giorgi conjecture in dimension two we have:
\[
v_k(x,y)-H({\tt e}^\perp_{i, n_k}\cdot (x,y))\to 0, \quad \mbox{in}\ {\mathcal C}^2_{loc}(\R^2).
\]
In other words, locally around $\left(  x^*_{i,n_{k}},f_{n_{k}%
}\left(  x^*_{i,n_{k}}\right)  \right)  $, the function  $u_{n_k}$ converges to 
\[
H\left({\tt e}^\perp_{i, n_k}\cdot(x-x^*_{i,n_{k}}, y-f_{n_{k}}(  x^*_{i,n_{k}})\right).
\]
On the other hand, again by (\ref{exp est 1}), we know that on the segment
$\partial\Omega_{n_{k}}\cap L_{i,n_{k}}$,
\[
|u_{n_{k}}^{2}\left({\tt x}\right)  -1|+|\nabla u_{n_{k}}\left({\tt x}\right)
|\leq Ce^{-\beta|x^*_{i,n_{k}}-x|}, \quad {\tt x}=(x,y).
\]
These facts, after some calculations,   yield
\[
\int_{\partial\Omega_{n_{k}}\cap L_{i,n_{k}}}\left(  \left(  \frac{1}%
{2}\left\vert \nabla u_{n_{k}}\right\vert ^{2}+F\left(  u_{n_{k}}\right)
\right)  X-X\left(  u_{n_{k}}\right)  \nabla u_{n_{k}}\right)  \cdot\nu
dS=\left(  -1\right)  ^{i+1}\sin\alpha_{i,n_{k}}\mathtt{e}_{F}+o\left(
1\right)  ,
\]
where $o\left(  1\right)  $ is a term goes to $0$ as $r_{k}\rightarrow
+\infty,$ while
\[
{\tt e}_{F}=\int_\R (H')^2.
\]
Combining all the above estimates, we infer
\[
\sin\alpha_{1,n_{k}}-\sin\alpha_{2,n_{k}}=o\left(  1\right)  ,
\]
which is a contradiction.

It remains to show \red{that $p_n\in \{x=0\}$. To this en we argue by contradiction and  assume that $p_n\in \{y=0\}$. Then we use a Hamiltonian identity:
\[
\int_{\{x=0\}} \big[\frac{1}{2} u_{n,y}^2-\frac{1}{2} u_{n,x}^2+F(u_n)\big]\, dy=\int_{\{x=s\}} \big[\frac{1}{2} u_{n,y}^2-\frac{1}{2} u_{n,x}^2+F(u)\big]\, dy.
\]
On the one hand, letting $s\to \infty$ we get the right hand side member converges to $2{\tt e}_{F}$, while the left hand side member converges to $0$ as $n\to \infty$. This leads to a contradiction. 

}

\end{proof}

Now, $N(u)\cap Q_1$ is a graph of a monotonically increasing, even function hence  the above proposition asserts that  as $\theta(u)\to 0$, the distance between the nodal set of $u$ and the $x$ axis will go the infinity as well.

\subsection{A refinement of the asymptotic behavior of the nodal set}
Let   $u$ be a four-end solution with small angle $\theta({u}).$
We denote $\varepsilon=\tan\theta(u)$ and use for simplicity $\ve$ as a small parameter. To obtain more precise
information about this solution, our first step is to define a good approximate
solution. As we will see later, the fact that there is a true solution around
the approximate one restricts the possible behavior of the nodal set and
enables us the analyze the equation satisfied by the nodal line.

The nodal set $N\left(  u\right)  $ in the lower half plane is the graph
$\Gamma$ of a function $y=f\left(  x\right)  $, $y\in \R$. Strictly speaking the function $f$ depends on $u$ but we will not indicate this dependence.  Using the De Giorgi conjecture in $2$ dimensions   it is not difficult
to show that $\left\Vert f^{\prime}\left(  x\right)  \right\Vert _{{\mathcal C}^{1}(\R)}\to 0$ as $\theta({u})\to 0$. Indeed, let us assume that there exist a sequence of solutions $u_n$, with $\theta(u_n)\to 0$, and a sequence of points ${\tt x}_n\in N(u_n)\cap Q_1$, such that 
\begin{equation}
\label{f prime 2}
{\tt x}_n=(x_n, f(x_n)),\quad  \mbox{and}\  |f'(x_n)|+|f''(x_n)|>\delta, \quad \mbox{with some} \ \delta>0. 
\end{equation}
Then we define:
\[
v_n({\tt x})=u_n({\tt x}_n+{\tt x}).
\]
Using the monotonicity of $u_n$ we get that $v_{n,y}>0$, for $y<-f(x)$ and then the De Giorgi conjecture implies that $v_n$ converges in ${\mathcal C}^2_{loc}(\R^2)$ to the heteroclinic solution which contradicts (\ref{f prime 2}). 
For future references we observe finally  that in general the graph of $N(u)\cap Q_1$ is at least a ${\mathcal C}^3(\R)$ function and, \red{bootstrapping the above argument it is not hard to show that $\|f'\|_{\mathcal{C}^2(\R)}=o(1)$, as $\theta(u)\to 0$ as well. In fact, we can use the fact that $u(x,f(x))=0$ and that $-u_y(x, f(x))>c>0$ in $Q_1$, to show this. In turn the uniform estimate on $-u_y$ along the nodal set  can be proven using the De Giorgi conjecture. In what follows we will often rely on the fact   $f\in \mathcal{C}^3(\R)$ and that $\|f'\|_{\mathcal{C}^2(\R)}\to 0$  as $\theta(u)\to 0$.  }

To fix attention we will always work with the solution whose nodal lines have a small slope $\ve=\tan\theta(u)$ at $\infty$. This means that the these lines are asymptotically parallel, as $\ve\to 0$,  to the $x$ axis and one of them is contained in the lower half plane and the other in the upper half plane. We know that they are symmetric with respect to the $x$ axis. 
In the sequel it will be convenient to denote the   component of the nodal set $N(u)$ in the lower half plane by $\Gamma_{\ve,1}$, and the  one in the upper  half plane by $\Gamma_{\ve,2}$. The nodal lines are bigraphs, and consequently we  have $\Gamma_{\ve,i}=\{y=f_{\ve,i}(x)\}$.

To introduce the functional analytic tools used in this paper we introduce a weight function%
\[
W_{a}({\tt x})  =\left(  \cosh x\right)  ^{a}, \quad {\tt x}=(x,y), \quad a>0.
\]
For $\ell =0,1,2,$ let $\mathcal{C}_{a}^{\ell,\mu}(  {\R}^{2})  =W_{a}%
^{-1}{\mathcal C}^{\ell,\mu}(  {\R}^{2}) $, endowed with the weighted norm%
\[
\| \phi\| _{{\mathcal C}_{a}^{\ell,\mu}\left(  {\R}^{2}\right)
}:=\sup_{{\tt x}\in \R^2}W_{a}({\tt x})\|\phi\|_{{\mathcal C}^{\ell,\mu}(B({\tt x}, 1))  }.
\]
Likewise, we let $W_a(x)=(\cosh x)^a$ and define the weighted space $\mathcal{C}^{\ell, \mu}_a(\R)$ by:
\[
\|f\|_{{\mathcal C}_{a}^{\ell,\mu}\left(  {\R}\right)
}:=\sup_{{x}\in \R}W_{a}({x})\|f\|_{{\mathcal C}^{\ell,\mu}(({x}-1, x+1))  }.
\]
In what follows we will measure the size of various  functions involved  in the $\mathcal{C}_{a}^{2,\mu}(  {\R}^{2})$, and  in the  ${{\mathcal C}_{a}^{2,\mu}\left(  {\R}\right)
}$ norm, where $a, \mu>0$. Mostly  we will have  $a\sim \ve$.

Let us recall that a four end solution $u$ is asymptotic to a model solution $u_{\lambda}$ defined in the introduction.   Using the  Proposition \ref{refined asymp} we know that $u-u_\lambda\in W^{2,2}_{\ve\bar\tau, \delta}(\R^2)$ with some $\bar \tau >0$ and $\delta>0$, which can be chosen independent on $\ve$. We claim that from this it follows  that in fact $u-u_\lambda\in \mathcal C_{\ve\tau}^{2,\mu}(\R^2)$, with some $\tau>0$. 
To see this we let  $\ve$  be fixed and ${\tt e}$ be the asymptotic direction of the end of $u$ with $\ve=\tan\theta(u)$ in $Q_1$. By  definition,   taking  $R$ large,    we have that 
\[
\Gamma_{\gamma, \delta}({\tt x})\sim
\left(\cosh (({\tt x}-{\tt x}_{\ve, 1})\cdot {\tt e})\right)^\gamma\left(\cosh(({\tt x}-{\tt x}_{\ve, 1})\cdot {\tt e}^\perp)\right)^\delta\sim (\cosh x)^{\gamma +\delta{\mathcal O}(\ve)}(\cosh y)^{\delta+\gamma\mathcal O(\ve)}, \quad {\tt x}\in Q_1\setminus B_R.
\]
Taking $\gamma=\bar\tau\ve$ and $\delta$ small we get the claim. In addition, we know 
 that outside of a large compact set in the first quadrant we have:
\[
u({\tt x})=H(({\tt x}-{\tt x}_{\ve, 1})\cdot {\tt e}^\perp)+ {\mathcal O}(e^{\,-\ve\bar\tau |x|-\delta|y|}).
\]
We can use this and the fact that  along the nodal set we have $u({\tt x})=0$ to show that $f''_{\ve, 2}\in \mathcal{C}^{0,\mu}_{\ve\tau}(\R)$. Indeed, we get from the above that, with some constant ${\mathcal A}_\ve$ we have:
\begin{equation}
\label{veeee}
0=H(f_{\ve, 2}(x)-\ve x-{\mathcal A}_\ve)+ {\mathcal O}(e^{\,-\tau\ve x}), \quad x\to \infty,
\end{equation}
from which it follows that 
\begin{equation}
\label{first f}
\|f_{\ve, 2}-\ve |x|-{\mathcal A}_\ve\|_{\mathcal C^{0,\mu}_{\ve\tau}(\R)}+\|f'_{\ve, 2}-\ve\mathrm{sign}\,(x)\|_{\mathcal C^{0,\mu}_{\ve\tau}(\R)}+\|f''_{\ve, 2}\|_{\mathcal C^{0,\mu}_{\ve\tau}(\R)}<\infty.
\end{equation}
%

\subsection{Fermi coordinates near  the nodal lines}

We will now describe neighborhoods of the nodal lines $\Gamma_{\ve, i}$, $i=1,2$, where one can define the Fermi coordinates  of ${\tt x}\in \R^2$ as the unique $(x_i, y_i)$ such that:
\[
{\tt x}=(x_i, f_{\ve, i}(x_i))+y_i  n_{\ve, i}(x_i), \quad n_{\ve, i}(x)=\frac{(-1)^i(f_{\ve, i}'(x), -1)}{\sqrt{1+\left(f_{\ve, i}'(x)\right)^2}}.
\]
We will first find a large, expanding  neighborhood of $\Gamma_{\ve, i}$ in  which that map ${\tt x}\mapsto (x_i, y_i)$ is a diffeomorphism. Because of symmetry it suffices to consider a neighborhood of $\Gamma_{\ve, 1}$. 
We define the projection of a point ${\tt x}\in \R^2$ onto $\Gamma_{\ve, 1}$ by:
\[
\pi_{\ve,1}({\tt x})=(\pi_{\ve, 1}^x({\tt x}), \pi_{\ve, 1}^y({\tt x)}):=(x_1, f_{\ve, 1}(x_1)), \quad \mbox{whenever}\  {\tt x}=(x_1, f_{\ve, 1}(x_1))+y_1n_{\ve, 1}(x_1), \quad y_1=\mathrm{dist}\,({\tt x}, \Gamma_{\ve, 1}).
\]  
Note that in general the projection $\pi_{\ve,1}({\tt x})$ is a multivalued function.

We fix a small number $\theta\in (0,1)$ and   let  ${m}_\ve\colon \R\to \R_+$ to be a  function such that for each $x\in \R$, $m_\ve(x)$ represents the largest number for   which the  following properties are satisfied simultaneously:
\begin{itemize}
\item[(i)]
\[
m_\ve (x) \leq\frac{\theta}{\left\vert f_{\ve, 1}^{\prime\prime}\left(
x\right)  \right\vert },
\]
\item[(ii)]
The projection function $\pi_{\ve, 1}$ is  a well defined single-valued function for any ${\tt x}\in \R^2$ such that
\[
\|{\tt x }-\pi_{\ve, 1}({\tt x})\| <(m_\ve\circ\pi_{\ve, 1}^x)({\tt x}).
\]
\end{itemize}
We can regard the function ${m}_\ve$ as the measure of the  size
of the maximal  neighborhood of $\Gamma_{\ve, 1}$, where the Fermi coordinate could be defined. In fact, conditions (i)--(ii) guarantee that the change of variables given by the Fermi coordinates is a diffeomorphism in a neighborhood of $\Gamma_{\ve, 1}$ determined by (ii). 

To state the next result we let $\tau<\alpha_0=\sqrt{F''(1)}$ be a positive constant.

\begin{lemma}\label{lemma 3.3}
For each $A>0$, and for each sufficiently small $\ve$ we have the following estimate:
\red{\begin{equation}
\label{est dve 1}
e^{\,-m_\ve\left(  x\right)  }(\cosh x)^{A \varepsilon}\leq C_{A, \tau, \theta}\|f_{\ve, 1}''\|^{\frac{A}{\tau}}
_{\mathcal {C}_{\varepsilon\tau}^{0}\left(  \mathbb{R}\right)  }\|f_{\ve, 1}''\|^{\frac{A}{\tau}}
_{\mathcal {C}^{0}\left(  \mathbb{R}\right)  }.
\end{equation}}
\end{lemma}

\begin{proof}
If $m_\ve\left(  x\right)  =\frac{\theta}{{\left\vert f_{\ve, 1}^{\prime\prime}\left(
x\right)  \right\vert }},$ then 
\[
e^{\,-m_\ve\left(  x\right)  }\leq \exp\{-\frac{\theta}{|f_{\ve, 1}''(x)|}\}
\leq 
\exp\{-\frac{\theta(\cosh x)^{\tau\ve}}{2\|f''_{\ve, 1}\|_{{\mathcal C}^0_{\ve\tau}(\R)}}\}\exp\{-\frac{\theta}{2\|f''_{\ve, 1}\|_{{\mathcal C}^0(\R)}}\}
%
\]
The desired estimate follows from this.

If for some $x_0\in \R$ we have $m_\ve\left(  x_0\right)  <\frac{\theta}{{\left\vert f_{\ve, 1}^{\prime\prime}\left(
x_0\right)  \right\vert }},$ then there is an ${\tt x}_0$, with $\pi_{\ve,1}({\tt x}_0)=(x_0, y_0)$, and ${\tt x}_{1}=\left(  x_{1},f_{\ve, 1}(x_1)\right)  $ and
${\tt x}_{2}=\left(  x_{2},f_{\ve, 1}(x_2)\right),$ on $\Gamma_{\ve, 1}$, with $x_{1}<x_{2}$, and such that, denoting $r:=m_\ve(x_0)$ we would have
\[
\|{\tt x}_0-{\tt  x}_1\|=\|{\tt x}_0-{\tt x}_2\|=r\leq \frac{\theta}{{|f''_{\ve, 1}(x_0)|}}, 
\]
i.e. ${\tt x}_{j}$, $j=1,2$ lie on the circle $S_r$ with center at ${\tt x}_0$. For convenience we fix an orientation of $S_r$  which agrees with the orientation of the segment of $\Gamma_{\ve,1}$ between the points ${\tt x}_1$ and ${\tt x}_2$. 

We claim, that the arc of $S_r$, between ${\tt x}_{1}$ and ${\tt x}_{2}$  is the graph of a function $y=g\left(  x\right)
,x\in\left[  x_{1},x_{2}\right]$.  Indeed, we recall  that $\Gamma_{\ve, 1}$ is a bigraph, and so a graph,  and the points ${\tt x}_i$ lie on $\Gamma_{\ve, 1}\cap S_r$.  If there were a vertical line $L$ intersecting the arc of $S_r$ between ${\tt x}_1$ and ${\tt x}_2$ more then once then, since the segment of $\Gamma_{\ve, 1}$ between these points lies outside of $S_r$, except at ${\tt x}_1$ and ${\tt x}_2$,  we would conclude that necessarily $L$ intersects $\Gamma_{\ve, 1}$ more then once. This is a contradiction.  

An elementary
calculation yields%
\[
\min_{x\in\left[  x_{1},x_{2}\right]  }\left\vert g^{\prime\prime}\left(
x\right)  \right\vert \geq  \frac{1}{r}
\]
On the other hand, 
\[
\left\vert g^{\prime}\left(  x_{2}\right)  -g^{\prime}\left(  x_{1}\right)
\right\vert =\left\vert f_{\ve,1}^{\prime}\left(  x_{2}\right)  -f^{\prime}_{\ve,1}\left(
x_{1}\right)  \right\vert .
\]
Therefore, one can find a point ${\tt x}_{3}=\left( x _{3},y_{3}\right)  \in\Gamma_{\ve, 1}$
which satisfies
\[
\left\vert f^{\prime\prime}_{\ve,1}\left(  x_{3}\right)  \right\vert \geq\min
_{x\in\left[  x_{1},x_{2}\right]  }\left\vert g^{\prime\prime}\left(
x\right)  \right\vert \geq \frac{1}{r}
%
\]
%
%
%
%
%
Summarizing,  we see that there exists $x_{3}%
\in\left(  x_0-2m_\ve\left(  x_0\right)  ,x_0+2m_\ve\left(  x_0\right)  \right)  \,\ $such that%
\begin{equation}
\label{meps xxx}
m_\ve\left(  x_0\right)=r  \geq\frac{1}{\left\vert f^{\prime\prime}\left(
x_{3}\right)  \right\vert }.
\end{equation}
This implies, 
\[
e^{\,-\frac{1}{2}m_\ve(x_0)}\leq e^{\,-\frac{1}{4}|x_0-x_3|}\leq e^{\,-2A\ve|x_0|} e^{\,2A\ve|x_3|}, 
\]
as long as $A\ve\leq \frac{1}{8}$.
%
But then it follows:
\[
e^{\,-\frac{1}{2}m_\ve(x_0)}e^{\,2A\ve |x_0|}\leq \exp(\ve \tau |x_3|)^{\frac{2A}{\tau}}\leq \Big(\frac{m_\ve(x_0)}{\theta}\|f''_{\ve, 1}\|_{\mathcal C^0_{\ve \tau}(\R)}\big)^\frac{2A}{\tau}.
\]
We claim that from this it follows:
\[
e^{\,-\frac{1}{2}m_\ve\left(  x_0\right)  }(\cosh x_0)^{A \varepsilon}\leq \wt C_{A, \tau, \theta} \exp\{\frac{A}{\tau}\log\|f_{\ve, 1}''\|
_{\mathcal {C}_{\varepsilon\tau}^{0}\left(  \mathbb{R}\right)  }\}.
\]
This estimate is  easy to obtain  when $m_\ve(x_0)\leq 1$. When $m_\ve(x_0)>1$ then the estimate follows from the following simple observation: 
\[
 \forall t>1 \quad -t-\frac{2A}{\tau}\log t\leq y\Longrightarrow -t\leq -\frac{1}{2}y-\frac{4A}{\tau}\log\frac{2A}{\tau}.
 \]
 Next we observe that from (\ref{meps xxx}) we have as well:
 \[
 e^{\,-m_\ve(x_0)}\leq \exp\{-\frac{1}{\|f_{\ve, 1}''\|_{\mathcal{C}^0(\R)}}\}.
 \]
 From this estimate (\ref{est dve 1}) follows if we write:
 \[
 e^{\,-m_\ve(x)}(\cosh x)^{A\ve}=e^{\,-\frac{1}{2}m_\ve(x)}e^{\,-\frac{1}{2}m_\ve(x)}(\cosh x)^{A\ve}.
 \]
The proof of the lemma is thus completed.
\end{proof}
Based on the result of the lemma we will define   a smooth function  function $\mathrm{d}_\ve$ satisfying property 
(\ref{est dve 1}) and such that ${\mathrm d}_\ve(x)\leq m_\ve(x)$. To this end we fix $A> 4\tau$ and take $\ve\ll 1$ small such that 
\[
C_{A, \tau, \theta}^{\frac{\tau}{A}}\leq \frac{1}{\sqrt{\|f_{\ve, 1}''\|_{{\mathcal C}^0(\R)}}}.
\]
Then, from (\ref{est dve 1}) it follows:
\[
m_\ve(x)\geq A\ve\log(\cosh x)+{\frac{A}{2\tau}}\log\frac{1}{\|f_{\ve, 1}''\|_{{\mathcal C}^0_{\ve \tau}(\R)}}
\]
\red{
This lower bound accounts well for the size of the function $m_\ve(x)$ whenever $\ve|x|>\log\frac{1}{\ve}$. However we need somewhat more precise estimate for the function $m_\ve(x)$ when $\ve|x|\leq \log\frac{1}{\ve}$. To obtain such an estimate we observe that arguing in a similar way as in the proof of the Lemma  \ref{lemma 3.3} we can get that:
\[
m_\ve(x)\geq \frac{\theta}{2\|f''_{\ve,1}\|_{\mathcal{C}^0(\R)}}.
\]
Then we set:
\begin{equation}
\label{def dveee}
{\mathrm d}_\ve(x)=\frac{1}{2\sqrt{2}}\Big(\frac{1}{\sqrt{\|f''_{\ve,1}\|_{\mathcal{C}^0(\R)}}}\Big)+\frac{1}{2\sqrt{2}}\Big(A\ve\log(\cosh x)+{\frac{A}{2\tau}}\log\frac{1}{\|f_{\ve, 1}''\|_{{\mathcal C}^0_{\ve \tau}(\R)}}\Big).
\end{equation}
With this definition we have $m_\ve(x)\geq \sqrt{2}\mathrm{d}_\ve(x)$.
}
From this we see that  the Fermi coordinates of $\Gamma_{\ve, 1}$ are well defined in  a  tubular neighborhood $\mathcal{{O}}_1$ of
$\Gamma_{\ve, 1}$ defined by 
\[
{\mathcal {O}}_1=\{{\tt x}\in \R^2\mid \|{\tt x}-\pi_{\ve, 1}({\tt x})\|\leq ({\mathrm d}_\ve\circ \pi^x_{\ve, 1})({\tt x})\},
\]
Moreover, $\partial\mathcal{{O}}_1$ is smooth and has bounded curvature. 
%

Letting $\left(  x_{1},y_{1}\right)  $ be the Fermi coordinate of $\Gamma_{\ve,1}$ in
$\mathcal{{O}}_1$, $y_{1}$ being the signed distance to $\Gamma_{\ve,1},$ positive in the
upper part of $\mathbb{R}^{2}\setminus\Gamma_{\ve,1}$. The coordinate transformation
is a diffeomorphism $\left(  x_{1},y_{1}\right)  \mapsto
\left(  x,y\right)  $ between the Fermi coordinates and the Euclidean
coordinate and  it is given explicitly by
\begin{align}
\label{fermi 50}
\begin{aligned}
x  &  =x_{1}-\frac{f_{\ve,1}^{\prime}\left(  x_{1}\right)  }{\sqrt{1+\left(
f_{\ve,1}^{\prime}\left(  x_{1}\right)  \right)  ^{2}}}y_{1},\\
y  &  =f_{\ve,1}\left(  x_{1}\right)  +\frac{y_{1}}{\sqrt{1+\left(  f_{\ve,1}^{\prime}\left(
x_{1}\right)  \right)  ^{2}}}.
\end{aligned}
\end{align}
Similarly, for the graph of $y=f_{\ve, 2}(x)=-f_{\ve,1}\left(  x\right)$, which is the symmetric image of $\Gamma_{\ve,1}$ with respect to the $x$ axis in the  upper half plane one can associate a
Fermi coordinate $\left(  x_{2},y_{2}\right)\in  \R\times (-{\mathrm d}_\ve,{\mathrm d}_\ve),$ with
in  ${\mathcal {O}}_2$, which is the symmetric image of ${\mathcal {O}}_1$ defined above.

Now, the change of coordinates $(x_1,y_1)\mapsto (x, y)$ is a diffeomorphism in $\mathcal{O}_1$ (respectively $(x_2,y_2)\mapsto (x,y)$ is a diffeomporhpism in the corresponding neighborhood $\mathcal {O}_2$).  We will use $\mathbf{{x}%
}_{\ve, i}:\left(  x_{i},y_{i}\right)  \rightarrow\left(  x,y\right)  $ to denote this diffeomorphism. For any function $w\colon {\mathcal{O}}_i\to \R$ we will also define its pullback by $\mathbf {x}_{\ve, i}$ by setting    $(\mathbf{{x}}_{\ve, i}^*w)\left(
x_{i},y_{i}\right)  =w\circ\mathbf{{x}}_{\ve, i}\left(  x_{i},y_{i}\right)  $.

Let ${\tt x}=(x,y)\in \mathcal {O}_i$ and let $(x_i, y_i)$ be the Fermi coordinates of this point. In what follows we will need to compare the values of $\mathrm{d}_\ve(x_i)$ with $\mathrm{d}_\ve(x)$. Note that we can write:
\[
\log\cosh x_i=\log\cosh\big(x+\mathcal{O}(\|f'_{\ve, i}\|_{\mathcal {C}^0(\R)})y_i\big), 
\]
which, since $\|f'_{\ve, i}\|_{\mathcal {C}^0(\R)}\to 0$, as $\ve\to 0$, implies:
\[
|\log\cosh x_i-\log\cosh x|\leq o(1) |y_i|\leq o(1)\mathrm {d}_\ve(x_i).
\]
Then, by definition of the function $\mathrm{d}_\ve$ it follows:
\begin{equation}
\label{ddd 1}
\mathrm{d}_\ve(x_i)=(1+o(\ve))\mathrm{d}_{\ve}(x).
\end{equation}


Another relation involving the Fermi coordinates that we will need is the following:
\begin{equation}
\label{ddd 2}
|y_1|+|y_2|\geq 2|f_{\ve,1}(x)|\big(1+\mathcal{O}(\|f'_{\ve, 1}\|^2_{\mathcal{C}^0(\R)})\big).
\end{equation}
This estimate follows from the explicit formulas for the Fermi coordinates and elementary geometry. 

\section{Asymptotic profile of a solution near its nodal line}\label{sec 33}

\subsection{An approximate solution of (\ref{AC})}

We will define now an approximate solution to (\ref{AC}) which accounts accurately for the asymptotic  the behavior of the true solution  as $\ve\to 0$.  We will use the nodal lines $\Gamma_{\ve, i}$ as the point of departure and 
will base our construction on the neighborhoods $\mathcal{O}_i$, which are expanding as $x\to  \infty$.  
 
To be  precise, we  let
${\eta}_i$ be a cutoff function satisfying ${\eta}_i\left(  {\tt x}\right)
=0,{\tt x}\not \in \mathcal{{O}}_i$ and ${\eta}_i\left(  {\tt x}\right)  =1$ for the
point ${\tt x}\in\mathcal{{O}}_i$ such that $\mathrm{dist}\,( {\tt x},\partial\mathcal{
{O}}_i)  >1.$ Moreover, ${\eta}_i$ could be chosen in such a way that
$\left\Vert {\eta}_i\right\Vert _{{\mathcal C}^{3}(\R^2)}\leq C.$ We will use $\left(
x_{i},y_{i}\right)$ to denote the Fermi coordinates associated to $\Gamma_{\ve, i}$, $i=1,2$. 
Finally, we will introduce an unknown function $h_\ve\colon \R\to \R$, which a priori is of class $\mathcal{C}^3(\R)$ , and we let
\begin{equation}
\begin{aligned}
({\mathbf x}_{\ve, 1}^*{H}_{\ve, 1})\left(x_1,y_1\right)  &=({\mathbf {x}}_{\ve, 1}^*{\eta}_1)H\left(y_{1}-{h}_\ve(x_1)\right)
+\left(  1-{\mathbf x}_{\ve, 1}^* {\eta}_1\right)  \frac{H\left(  y_{1}-{h}_\ve(x_1)\right)
}{\left\vert H\left(  y_{1}-{h}_{\ve}(x_1)\right)  \right\vert },
\\
\quad {H}_{\ve, 2}\left(  x,y\right)  & ={H}_{\ve, 1}\left(  x,-y\right)  ,\quad \bar
{u}_\ve={H}_{\ve, 1}-{H}_{\ve, 2}-1.
\end{aligned}
\label{def hve}
\end{equation}
The  function ${h}_\ve$  is called the modulation function and we will show (Proposition \ref{prop gamma})  that it can be defined  through the orthogonality condition:
\[
\int_{\mathbb{R}}\big(\mathbf{{x}}_{\ve, i}^*\left(u-\bar{u}_\ve\right)
{\rho}_{\ve, i}{H}_{\ve, i}^{\prime}\big)dy_{i}=0,\quad \forall x_{i}\in\mathbb{R},
\]
where
\[
({\mathbf x}_{\ve, i}^*  H'_{\ve, i})(x_i, y_i)= ({\mathbf {x}}_{\ve, i}^*{\eta}_i)H'\left(y_{i}-(-1)^{i+1}{h}_\ve(x_i)\right), \quad i=1,2,
\]
and smooth cutoff functions  $\rho_{\ve, i}$ are defined through a smooth cutoff function $\rho$ by:
\[
({\tt x}_{\ve, i}^* \rho_{\ve, i})(x_i,  y_i)=\rho(x_i, y_i-(-1)^{i+1}{h}_\ve(x_i)),
\]
where 
\[
\rho(s,t)=\begin{cases}
1, \quad |t|\leq \frac{1}{2}\mathrm{d}_\ve (s),\\
0<\rho<1, \quad \frac{1}{2} \mathrm{d}_\ve(s)<t<\frac{3}{4}\mathrm{d}_\ve(s),\\
0 \quad\mbox{othewise}.
\end{cases}
\]
%
Note that because of the definition of the function ${\mathrm d}_\ve$ we can assume that $|\nabla \rho_{\ve, i}({\tt x})|=\mathcal{O}(\frac{1}{\ve|{\tt x}|})$, $\ve|\tt x|\gg 1$ with similar estimates for higher order derivatives. 

The proof of existence of the modulation function  ${h}_\ve$  will be  given later on but anticipating it we  observe that due to the exponential decay in $x$ of the functions involved, we have  ${h}_\ve\in \mathcal {C}_{\varepsilon\tau
}^{2,\mu}\left(  \mathbb{R}\right)$ and in fact we will show 
\begin{equation}
\label{estim modulation}
\|h_\ve\|_{{\mathcal C}^{2,\mu}_{\ve\tau}(\R)}\leq C\ve^2.
\end{equation}

If we let $\hat{\phi} = u-\bar{u}_\ve$ then we have:
\[
L_{\bar u_\ve} \phi=E(\bar u_\ve)-P(\phi), \quad E(\bar u_\ve)=\Delta \bar u_\ve-F'(\bar u_\ve).
\]
Our first result is the following:
\begin{proposition}\label{estim hat phi}
There exist constants $\tau\in (0, \alpha_0)$, $\mu>0$ such that the following estimate holds: 
\[\|{\phi}\| _{C_{\varepsilon\tau}^{2,\mu}\left(
\mathbb{R}^{2}\right)  }\leq C\varepsilon^{2}.
\]
\end{proposition}
The proof of this Proposition, which is based on the a priori estimates for linear operator $L_{\bar u_\ve}$ in weighetd spaces and careful estimates of the error of the approximation $E(\bar u_\ve)$ is postponed for now and will be given in section \ref{proof {prop gamma}}. \red{However, it is not hard to show that a priori we have $\|\phi\|_{\mathcal{C}^0(\R^2)}=o(1)$, as $\ve\to 0$. The proof is based on De Giorgi conjecture and the fact that we know $\|\phi\|_{\mathcal{C}^{2,\mu}_{\ve\tau_0}(\R^2)}<\infty$ for some $\tau_0>0$. }

\subsection{Precise asymptotic of the nodal lines}

The point of this section is to describe precise,  and in particular  uniform  as $\ve\to 0$, estimates for the weighted norm  $\|f''_{\ve, i}\|_{\mathcal{C}^{0,\mu}_{\ve \tau}(\R)}$.
%
Our  curve of reference will be given by  a solution of the Toda system: 
\begin{equation}
\begin{aligned}
{{c_0}} \, q_1''  &=  - e^{\,\alpha_0 ( q_{1}- q_{2}) },\\
{{c_0}}\, q_2''  &= e^{\,\alpha_0(q_{1}- q_2)},
\end{aligned}
\label{toda 2}
\end{equation}
for which $q_1<0<q_2$ and $q_1(x)=-q_2(x)$, as well as $q_j(x)=q_j(-x)$, $j-1,2$ (c.f. (\ref{toda 02})). Such solution is determined by one parameter only, and in fact  we only need to solve 
\begin{align}
\label{toda 3}
{{c_0}}\, q_1''=-2e^{\, 2\alpha_0 q_1},
\end{align}
in the class of even functions. It is easy to see that solutions of (\ref{toda 3}) form a one parameter family, and each solution of this family has asymptotically linear behavior. In fact this family can be parametrized by the slope of this straight line.   To describe this family precisely  let us consider the unique solution $U_0(x)$, whose slope at $\infty$ is $-1$. Asymptotically, as $|x|\to \infty$, we have
\[
U_0(x)=-|x|+{b}_0+\mathcal{O}(e^{\,-2\alpha_0|x|}),
\]
where $b_0$ is a fixed constant. 
Then we have 
\[
q_{\ve, 1}(x)=U_0(\ve x)-\frac{1}{\alpha_0}\log\frac{1}{\ve}.
\]
Thus, given the nodal line $\Gamma_{\ve,1}$ of a solution $u$, with $\ve=\tan\theta(u)$, by $q_{\ve,1}$ we will denote the solution of (\ref{toda 3}) whose slope at infinity is $\ve$. Respectively we set 
\[
q_{\ve,2}=- q_{\ve,1}.
\] 
We will denote by $\widetilde \Gamma_{\ve,1}$ the curve $y=q_{\ve,1}(x)$ in the lower half plane and by $\widetilde \Gamma_{\ve,2}$ the graph of $y=q_{\ve,2}$.  The hope is that the nodal set in the lower half plane of a function $u$, with $\ve=\tan\theta(u)$ small,   and $\widetilde\Gamma_{\ve,1}$ should be close to each other. To quantify this  is the objective of the next result. 
\begin{proposition}\label{prop gamma}
Let $u$ be a four end solution of (\ref{AC}) such that  $\ve =\tan \theta(u)$ is small and let $\Gamma_{\ve,1}$ be the nodal line of this solution on the lower half plane, given as a graph of the function $y=f_{\ve,1}(x)$, and let  $h_\ve\in {\mathcal C}^{2,\mu}(\R)$ be the modulation function described above. There exist  $\alpha>0$ and a constant $j_\ve$, where $|j_\ve|\leq C\ve^\alpha$,  such that the following estimates  hold for the function $\chi_{\ve,1}=f_{\ve,1}+h_\ve+j_\ve-q_{\ve,1}$:
\begin{equation}
\label{first estimate}
\begin{aligned}
\left\Vert \chi_{\ve,1}\right\Vert _{{\mathcal C}^{0,\mu}_{\ve \tau}\left(\R\right)}  &  \leq C\varepsilon^{\alpha},\\
\left\Vert \chi_{\ve,1}^{\prime}\right\Vert _{{\mathcal C}^{0,\mu}_{\ve\tau}\left(\R\right)}  &  \leq C\varepsilon^{1+\alpha
},\\
\left\Vert \chi_{\ve,1}^{\prime\prime}\right\Vert _{{\mathcal C}^{0,\mu}_{\ve \tau}\left(\R\right)}  &  \leq C\varepsilon
^{2+\alpha}.
\end{aligned}
\end{equation}
Similar statements hold for the nodal line $\Gamma_{\ve,2}$ of $u$ in the upper half plane, with $\chi_{\ve,2}=f_{\ve,2}-h_\ve-j_\ve-q_{\ve,2}$,  replacing $\chi_{\ve,1}$.
\end{proposition}
This proposition is the main technical tool needed to prove the  uniqueness and.  Its proof is quite involved and we will postpone it for now proceeding directly to the proof of Theorem \ref{teo uniqueness}.

\section{Proofs of Proposition \ref{estim hat phi} and  Proposition \ref{prop gamma}}\label{proof {prop gamma}}

We recall that by definition  $h_\ve$ is required to be such that the following  orthogonality  condition are satisfied
\begin{equation}
\label{orto h}
\int_{\mathbb{R}}\big(\mathbf{x}_{\ve,i}^{\ast}\left(  u-\bar{u}_\ve\right)  \rho_{i,\ve}%
H_{\ve,i}^{\prime}\big)(x_i,y_i)dy_{i}=0,\quad \forall x_{i}\in\mathbb{R}, \quad i=1,2.
\end{equation}
We will refer to $h_\ve$ as the {\it modulation function}, and we keep in mind that $h_\ve$ is required to be small.
Our first objective is to show that the modulation function $h_\ve$ indeed exists. 

\begin{lemma}\label{exists h}
For each sufficiently small $\ve$ there exists a function $h_\ve\in {\mathcal C}^3(\R)$ such that (\ref{orto h}) holds. 
\end{lemma}

\begin{proof}
To find $h_\ve$ such that the orthogonality (\ref{orto h})  condition is satisfied, we first replace the function $h_\ve$ in the definition of 
the functions $H_{\ve,1}$ and $H_{\ve,2}$ be  two undetermined, bounded 
functions $h_{\ve,1}$ and $h_{\ve,2}.$ More precisely, given a function $h_{\ve,2}$ in
suitable function space, we have a function $H_{\ve,2}$ which in the Fermi
coordinate $\left(  x_{2},y_{2}\right)  $ is equal to $-H\left(  y_{2}%
+h_{2,\ve}\left(  x_{i}\right)  \right)$, at least near $\Gamma_{\ve, 2}$.  Given this, we want to find the function $h_{\ve,1},$
corresponding to the modulation of the nodal line $\Gamma_{\ve,1}$,   
such that for the resulting approximate solution $H_{\ve,1}$, the orthogonal
condition (\ref{orto h})  is satisfied for $i=1$.  Note that so far the orthogonal
condition for $i=2$ still may not be hold. However, if it
happens that $h_{\ve,2}=h_{\ve,1}=h_\ve$ then, by symmetry,  the orthogonal condition is also satisfied
for $i=2$ and this will yield the desired modulation function.

To find a $h_{\ve,2}$ such that $h_{\ve,1}=h_{\ve,2},$ we will use a fixed point argument. For brevity we will assume a priori that $h_{\ve,2}\in {\mathcal C}^0(\R)$, thus yielding $h_\ve\in {\mathcal C}^0(\R)$. Generalizing the argument to get $h_\ve\in {\mathcal C}^3(\R)$ is straightforward. 

Obviously,
\[
\int_{\mathbb{R}}({\mathbf x}_{\ve,1}^*\bar{u}_\ve\rho_{\ve,1}H_{\ve,1}^{\prime})(x_1, y_1)dy_{1}=-\int_{\mathbb{R}%
}\left[{\mathbf x}^*_{\ve,1} \left(H_{\ve,2}+1\right)  \rho_{\ve,1}H_{\ve,1}^{\prime}\right](x_1, y_1)dy_{1}.
\]
This identity suggests to  consider the function
\[
k_\ve(s, x_1)  :=\int_{\mathbb{R}}\rho\left( x_1,  y_{1}\right)  H^{\prime
}\left(  y_{1}\right) \big[{\mathbf x}_{\ve,1}^*(u+H_{\ve,2}+1)\big](  x_{1},y_{1}+s)   dy_{1}, \quad s\in \R.
\]
Note that the orthogonality condition (\ref{orto h}) is equivalent  to $k_\ve(s, x_1)=0$ with $s=h_{\ve,1}(x_1)$, which follows by changing variables $y_1\mapsto y_1+h_{\ve,1}(x_1)$ in the integral expression. 
  
We have
\begin{align*}
\partial_s k_\ve\left(  s, x_1\right)  &=\int_{\mathbb{R}}\rho\left(x_1,  y_{1}\right) H^{\prime}\left(  y_{1}\right)\partial_{y_1}  \big[{\mathbf x}_{\ve,1}^*(u+H_{\ve,2}+1)\big](  x_{1},y_{1}+s) dy_{1}\\
&=\underbrace{\int_{\mathbb{R}}\rho\left(x_1,  y_{1}\right) H^{\prime}\left(  y_{1}\right)\partial_{y_1}  \big[{\mathbf x}_{\ve,1}^*(H_{\ve,1}+H_{\ve,2}+1)\big](  x_{1},y_{1}+s) dy_{1}}_{l_\ve(s,x_1)}
\\
&\quad +\underbrace{\int_{\mathbb{R}}\rho\left( x_1, y_{1}\right) H^{\prime}\left(  y_{1}\right)\partial_{y_1}  \big[{\mathbf x}_{\ve,1}^*(u-H_{\ve,1})\big](  x_{1},y_{1}+s) dy_{1}}_{m_\ve(s, x_1)}
\end{align*}
Then, we see that for each $\delta>0$ there exists an $a>0$ such that $l_\ve(s,x_1)>\delta$ for $s\in (-a,a)$ uniformly in small $\ve$ and $x_1$. 

Since  $u$ converges locally, as $\ve\to 0$, to the heteroclinic solution, for each sufficiently small $\ve$ we claim that  it holds, 
\begin{align*}
|m_\ve(s, x_1)|\leq \frac{\delta}{2}, \quad \forall x_1\in \R.
\end{align*}
The proof of this claim based on the De Giorgi conjecture is left to the reader. 

Then it is seen that $\partial_s k_\ve\left(s, x_1\right)  >\delta/2$ for $s\in\left(
-a,a\right)  ,$ and $x_1\in \R$ where $a$ is small but independent of $\varepsilon.$ We will prove  that taking $a$ smaller if necessary we may assume
$k_\ve\left(  a, x_1\right)  >0$ and $k_\ve\left(  -a, x_1\right)  <0$ for $\varepsilon$ small
enough. Indeed let us write:
\begin{align*}
k_\ve(s, x_1)&=\underbrace{\int_{\mathbb{R}}\rho\left(x_1,  y_{1}\right) H^{\prime}\left(  y_{1}\right) H(y_{1}+s) dy_{1}}_{k_0(s)}\\
&\quad +
\underbrace{\int_{\mathbb{R}}\rho\left( x_1, y_{1}\right) H^{\prime}\left(  y_{1}\right)  \big[{\mathbf x}_{\ve,1}^*(H_{\ve,2}+1)\big](  x_{1},y_{1}+s) dy_{1}}_{g_{1,\ve}(s,x_1)}\\
&\quad +
\underbrace{\int_{\mathbb{R}}\rho\left( x_1, y_{1}\right) H^{\prime}\left(  y_{1}\right)  \big[{\mathbf x}_{\ve,1}^*(u-H_{\ve,1})\big](  x_{1},y_{1}+s) dy_{1}}_{g_{2,\ve}(s,x_1)}.
\end{align*}
Then we see that
\begin{equation}
\label{hhh 1}
k_0(s)=s\int\rho(x_1,y_1)\big(H'(y_1)\big)^2\,dy_1+k_1(s), \quad k_1(s)\sim s^2, 
\end{equation}
while
\begin{equation}
\label{hhh 2}
g_{1,\ve}(s,x_1)=o(1) e^{\,-\sqrt{2}( |h_{\ve,2}|+|s|)}, \quad g_{2,\ve}(s,x_1)=o(1), \quad\mbox{as}\ \ve\to 0.
\end{equation}
Therefore for fixed $h_{\ve,2},$ the existence of $h_{\ve,1}$ which fulfills
the orthogonal condition (\ref{orto h})  follows immediately.
The above argument implies that for any $h_{\ve,2}\in {\mathcal C}^{0}\left(  \mathbb{R}%
\right)  ,\left\Vert h_{\ve,2}\right\Vert_{{\mathcal C}^0(\R)} < a,$ we have a nonlinear map $T$ defined by
$h_{\ve,2}\mapsto h_{\ve,1}.$ The map $T$  satisfies
\[
TB\left(  0,a\right)  \subset B\left(  0,a\right), \quad B(0,a)=\{h\in {\mathcal C}^0(\R)\mid \|h\| _{{\mathcal C}^0(\R)} < a\}.
\]
The proof that $T$ is a contraction map is standard and is omitted. At the end we obtain the existence of a fixed point $h_\ve$. To prove its regularity we note that the fact that we have $\partial_s k_\ve(s, x_1)=l_\ve(x_1s)+m_\ve(s,x_1)$, allows to use the implicit function theorem and thus the regularity follows in a  straightforward manner. This ends the proof.  
\end{proof}

\begin{corollary}
\label{corollary orto 1}
The modulation function $h_\ve$ satisfies:
\begin{equation}
\|h_\ve\|_{\mathcal C^{2,\mu}(\R)}=o(1), \quad \ve\to 0.
\end{equation}
We also have $h_\ve\in \mathcal{C}^{2,\mu}_{\ve\tau}(\R)$.
\end{corollary}
\begin{proof}
The fact that $\|h_\ve\|_{\mathcal {C}^0(\R)}\to 0$ as $\ve \to 0$ follows from (\ref{hhh 1})--(\ref{hhh 2}). Then the same can be shown for the higher order derivatives. Once the existence of small $h_\ve$ is established one can use again (\ref{hhh 1})--(\ref{hhh 2}) and the fact that a priori $u\in \mathcal {C}^{2,\mu}_{\ve\tau}(\R^2)$ to show that $h_\ve\in \mathcal {C}^{2,\mu}_{\ve \tau}(\R)$. 
\end{proof}

Let  us  write $u=\bar{u}_\ve+\phi$, where $u$ is a solution.  Let us denote the linearization of the Allen-Cahn equation around $\bar u$ by $L_{\bar u_\ve}:=-\Delta + F''(\bar u_\ve)$. Then $\phi$ satisfies
\begin{equation}
L_{\bar{u}_\ve}\phi
=\Delta\bar{u}_\ve-F^{\prime}\left(  \bar{u}_\ve\right)  -P\left(  \phi\right)  .
\label{ubar}%
\end{equation}
Here
\[
P\left(  \phi\right)  =F^{\prime}\left(  \bar{u}_\ve+\phi\right)  -F^{\prime
}\left(  \bar{u}_\ve\right)  -F^{\prime\prime}\left(  \bar{u}_\ve\right)  \phi\sim\phi^2,
\]
is a higher order term in $\phi$. Note that our definition of $\bar u_\ve$ and the construction of the function $h_\ve$
implies  that $\phi=u-\bar u_\ve$ satisfies the  orthogonality condition  (\ref{orto h}). Our strategy 
 to get suitable estimate for $\phi$ relies on the a priori estimates
for the operator $L_{\bar{u}}$, taking into account this orthogonality condition.

To carry out the analysis, we will  study the error term $E\left(  \bar
{u}_\ve\right)  :=\Delta\bar{u}_\ve-F^{\prime}\left(  \bar{u}_\ve\right)  .$ 
First we  consider the
projection of $E\left(  \bar{u}_\ve\right)  $ onto the two dimensional space  $K=\mathrm{span}\,\{H_{\ve,i}^{\prime}\rho_{\ve,i}, i=1,2\}$, which we will denote by $E(\bar u_\ve)^\parallel$. We will also set $E(\bar u_\ve)^\perp=E(\bar u_\ve)-E(\bar u_\ve)^\parallel$. Explicitly $E(\bar u_\ve)^\perp$ is given through its pullback by the Fermi coordinates of $\Gamma_{\ve, i}$ as:
\[
{\mathbf x}_{\ve,i}^*E\left(  \bar{u}_\ve\right)^{\perp}:={\mathbf x}_{\ve,i}^*E\left(  \bar{u}_\ve\right)
-\sum\limits_{i=1}^{2}c_{\ve}(x_i)\big({\mathbf x}_{\ve,i}^*\rho_{\ve,i}H_{\ve,i}^{\prime}\big)\int_{\mathbb{R}}\big[
{\mathbf x}_{\ve,i}^*E(\bar{u}_\ve)\rho_{\ve,i}H_{\ve,i}^{\prime}\big]dy_{i},
\]
where
\[
c_{\ve}(s)=\left(  \int_{\mathbb{R}}\rho^{2}(s,t)\left(  H^{\prime}(t)\right)  ^{2}\,dt\right)
^{-1}\approx \left(\int_\R(H')^2\right)^{-1}.
\]
The main idea in what follows is that the size of the function $f_{\ve, 1}$ and its derivatives should be controlled by $E(\bar u_\ve)^\parallel$, while the size of $u-\bar u_\ve =\phi$ is controlled by $E(\bar u_\ve)^\perp$. Of course, both projections of the error $E(\bar u_\ve)$ are coupled, in the sense that the dependence on $f_{\ve, 1}$ and $\phi$ appears in both of them, but  this coupling is relatively weak.

Recall the expression of Laplace operator in the Fermi coordinate of $\Gamma_{i,\ve}$:
\begin{equation}
\Delta=\frac{1}{A_{i}}\partial_{x_{i}}^{2}+\partial_{y_{i}}^{2}+\frac{1}%
{2}\frac{\partial_{y_{i}}A_{i}}{A_{i}}\partial_{y_{i}}-\frac{1}{2}%
\frac{\partial_{x_{i}}A_{i}}{A_{i}^{2}}\partial_{x_{i}}, \label{laplacian}%
\end{equation}
where
\[
A_{i}=1+\left(  f^{\prime}_{\ve, i}\left(  x_{i}\right)  \right)  ^{2}+2y_{i}%
\frac{\left(  -1\right)  ^{i}f^{\prime\prime}_{\ve, i}\left(  x_{i}\right)  }%
{\sqrt{1+\left(  f^{\prime}_{\ve, i}\left(  x_{i}\right)  \right)  ^{2}}}+y_{i}%
^{2}\frac{\left(  f^{\prime\prime}_{\ve, i}\left(  x_{i}\right)  \right)  ^{2}}{\left(
1+\left(  f^{\prime}_{\ve, i}\left(  x_{i}\right)  \right)  ^{2}\right)  ^{2}}.
\]
Using these formulas,  we can write down the expression of the error
$E\left(  \bar{u}_\ve\right)  .$ Because of symmetry, it suffices to carry out the calculation in  the lower half
plane. Observe that,
\begin{align*}
&  -F^{\prime}\left(  H_{\ve, 2}\right)  -F^{\prime}\left(  H_{\ve, 1}-H_{\ve, 2}-1\right) \\
&  =-F^{\prime}\left(  H_{\ve, 2}\right)  -F^{\prime}\left(  H_{\ve, 1}\right)
+F^{\prime\prime}\left(  H_{\ve, 1}\right)  \left(H_{\ve, 2}+1\right)  
+{\mathcal O}\left((H_{\ve, 2}+1)^{2}\right)\\
&  =-F^{\prime}\left(  H_{\ve, 1}\right)  -\left(  F^{\prime\prime}\left(
1\right)  -F^{\prime\prime}\left(  H_{\ve, 1}\right)  \right)  \left(
H_{\ve, 2}+1\right)  +{\mathcal O}\left( ( H_{\ve, 2}+1)^2\right) .
\end{align*}
The same calculation as in formula (5.65) in \cite{MR2557944} yields that in the portion of the lower half plane where   ${\mathrm{dist}}\,(\Gamma_{\ve, i}, {\tt x})\leq (\mathrm {d}_\ve\circ\pi_{\ve, i}^x)({\tt x}) -1$, for both  $i=1,2$, we have
\begin{align}
\begin{aligned}
E\left(  \bar{u}_\ve\right)   &  =\left(  \frac{1}{2}\frac{\partial_{y_{1}}A_{1}%
}{A_{1}}-\frac{h_\ve^{\prime\prime}\left(  x_{1}\right)  }{A_{1}}  +\frac{1}{2}\frac{\partial_{x_{1}}A_{1}%
}{A_{1}^{2}}h_\ve^{\prime}\left(  x_{1}\right)  \right)  H_{\ve, 1}^{\prime}\\
& \quad -\left(  \frac{1}{2}\frac{\partial_{y_{2}}A_{2}}{A_{2}}+\frac
{h_\ve^{\prime\prime}\left(  x_{2}\right)  }{A_{2}}-\frac{1}{2}\frac
{\partial_{x_{2}}A_{2}}{A_{2}^{2}}h_\ve^{\prime}\left(  x_{2}\right)  \right)
H_{\ve,2}^{\prime}\\
& \quad +\left(  \frac{\left(  h_\ve^{\prime}\left(  x_{1}\right)  \right)  ^{2}}%
{A_{1}}H^{\prime\prime}\left(  y_{1}-h_\ve\left(  x_{1}\right)  \right)
-\frac{\left(  h_\ve^{\prime}\left(  x_{2}\right)  \right)  ^{2}}{A_{2}}%
H^{\prime\prime}\left(  y_{2}+h_\ve\left(  x_{2}\right)  \right)  \right)\\
&\quad  -\left(
F^{\prime\prime}\left(  1\right)  -F^{\prime\prime}\left(  H_{\ve, 1}\right)
\right)  \left(  H_{\ve, 2}+1\right)+{\mathcal O}\left( (H_{\ve, 2}+1)^2\right) .
\end{aligned}
\label{eu}%
\end{align}
In fact this formula generalizes in the set  ${\mathrm {dist}}\,(\Gamma_{\ve, 1}, {\tt x})<(\mathrm{d}_\ve\circ\pi_{\ve, 1}^x)({\tt x})-1$, ${\mathrm {dist}}\,(\Gamma_{\ve, 2}, {\tt x})>({\mathrm d}_\ve\circ\pi_{\ve, 2}^x)({\tt x})$,  if we set $H_{\ve, 2}=-1$. In the intermediate region where $(\mathrm {d}_\ve\circ \pi_{\ve, 1}^x)({\tt x})-1\leq {\mathrm{dist}}\,(\Gamma_{\ve, 1}, ({\tt x})<(\mathrm {d}_\ve\circ \pi_{\ve, 1}^x)({\tt x})$ we control $E(\bar u_\ve)$ by $Ce^{\,-\sqrt{2} (\mathrm {d}_\ve\circ \pi_{\ve, 1}^x)({\tt x})}$.  Finally, in the set where ${\mathrm {dist}}\,(\Gamma_{\ve, i}, {\tt x})>(\mathrm {d}_\ve\circ\pi_{\ve, i}^x)({\tt x})$ the error is $0$, since $\bar u_\ve=\pm 1$ in this set.  These is the basis of the proof of the next lemma.

\begin{lemma}
\bigskip\label{Eu}
For any  $\mu\in\left(  0,1\right)$, 
the following estimate holds%
\begin{equation}
\|E(\bar{u}_\ve)^{\perp}\|
_{{\mathcal C}^{0,\mu}_{\ve \tau}(\R^2)}=o\left(  \left\Vert f_{\ve,1}^{\prime\prime}\right\Vert _{{\mathcal C}^{0,\mu}_{\ve\tau}(\R)%
}+\left\Vert h_\ve^{\prime}\right\Vert _{{\mathcal C}^{0,\mu}_{\ve\tau}(\R)}+\left\Vert h_\ve^{\prime\prime
}\right\Vert _{{\mathcal C}^{0,\mu}_{\ve\tau}(\R)}\right)  +\mathcal{O}\left(  \|\exp\{-2\sqrt{2}|f_{\ve,1}|(1+\varDelta_\ve)\}\|_{\mathcal{C}_{\ve\tau}^{0,\mu}(\R)}\right), \label{E}%
\end{equation}
where we have denoted  $\varDelta_\ve=\mathcal{O}(\|f_{\ve, 1}'\|^2_{\mathcal{C}^0(\R)})$. 
\end{lemma}

\begin{proof}
\red{First we note that, by (\ref{ddd 1}) whenever we control $\mathrm{d}_\ve(x_i)$ we also control $\mathrm{d}_\ve(x)$ and then from the definition of the function $\mathrm{d}_\ve$ and (\ref{est dve 1}) it follows 
\[
\|e^{-\sqrt{2}\mathrm{d}_{\ve}}\|_{\mathcal{C}_{\ve\tau}^0(\R)}=o\left(  \left\Vert f_{\ve,1}^{\prime\prime}\right\Vert _{\mathcal {C}^{0,\mu}_{\ve\tau}(\R)}\right)  .
\]
Thus, whenever a term the error $E(\bar u_\ve)({\tt x})$ can be controlled by 
$e^{-\sqrt{2}\mathrm{d}_{\ve}(x_i)},$ then the $\mathcal{C}^0_{\ve\tau}(\R)$ norm of such term  can be 
controlled as well by $o\left(  \left\Vert f_{\ve,1}^{\prime\prime}\right\Vert _{C^{0,\mu}_{\ve\tau}(\R)%
}\right)  .$ }Therefore to prove $\left(  \ref{E}\right)  ,$ it suffices to
check the expression $\left(  \ref{eu}\right)  $ for  $E\left(  \bar{u}_\ve\right)$, which applies whenever at least for one $i$ it holds ${\mathrm {dist}}\,(\Gamma_{\ve, i}, {\tt x})<\mathrm{d}_\ve(x_i)-1$, as we have pointed out above. This means that we only need to consider the subset of the lower half plane were $|y_i|=\mathrm{dist}\,(\Gamma_{\ve, i}, {\tt x})\leq \mathrm {d}_\ve(x_i)-1$ for at least one $i$. Thus we may focus on studying  the formula (\ref{eu}).

Projecting the $E(\bar u_\ve)$ on $K$, and using formula (\ref{eu}), we get for instance the following term in  $E\left(  \bar{u}_\ve
\right)  ^{\perp}$:
\[
T_{1}:=\frac{\partial_{y_{1}}A_{1}}{A_{1}}H_{\ve, 1}^{\prime}-c_{\ve}\rho_{\ve, 1}%
H_{\ve, 1}^{\prime}\int_{\mathbb{R}}\frac{\partial_{y_{1}}A_{1}}{A_{1}}\rho
_{\ve,1}\left(  H_{\ve, 1}^{\prime}\right)  ^{2}dy_{1}.
\]

Recall that
\[
\frac{\partial_{y_{1}}A_{1}}{A_{1}}=-2\frac{f_{\ve,1}^{\prime\prime}\left(
x_{1}\right)  }{A_{1}\sqrt{1+\left(  f^{\prime}_{\ve, 1}\left(  x_{1}\right)  \right)
^{2}}}+2\frac{y_{1}\left(  f^{\prime\prime}_{\ve, 1}\left(  x_{1}\right)  \right)
^{2}}{A_{1}\left(  1+\left(  f^{\prime}_{\ve, 1}\left(  x_{1}\right)  \right)
^{2}\right)  ^{2}}.
\]
Substituting this into the expression of $T_{1}$ results in
\begin{align*}
T_{1}  &  =\frac{\partial_{y_{1}}A_{1}}{A_{1}}H_{\ve, 1}^{\prime}+\frac{2c_{\ve}%
\rho_{\ve, 1}H_{\ve, 1}^{\prime}f_{\ve, 1}^{\prime\prime}\left(  x_{1}\right)  }{\sqrt{1+\left(
f_{\ve, 1}^{\prime}\left(  x_{1}\right)  \right)  ^{2}}}\int_{\mathbb{R}}\frac{\rho
_{\ve, 1}\left(  H_{\ve, 1}^{\prime}\right)  ^{2}}{A_{1}}dy_{1}\\
&  -\frac{2c_{\ve}\rho_{\ve, 1}H_{\ve, 1}^{\prime}\left(  f_{\ve, 1}^{\prime\prime}\left(
x_{1}\right)  \right)  ^{2}}{\left(  1+\left(  f_{\ve, 1}^{\prime}\left(  x_{1}\right)
\right)  ^{2}\right)  ^{2}}\int_{\mathbb{R}}\frac{y_{1}\rho_{\ve, 1}\left(
H_{\ve, 1}^{\prime}\right)  ^{2}}{A_{1}}dy_{1}.
\end{align*}
\red{Note that although $y_{1}$ appears in $\frac{\partial_{y_1}A_{1}}{A_1}$, it is always multiplied by
$f^{\prime\prime}_{\ve, 1}\left(  x_{1}\right)$ and an exponentially decaying function, namely $H'_{\ve, 1}$,  which gives  a lower order  term since
\begin{equation}
\label{yonef}
H'_{\ve, 1}|y_1||f_{\ve, 1}''(x_1)|\leq C\|f_{\ve,1}''\|_{\mathcal{C}^0(\R)}=o(1).
\end{equation}}
Another observation we make is that when we estimate $\mathcal{C}^0_{\ve\tau}(\R)$ norms we need to take into account the relation between the Fermi variables $(x_1, y_1)$ and the euclidean coordinates $(x,y)$ of a point ${\tt x}\in \mathcal{O}_1$. To see this let us consider  a typical term that appears in $T_1$:
\[
|(\cosh x)^{\ve\tau} f''_{\ve, 1}(x_1)|\leq C e^{\,\ve\tau|x_1-x|}\|f''_{\ve, 1}\|_{\mathcal{C}^0_{\ve\tau}{(\R)}}\leq C\exp\{\ve\tau|y_1|\mathcal{O}(\|f'_{\ve, 1}\|_{\mathcal{C}^{0}(\R)})\}\|f''_{\ve, 1}\|_{\mathcal{C}^0_{\ve\tau}{(\R)}}.
\]
Any term of this form is additionally multiplied by $o(1)H_{\ve, 1}'$ or $o(1)H''_{\ve, 1}$ thus yielding a term of the order 
$o(\|f''_{\ve, 1}\|_{\mathcal{C}^0_{\ve\tau}{\R}})$. 


Finally, it is important to notice that although it appears at first that $T_1$ carries a term of order $\mathcal{O}(\|f''_{\ve, 1}\|_{\mathcal{C}^0_{\ve\tau}{\R}})$, there is a cancellation between the first and the second term (the one containing the integral) in $T_1$. In estimating this term it is important to use the properties of the cut off function $\rho_{\ve, 1}$.

Now, using the fact that $f_{\ve,1}^{\prime}$
and $f_{\ve, 1}^{\prime\prime}$ are of order $o\left(  1\right)$, as $\ve \to 0$,  and the definition
of the cutoff function $\rho_{\ve, 1},$ we conclude
\[
\left\Vert T_{1}\right\Vert _{{\mathcal C}^{0,\mu}(\R^2)}=o\left(  \left\Vert f_{\ve, 1}^{\prime\prime
}\right\Vert _{{\mathcal C}^{0,\mu}(\R)}\right)  .
\]
Similar estimates hold for the terms involving $h_\ve^{\prime\prime}\left(
x_{1}\right)  .$ Regarding to the terms involving $h_\ve^{\prime}\left(
x_{1}\right)  ,h_\ve^{\prime}\left(  x_{2}\right)  ,h_\ve^{\prime\prime}\left(
x_{2}\right)  $, we note that they are all multiplied by a small order term. Furthermore,
the norms of $\left(  H_{\ve, 2}+1\right)  H_{\ve, 1}^{\prime}$ and ${\mathcal O}(  \left(
H_{\ve, 2}+1\right)  ^{2})  $ are controlled by $Ce^{\,-\sqrt{2}(|y_1|+|y_2|)}$. To estimate terms of this from we use the expressions for the Fermi coordinates of $\Gamma_{\ve, i}$ to arrive at the following lower bound:
\[
|y_1|+|y_2|\geq 2(1+\varDelta_\ve)|f_{\ve, 1}(x)|.
\]
This ends  the proof.
\end{proof}

Observe that there are terms involving $h_\ve$ which appear in the right hand side of
$\left(  \ref{E}\right)  .$ This somewhat complicates the situation. However,
since the Fermi coordinates are defined using the nodal line, we have the following

\begin{lemma}
It holds
\label{hf1}%
\begin{equation}
\left\Vert h_\ve\right\Vert _{C^{2,\mu}_{\ve\tau}\left(  \mathbb{R}\right)  }\leq
C\left\Vert \phi\right\Vert _{C^{2,\mu}_{\ve\tau}\left(  \mathbb{R}^{2}\right)
}+C\|\exp\{-2\sqrt{2}|f_{\ve,1}|(1+\varDelta_\ve)\}\|_{\mathcal{C}_{\ve\tau}^{0,\mu}(\R)}
\label{hf}
\end{equation}

\end{lemma}

\begin{proof}
We first recall that in the set  $\mathrm{dist}\, (\Gamma_{\ve, 1}, {\tt x})<\mathrm{d}_\ve({\tt x})-1$,
\begin{equation}
({\mathbf x}_{\ve, 1}^*u)\left(
x_{1},y_{1}\right) =H\left(  y_{1}-h_\ve\left(  x_{1}\right)  \right)  -({\mathbf x}_{\ve, 1}^*H_{\ve, 2})\left(
x_{1},y_{1}\right)  -1+({\mathbf x}_{\ve, 1}^*\phi)\left(  x_{1},y_{1}\right). \label{hfu0}%
\end{equation}
Letting $y_{1}=0$ in the above identity and using  that $x_1=x$,  we have:
\[
\left\vert ({\mathbf x}_{\ve, 1}^*H_{\ve, 2})\left(  x_{1},y_{1}\right)  +1\right\vert \leq\begin{cases}
C\exp\big\{-2\sqrt{2}(|f_{\ve,1}(x)|(1+\varDelta_\ve)\big\}, \quad (x,f_{\ve, 1}(x))\in \mathcal{O}_2,\\
0, \quad (x,f_{\ve, 1}(x))\notin \mathcal{O}_2.
\end{cases}
\]
Then from  $({\mathbf x}_{\ve, 1}^*u)\left(
x_{1},0\right)=0$ one gets%
\[
\|h_\ve\|_{\mathcal{C}_{\ve\tau}^{0}(\R)} \leq C\|\phi\|
_{{\mathcal C}^{0}_{\ve\tau}(\R^2)}+C\|\exp\{-2\sqrt{2}|f_{\ve,1}|(1+\varDelta_\ve)\}\|_{\mathcal{C}_{\ve\tau}^{0,\mu}(\R)}.
\]
This gives us a ${\mathcal C}^{0}$ estimate. To estimate  the ${\mathcal C}^{1}$ norm of $h_\ve$, we differentiate the relation 
$\left(  \ref{hfu0}\right)  $ with respect to $x_{1}$ and let $y_{1}=0$ in
the resulting equation. Then we find that%
\begin{equation}
-H^{\prime}\left(  -h_\ve\left(  x_{1}\right)  \right)  h_\ve^{\prime}\left(
x_{1}\right)  -\frac{\partial }{\partial x_{1}}({\mathbf x}_{\ve, 1}^* H_{\ve, 2})+\frac{\partial}{\partial x_{1}}({\mathbf{x}}_{\ve, 1}^*\phi)=0, \label{d1}%
\end{equation}
from which the ${\mathcal C}^1_{\ve\tau}$ estimate follows.
Similarly, we could differentiate the equation $\left(  \ref{hfu0}\right)  $
twice with respect to $x_{1}$ and let $y_{1}=0$ to estimate $h_{\ve}''$.

Corresponding estimates for the Holder norm are also straightforward. 
\end{proof}

To proceed, we need the following {\it a priori} estimate:
\red{
\begin{proposition}
\label{apriori}
Suppose $\varphi$ is a solution of the following equation:%
\[
-\Delta\varphi+F^{\prime\prime}\left(  \bar{u}_\ve\right)  \varphi=f+\sum_{i=1,2}\kappa_{i,\ve}\rho_{\ve,i}H'_{\ve, i},\quad \mbox{in}\ \R^2,
\]
with some given functions $f\in \mathcal{C}^{0,\mu}_{\ve\tau}(\R^2)$ and $\kappa_{\ve, i}\in\mathcal{C}^{0,\mu}_{\ve\tau}(\R)$. Assume furthermore that the function $\varphi$ satisfies the  orthogonality condition:
\begin{equation}
\int_{\mathbb{R}}\big[\mathbf{x}_{\ve, i}^{\ast}\varphi\rho_{\ve, i}H_{\ve, i}^{\prime}\big]
dy_{i}=0, \quad i=1,2.
\label{ort xxx}
\end{equation}
Then it holds
\[
\left\Vert \varphi\right\Vert _{C^{2,\mu}_{\ve\tau}(\R^2)}\leq C\left\Vert f\right\Vert
_{C^{0,\mu}_{\ve\tau}(\R^2)},\quad \|\kappa_{\ve, i}\|_{C^{0,\mu}_{\ve\tau}(\R)}\leq C\left\Vert f\right\Vert
_{C^{0,\mu}_{\ve\tau}(\R^2)},
\]
provided $\varepsilon$ is small enough.
\end{proposition}

}
The proof is by contradiction and it is  essentially the same as that of Proposition  5.1 in \cite{MR2557944} and consists of the following steps:  first  an {\it a priori} estimate is proven for a solution of  the following problem:
\[
-\Delta \varphi+F''(\bar u_\ve)=f, \quad\mbox{in}\ \R^2, 
\]
where $\varphi$ satisfies the orthogonality condition (\ref{ort xxx}). In fact, using the fact that the heteroclinic solution in $\R$ is neutrally stable one can prove that $\varphi$ satisfies an estimate of the form claimed in the Proposition. This type of argument can be found for example in \cite{dkp_dg}. 

Second, we project the equation on the functions of the form $\rho_{\ve,j}H'_{\ve, j}$, $j=1,2$ and arrive at the following expressions
\[
\int_\R {\tt x}^*_{\ve, j}\big(\rho_{\ve,j}H'_{\ve, j}[-\Delta\varphi+F^{\prime\prime}\left(  \bar{u}_\ve\right)  \varphi]\big)(x_j, y_j)\,dy_j-\int_\R{\tt x}^*_{\ve, j}\big(\rho_{\ve,j}H'_{\ve, j}f\big)(x_j,y_j)\,dy_j=\sum_{i=1,2}\kappa_{i,\ve}\int_R{\tt x}^*_{\ve, j}\big(\rho_{\ve,i}H'_{\ve, i}\big)^2(x_j,y_j)\,dy_j.
\]
After an integration by parts and some calculations we can prove using the above identity that the $\mathcal{C}^{0,\mu}_{\ve\tau}(\R)$ norm of the functions $\kappa_{i,\ve}$  can be controlled by $o(1)\|\varphi\|_{\mathcal{C}^{0,\mu}_{\ve\tau}(\R^2)}+C\|f\|_{\mathcal{C}^{0,\mu}_{\ve\tau}(\R^2)}$. From this and the first step the assertion of the Proposition follows.
We omit the details.

With this result at hand, now we could prove

\begin{lemma}
\label{fif}
Let $\phi=u-\bar u_\ve$ be the solution of (\ref{ubar}). The following estimate is true:%
\begin{equation}
\left\Vert \phi\right\Vert _{{\mathcal C}^{2,\mu}_{\ve\tau}(\R^2)}\leq o\big(\left\Vert
f_{\ve, 1}^{\prime\prime}\right\Vert _{{\mathcal C}^{0,\mu}_{\ve\tau}(\R)}\big)+C\|\exp\{-2\sqrt{2}|f_{\ve,1}|(1+\varDelta_\ve)\}\|_{\mathcal{C}_{\ve\tau}^{0,\mu}(\R)}. \label{f1}%
\end{equation}
\end{lemma}

\begin{proof}
We will use Proposition \ref{fif}. 
Thus we write:
\[
-\Delta\phi+F''(\bar u_\ve)\phi=E(\bar u_\ve)^\perp+P(\phi)+E(\bar u_\ve)^\parallel.
\]
Because of Proposition \ref{apriori} to control the size of the function $\phi$ it suffices to control the size of $E(\bar u_\ve)^\perp$ (which we already do by Lemma \ref{Eu}) and the size of $P(\phi)$.

Next 
we observe that $P\left(
\phi\right)  $ is essentially quadratic in $\phi$, and  therefore it is not difficult to
show
\[
\left\Vert  P\left(  \phi\right) 
\right\Vert _{{\mathcal C}^{0,\mu}_{\ve\tau}(\R^2)}=o\left(  \left\Vert \phi\right\Vert
_{{\mathcal C}^{2,\mu}_{\ve\tau}(\R^2)}\right)  .
\]
Collecting all these estimates, we
conclude (\ref{f1}).
%

\end{proof}

Roughly speaking, the above result indicates that we can control $\phi$ by $\|e^{\,-2\sqrt{2}|f_{1,\ve}|}\|_{\mathcal{C}^0_{\ve\tau}(\R)}$
and the second derivative of $f_{\ve, 1}$. However, this is not quite enough for our
later purpose. Remark that for the solution constructed in \cite{MR2557944}, the
corresponding error is roughly speaking controlled by $C\varepsilon^{2}.$ On the other hand, intuitively, $\|f_{\ve, 1}-\ve |x|\|_{\mathcal{C}^0(\R)}\sim \log\frac{1}{\ve}$. This indeed would be true if $f_{\ve,1}$, $f_{\ve, 2}$ were solutions of the Toda system with the asymptotic slope $\ve$ at $\pm \infty$. For now we will show:

\begin{lemma}\label{fi}
The following estimate holds:
\begin{align*}
\left\Vert \phi\right\Vert _{{\mathcal C}^{2,\mu}_{\ve\tau}(\R^2)}\leq C\|\exp\{-2\sqrt{2}|f_{\ve,1}|(1+\varDelta_\ve)\}\|_{\mathcal{C}_{\ve\tau}^{0,\mu}(\R)}.
\end{align*}
\end{lemma}

\begin{proof}
Let us consider the integral $\int_{\mathbb{R}}\left( {\mathbf x}_{\ve, 1}^* E\left(  \bar{u}_\ve\right) \rho_{\ve, 1}H_{\ve, 1}^{\prime} \right)dy_{1}.$ We will show below (Step 1) that on one hand 
its ${\mathcal C}^{0,\mu}_{\ve\tau}(\R)$ norm is controlled by $o\left(  \left\Vert \phi\right\Vert
_{C^{0,\mu}_{\ve\tau}}\right)  .$ On the other hand (Step 2) we will show that this integral is controlled by $f''_{\ve,1}$. Then the proof of the Lemma will follow by combining this with the previous estimates.

\noindent
{\bf Step 1.}
We claim that the integral $\int_{\mathbb{R}}\left( {\mathbf x}_{\ve, 1}^* E\left(  \bar{u}_\ve\right) \rho_{\ve, 1}H_{\ve, 1}^{\prime} \right)dy_{1}$ is controlled by $o\left(  \left\Vert \phi\right\Vert
_{C^{0,\mu}_{\ve\tau}}\right)  .$
Clearly it is sufficient to estimate $E(\bar u_\ve)^\parallel$. 
We will show
\begin{equation}
\left\Vert c_{\ve}\big({\mathbf x}_{\ve,1}^*\rho_{\ve,1}H_{\ve,1}^{\prime}\big) \int_{\mathbb{R}}\big[{\mathbf x}_{\ve, 1}^* E(\bar{u}_\ve)\rho_{\ve, 1}H_{\ve, 1}^{\prime}\big]dy_{1}\right\Vert _{{\mathcal C}^{0,\mu}_{\ve\tau}(\R)}=o\left(  \left\Vert
\phi\right\Vert _{{\mathcal C}_{\ve\tau}^{2,\mu}(\R^2)}\right)  . \label{ep}%
\end{equation}
In fact,
\begin{align*}
\int_{\mathbb{R}}\big[{\mathbf x}_{\ve, 1}^* E(\bar{u}_\ve)\rho_{\ve, 1}H_{\ve, 1}^{\prime}\big]dy_{1}  &  =\int_{\mathbb{R}}\big[{\mathbf x}_{\ve, 1}^*(  -\Delta\phi
+F^{\prime\prime}(  \bar{u}_\ve)  \phi)\rho_{\ve, 1}%
H_{\ve, 1}^{\prime}\big]dy_{1}\ \\
&  +\int_{\mathbb{R}}\big[{\mathbf x}_{\ve, 1}^*  P(\phi)\rho
_{\ve, 1}H_{\ve, 1}^{\prime}\big]dy_{1}.
\end{align*}
To handle the first term appearing in the right hand side we write $\Delta_{(x_1, y_1)}=\partial_{x_1}^2+\partial_{y_1}^2$ and:
\begin{align*}
T_{2}  &  :=\underbrace{\int_{\mathbb{R}}(  -\Delta_{\left(  x_{1},y_{1}\right)
}{\mathbf x}_{\ve, 1}^*\phi+F^{\prime\prime}(H) {\mathbf x}_{\ve, 1}^* \phi)({\mathbf x}_{\ve, 1}^*\rho_{\ve, 1}%
H_{\ve, 1}^{\prime})dy_{1}}_{T_{21}}\\
& + \underbrace{\int_{\mathbb{R}}\big[ - ({\mathbf x}_{\ve, 1}^*  \Delta\phi) 
-\Delta_{(x_{1},y_{1})}({\mathbf x}_{\ve, 1}^* \phi)\big]+\big[({\mathbf x}_{\ve, 1}^* F^{\prime\prime}\left(  \bar{u}_\ve\right)\phi)
-F^{\prime\prime}\left(  H\right){\mathbf x}_{\ve, 1}^*   \phi\big] ({\mathbf x}_{\ve, 1}^* \rho_{\ve, 1}%
H_{\ve, 1}^{\prime})dy_{1}}_{T_{22}}.
\end{align*}
Using that  from $\int_{\mathbb{R}}\big[{\mathbf x}_{\ve, 1}^*\phi\rho_{\ve,1}H_{\ve,1}^{\prime}\big]dy_{1}=0,$ it follows $\partial_{x_1}\int_{\mathbb{R}}\big[{\mathbf x}_{\ve, 1}^*\phi\rho_{\ve,1}H_{\ve,1}^{\prime}\big]dy_{1}=0$, 
we get%
\begin{align*}
\  T_{21}&  =-2\int_{\mathbb{R}}\frac{\partial({\mathbf x}_{\ve, 1}^*\phi)}{\partial x_{1}}%
\frac{\partial\left( {\mathbf x}_{\ve, 1}^* \rho_{\ve,1}H_{\ve, 1}^{\prime}\right)  }{\partial x_{1}}%
-\int_{\mathbb{R}}({\mathbf x}_{\ve, 1}^*\phi)\frac{\partial^{2}\left( {\mathbf x}_{\ve, 1}^* \rho_{\ve,1}H_{\ve,1}^{\prime
}\right)  }{\partial x_{1}^{2}}\\
&\quad   +\int_{\mathbb{R}}({\mathbf x}_{\ve, 1}^*\phi)\left(  \frac{\partial^{2}\left({\mathbf x}_{\ve, 1}^*\rho
_{\ve, 1}H_{\ve, 1}^{\prime}\right)  }{\partial y_{1}^{2}}+ F^{\prime\prime
}\left(  H\right) ({\mathbf x}_{\ve, 1}^*\rho_{\ve,1}H_{\ve,1}^{\prime})\right) \\
&  =-2\int_{\mathbb{R}}\frac{\partial{\tt x}^*_{\ve, 1}\phi}{\partial x_{1}}%
\frac{\partial\left( {\mathbf x}_{\ve, 1}^* \rho_{\ve,1}H_{\ve,1}^{\prime}\right)  }{\partial x_{1}}%
-\int_{\mathbb{R}}({\mathbf x}_{\ve, 1}^*\phi)\frac{\partial^{2}\left( {\mathbf x}_{\ve, 1}^* \rho_{\ve,1}H_{\ve,1}^{\prime
}\right)  }{\partial x_{1}^{2}}\\
& \quad +\int_{\mathbb{R}}({\mathbf x}_{\ve, 1}^*\phi)\left(  \frac{\partial^{2}({\mathbf x}_{\ve, 1}^*\rho_{\ve,1})}{\partial
y_{1}^{2}}({\mathbf x}_{\ve, 1}^*H_{\ve, 1}^{\prime})+2\frac{\partial({\mathbf x}_{\ve, 1}^*\rho_{\ve,1})}{\partial y_{1}}%
\frac{\partial ({\mathbf x}_{\ve, 1}^*H_{\ve,1}^{\prime})}{\partial y_{1}}\right)  .
\end{align*}
Due to the presence of the derivatives of ${\mathbf x}_{\ve, 1}^*\rho_{\ve, 1}$ with respect to $x_1$, $y_1$ and also the presence of  $H_{\ve,1}^{\prime}$ in each term, we now
obtain that
\begin{equation}
\|c_{\ve}\big({\mathbf x}_{\ve,1}^*\rho_{\ve,1}H_{\ve,1}^{\prime}\big)T_{21}\|_{{\mathcal C}^{0,\mu}_{\ve\tau}(\R)}=o\left(  \left\Vert
\phi\right\Vert _{{\mathcal C}^{2,\mu}_{\ve\tau}(\R^2)}\right)  . \label{i1}%
\end{equation}
On the other hand,
\begin{align*}
T_{22}&  =-\int_{\mathbb{R}}\left\{  \left(  \frac{1}{A_{1}}-1\right)
\partial_{x_{1}}^{2}({\mathbf x}_{\ve, 1}^*\phi)+\frac{1}{2}\frac{\partial_{y_{1}}A_{1}}{A_{1}%
}\partial_{y_{1}}({\mathbf x}_{\ve, 1}^*\phi)-\frac{1}{2}\frac{\partial_{x_{1}}A_{1}}{A_{1}%
^{2}}\partial_{x_{1}}({\mathbf x}_{\ve, 1}^*\phi)\right\}  ({\mathbf x}_{\ve, 1}^*\rho_{\ve,1}H_{\ve, 1}^{\prime})\\
& \quad  +\int_{\mathbb{R}}\big[({\mathbf x}_{\ve, 1}^* F^{\prime\prime}\left(  \bar{u}_\ve\right)\phi)
-F^{\prime\prime}\left(  H\right){\mathbf x}_{\ve, 1}^*   \phi\big] ({\mathbf x}_{\ve, 1}^* \rho_{\ve, 1}%
H_{\ve, 1}^{\prime})dy_{1}
%
%
\end{align*}
Let ${\tt x}=(t_1,z_1)$, ${\tt x}_2=(t_2,z_2)\in \mathcal{O}_1$ be given.  Observe that if\ ${\tt x}^*_{1}=\left(  s_{1},y_{1}\right)  ,{\tt x}^*_{2}=\left(  s_{2},y_{1}\right)  $ are the Fermi coordinates  with respect to $\Gamma_{\ve, 1},$  of ${\tt x}_1$, and ${\tt x}_2$ respectively, then
\[
c(|s_1-s_2|^2+|y_1-y_2|^2)^{\mu/2}\leq (|{t}_{1}-{t}_{2}|^2+|z_1-z_2|^2)^{\mu/2} \leq C(| s_{1}-s_{2}|^2+|y_1-y_2|^2)^{\mu/2} .
\]
Now,  we will make use of the fact that $1-\frac{1}{A_1}$, $\frac{\partial_{y_{1}}A_{1}}{A_{1}}%
,\frac{\partial_{x_{1}}A_{1}}{A_{1}^{2}}$, $\big[{\mathbf x}_{\ve, 1}^*F^{\prime\prime}\left(  \bar
{u}_\ve\right)  -F^{\prime\prime}\left(  H\right)\big] $ are small terms. Let us write for example:
\[
\left(  \frac{1}{A_{1}}-1\right)  \partial_{x_{1}}^{2}({\mathbf x}_{\ve, 1}^*\phi
\rho_{\ve, 1}H_{\ve, 1}^{\prime})\left(  x_{1},y_{1}\right)=\wt \phi(x_1, y_1) H'(y_1), \quad c_{\ve}\big({\mathbf x}_{\ve,1}^*\rho_{\ve,1}H_{\ve,1}^{\prime}\big)(x_1,y_1)=\wt H^{'}(x_1,y_1).
\]
Abusing slightly  the notation one has for instance
\begin{align*}
&  \sup_{\|{\tt x}_{1}-{\tt x}_{2}\| \leq 1}\frac{1}{\|{\tt x}_{1}-{\tt x}_{2}\|^\mu}
\left| \wt H'(s_1,y_1) \int_\R \wt\phi(s_1, \wt y)H'(\wt y)\,d\wt y-\wt H'(s_2, y_2)\int_\R \wt\phi(s_2, \wt y)H'(\wt y)\,d\wt y \right|
 \\
& \leq  
 \sup_{\|{\tt x}^*_{1}-{\tt x}^*_{2}\| \leq 1}\frac{1}{\|{\tt x}_{1}-{\tt x}_{2}\|^\mu}
\left| \wt H'(s_1,y_1) \int_\R \wt\phi(s_1, \wt y)H'(\wt y)\,d\wt y-\wt H'(s_2, y_2)\int_\R \wt\phi(s_2, \wt y)H'(\wt y)\,d\wt y \right|
%
%
&  =o\left(  \left\Vert \phi\right\Vert _{{\mathcal C}^{2,\mu}(\R^2)}\right)  ,
\end{align*}
which leads to
\[
\left\Vert c_{\ve}\big({\mathbf x}_{\ve,i}^*\rho_{\ve,i}H_{\ve,i}^{\prime}\big) \int_{\mathbb{R}}\left(  \frac{1}{A_{1}}-1\right)  \partial_{x_{1}%
}^{2}({\mathbf x}_{\ve, 1}^*\phi\rho_{\ve, 1}H_{\ve, 1}^{\prime})dy_{1}\right\Vert _{{\mathcal C}^{0,\mu}(\R^2)}=o\left(
\left\Vert \phi\right\Vert _{{\mathcal C}^{2,\mu}(\R^2)}\right)  .\
\]
Other terms apperaing in the definition of $T_{22}$ can be checked similarly
%
whence we obtain%
\[
\left\Vert c_{\ve}\big({\mathbf x}_{\ve,1}^*\rho_{\ve,1}H_{\ve,i}^{\prime}\big)T_{22}\right\Vert _{{\mathcal C}^{0,\mu}(\R^2)}=o\left(  \left\Vert
\phi\right\Vert _{{\mathcal C}^{0,\mu}(\R^2)}\right)  .
\]
This together with $\left(  \ref{i1}\right)  $ tells us%
\[
\left\Vert c_{\ve}\big({\mathbf x}_{\ve,i}^*\rho_{\ve,i}H_{\ve,i}^{\prime}\big)T_{2}\right\Vert _{{\mathcal C}^{0,\mu}(\R^2)}=o\left(  \left\Vert \phi\right\Vert
_{{\mathcal C}^{2,\mu}(\R^2)}\right)  .
\]
The estimate (\ref{ep}) follows from this in a  straightforward way.

\noindent
{\bf Step 2.} We claim that the weighted norm of the integral $\int_{\mathbb{R}}\left( {\mathbf x}_{\ve, 1}^* E\left(  \bar{u}_\ve\right) \rho_{\ve, 1}H_{\ve, 1}^{\prime} \right)dy_{1}$ is controlled by $f''_{\ve,1}$. 
To do this we will now check more closely the above integral using the
definition of $\bar{u}_\ve,$ these calculations are actually similar  as in the proof Lemma
\ref{Eu}. We see that one term appearing in the integral is
\[
\frac{1}{2}\int_{\mathbb{R}}\frac{\partial_{y_{1}}A_{1}}{A_{1}}%
({\mathbf x}_{\ve, 1}^* \rho_{\ve,1}H_{\ve,1}^{\prime2})dy_{1}.
\]
We will concentrate on this term since the ${\mathcal C}^{0,\mu}_{\ve\tau}(\R)$ norm of other terms can
 be estimated by ${\mathcal O}(\|h_\ve\| _{C^{2,\mu}_{\ve\tau}(\R)})
+C\|\exp\{-2\sqrt{2}|f_{\ve,1}|(1+\varDelta_\ve)\}\|_{\mathcal{C}_{\ve\tau}^{0,\mu}(\R)}$, as  we have seen in the proof of Lemma \ref{Eu}. Plugging in  the formula for $A_{1}$ into the
above integral, one gets
\begin{align*}
\frac{1}{2}\int_{\mathbb{R}}\frac{\partial_{y_{1}}A_{1}}{A_{1}}({\mathbf x}_{\ve, 1}^*\rho_{\ve,1}%
H_{\ve,1}^{\prime2})dy_{1}  &  =\int_{\mathbb{R}}\frac{1}{A_{1}}\left(  y_{1}%
\frac{\left(  f_{\ve, 1}^{\prime\prime}\left(  x_{1}\right)  \right)  ^{2}}{\left(
1+\left(  f_{\ve, 1}^{\prime}\left(  x_{1}\right)  \right)  ^{2}\right)  ^{2}}%
-\frac{f_{\ve, 1}^{\prime\prime}\left(  x_{1}\right)  }{\sqrt{1+\left(  f_{\ve, 1}^{\prime
}\left(  x_{1}\right)  \right)  ^{2}}}\right)  ({\mathbf x}_{\ve, 1}^*\rho_{\ve, 1}H_{\ve, 1}^{\prime2})\\
&  =-\frac{1}{c_\ve}f_{\ve, 1}^{\prime\prime}\left(  x_{1}\right)  +T_{4},
\end{align*}
where $T_{4}$ is a function such that
\[
\left\Vert T_{4}\right\Vert _{{\mathcal C}^{0,\mu}_{\ve\tau}(\R)}=o\left(  \left\Vert f_{\ve, 1}^{\prime\prime
}\right\Vert _{{\mathcal C}^{0,\mu}_{\ve\tau}(\R)}\right)  .
\]
Consequently,%
\begin{align*}
\left\Vert f_{\ve, 1}^{\prime\prime}\right\Vert _{{\mathcal C}^{0,\mu}_{\ve\tau}(\R)} &\leq C\left\Vert
\int_{\mathbb{R}}( {\mathbf x}_{\ve, 1}^* E\left(  \bar{u}_\ve\right)\rho
_{\ve, 1}H_{\ve, 1}^{\prime})dy_{1}\right\Vert _{{\mathcal C}^{0,\mu}_{\ve\tau}(\R)}\\
&\quad +{\mathcal O}\left(  \left\Vert
h_\ve\right\Vert _{{\mathcal C}^{2,\mu}_{\ve\tau}(\R)}\right)  +o\left(  \left\Vert f_{\ve, 1}^{\prime\prime
}\right\Vert _{{\mathcal C}^{0,\mu}_{\ve\tau}(\R)}\right)  +C\|\exp\{-2\sqrt{2}|f_{\ve,1}|(1+\varDelta_\ve)\}\|_{\mathcal{C}_{\ve\tau}^{0,\mu}(\R)}.
\end{align*}
We then could apply Lemma $\ref{fif}$ to get%
\begin{equation}
\left\Vert f_{\ve, 1}^{\prime\prime}\right\Vert _{{\mathcal C}^{0,\mu}(\R)}={\mathcal O}\left(\|\exp\{-2\sqrt{2}|f_{\ve,1}|(1+\varDelta_\ve)\}\|_{\mathcal{C}_{\ve\tau}^{0,\mu}(\R)}\right)  +{\mathcal O}\left(  \left\Vert h_\ve\right\Vert _{{\mathcal C}^{2,\mu}_{\ve\tau}(\R)%
}\right)  . \label{eq1}%
\end{equation}
This together with (\ref{hf}), and $\left(  \ref{f1}\right)  $ implies that
\begin{equation}
\left\Vert f_{\ve, 1}^{\prime\prime}\right\Vert _{{\mathcal C}^{0,\mu}_{\ve\tau}(\R)}\leq C\|\exp\{-2\sqrt{2}|f_{\ve,1}|(1+\varDelta_\ve)\}\|_{\mathcal{C}_{\ve\tau}^{0,\mu}(\R)}. \label{eq2}%
\end{equation}
As a consequence:
\[
\left\Vert \phi\right\Vert _{{\mathcal C}^{2,\mu}_{\ve\tau}(\R^2)}\leq C\|\exp\{-2\sqrt{2}|f_{\ve,1}|(1+\varDelta_\ve)\}\|_{\mathcal{C}_{\ve\tau}^{0,\mu}(\R)}.
\]

\end{proof}
\red{Above we were measuring the quantities involved in  the weighted H\"older norms $\mathcal{C}^{2,\mu}_{\ve\tau}(\R^2)$ and $\mathcal{C}^{2,\mu}_{\ve\tau}(\R)$ with $\tau>0$. However all proofs extend to the case $\tau=0$ i.e. when standard H\"older norms are measured.  For convenience we will state the relevant result 
\begin{corollary}\label{col holder}
The following estimates hold:
\[
\|\phi\|_{\mathcal{C}^{2,\mu}(\R^2)}=\mathcal{O}(\varUpsilon_\ve), \quad \|f''_{\ve, 1}\|_{\mathcal{C}^{0,\mu}(\R)}=\mathcal{O}(\varUpsilon_\ve),\quad \|h_\ve\|_{\mathcal{C}^{2,\mu}(\R)}=\mathcal{O}(\varUpsilon_\ve),
\]
where we have denoted:
\[
\varUpsilon_\ve=\|\exp\{-2\sqrt{2}|f_{\ve,1}|(1+\varDelta_\ve)\}\|_{\mathcal{C}^{0,\mu}(\R)}.
\]
\end{corollary}

Observe that with in particular we have: 
\begin{equation}
\label{517}
\|\phi\|_{\mathcal{C}^{2,\mu}(\R^2)}\leq Ce^{\,-2\sqrt{2}|f_{\ve, 1}(0)|(1+o(1))},\quad  \|f''_{\ve, 1}\|_{\mathcal{C}^{0,\mu}(\R)}\leq Ce^{\,-2\sqrt{2}|f_{\ve, 1}(0)|(1+o(1))}, \quad \|h_\ve\|_{\mathcal{C}^{2,\mu}(\R)}\leq Ce^{\,-2\sqrt{2}|f_{\ve, 1}(0)|(1+o(1))}. 
\end{equation}}
Now we will proceed to estimate the quantity
$e^{\,-\sqrt{2}|f_{\ve, 1}(0)|}$. To this end, we first need to obtain some exponential
decay estimate of $\phi$ along the $y$ axis away from $\Gamma_{\ve, 1}$. Note that
$E\left(  \bar{u}_\ve\right)  $  decays exponentially 
the direction transversal to the nodal line $\Gamma_{\ve, 1}$. Indeed, using (\ref{eu}) and the exponential decay of $H\pm 1$, $H'$ and $H''$ one can show:
\begin{equation}
\label{eu expo 1}
|E(\bar u_\ve)({\tt x})|\leq Ce^{\,-\sqrt{2}|f_{\ve, 1}(0)|} e^{\,-\sqrt{2}\mathrm{\dist}\,({\tt x}, \Gamma_{\ve, 1})}, \quad {\tt x}=(x, y), \quad y\leq 0.
\end{equation}

We write the equation for  $\phi$ in the form 
\[
-\Delta\phi+\big[F^{\prime\prime}\left(  \bar{u}_\ve\right)+\frac{P(\phi)}{\phi}\big]  \phi=E\left(  \bar
{u}_\ve\right).
\]
For any constant 
$0<\iota_0<\sqrt{2}$ there exists $r_0$ large and $\ve$ small, such that:
\[
F^{\prime\prime}\left(  \bar{u}_\ve({\tt x})\right)+\frac{P(\phi({\tt x}))}{\phi({\tt x})}\geq (\sqrt{2}-\iota_0)^2, \quad \mathrm{dist}\,({\tt x}, \Gamma_{\ve, 1})>r_0, \quad {\tt x}=(x,y), \quad y\leq 0.
\]
%
Using barriers we then obtain:
\begin{equation}
|\phi({\tt x})|\leq C_{r_0}e^{\,-\sqrt{2}|f_{\ve, 1}(0)|} e^{\,-(\sqrt{2}-\iota_0)\mathrm{dist}\, ({\tt x}, \Gamma_{\ve, 1})}, 
\label{expf}%
\end{equation}
in the lower half plane. Note that this estimate is in some sense precise only along the $y$ axis, since in reality we expect that $\phi({\tt x})\sim e^{\,-(2\sqrt{2}-\iota_0)\mathrm{dist}\, ({\tt x}, \Gamma_{\ve, 1})}$. This estimate can be  bootstrapped using elliptic estimates to get a similar estimate for the derivatives of the function $\phi$.

Let  us go back to the  Toda system (\ref{toda 2})--(\ref{toda 3}) and recall that by $q_{\ve, 1}(x)<0<q_{\ve, 2}(x)$ we have denoted the solution of this system whose slope at $\infty$ is $\ve$ (this means the tangent of the  asymptotic angle between the line $y=q_{\ve, 2}(x)$ and the $x$ axis in the first quadrant). We note that the curve $\widetilde \Gamma_{\ve, 1}=\{y=q_{\ve, 1}(x)\}$ is contained in the lower half plane. 

In what follows we  use $\alpha,\beta$ to denote general positive
constants independent of $\varepsilon.$

\begin{lemma}\label{est at 0}
There exists $\alpha_1>0$ such that $\left\vert f_{\ve, 1}\left(  0\right)  -q_{\ve, 1}\left(
0\right)  \right\vert \leq C\varepsilon^{\alpha_1}.$
\end{lemma}

\begin{proof}
The idea of the proof is to relate the asymptotic  behavior of $u$ along vertical straight lines,   as $\ve \to 0$, using the Hamiltonian identity:
\begin{equation}
\label{ham 23}
\int_\R \left\{  \frac{1}{2}u_{y}^{2}\left(  0,y\right)-\frac{1}{2}u_{x}^{2}\left(  0,y\right)
+F\left(  u\left(  0,y\right)  \right)  \right\}  dy=\int_\R \left\{  \frac{1}{2}u_{y}^{2}\left(  x,y\right)-\frac{1}{2}u_{x}^{2}\left(  x,y\right)
+F\left(  u\left(  x,y\right)  \right)  \right\}  dy, \quad \forall x,
\end{equation}
and in particular take $x\to \infty$ on the right hand side of (\ref{ham 23}).  Indeed, using the asymptotic behavior of a four ended solution it is not hard to show that:
\[
\lim_{x\to \infty}\int_\R \left\{  \frac{1}{2}u_{y}^{2}\left(  x,y\right)-\frac{1}{2}u_{x}^{2}\left(  x,y\right)
+F\left(  u\left(  x,y\right)  \right)  \right\}  dy=2{\mathbf{e}}_F \cos\theta(u), \quad {\mathbf e}_F=\int_\R\frac{1}{2}(H')^2+F(H), \quad \ve=\tan\theta(u).
\]
Since $u$ is an even function of $x$ we also have $u_x(0,y)=0$ and thus it follows from (\ref{ham 23}):
\[
\int_\R \left\{  \frac{1}{2}u_{y}^{2}\left(  0,y\right)
+F\left(  u\left(  0,y\right)  \right)  \right\}  dy= 2{\mathbf{e}}_F \cos\theta(u).
\]
We will now calculate the left hand side of the above identity. 

Recall that the heteroclinic solution has the following asymptotic behavior,
which can also be differentiated:
\[
H\left(  s\right)  =1-\mathbf{a}_{F}e^{-\sqrt{2}s}+O\left(  e^{-2\sqrt{2}%
s}\right)  ,\text{ as }s\rightarrow+\infty.
\]
Denote $t=f_{\ve, 1}\left(  0\right)  +h_{\ve}\left(  0\right)  .$ 
 Let  $\eta_1$ be the cut off function appearing in the definition of the approximate solution  (\ref{def hve}). 
 Along the $y$-axis  it holds $(x_1, y_1)=(0,y)$, where $(x_1,y_1)$ are the Fermi coordinates of $\Gamma_{\ve, 1}$ and then, abusing the notation slightly we can write
\begin{align*}
u\left(  0,y\right)  &=\underbrace{H\left(  y-t\right)  -H\left(  y+t\right)
-1+\phi\left(  0,y\right)}_{u_0(y)}\\
&\quad + \underbrace{(1-\eta_1(0,y))\Big[\frac{H\left(  y-t\right)}{|H\left(  y-t\right)}-H\left(  y-t\right)\Big]- 
(1-\eta_1(0,y))\Big[\frac{H\left(  y+t\right)}{|H\left(  y+t\right)|}-H\left(  y+t\right)\Big] }_{\psi(y)}.
\end{align*}
We observe that for all $\sigma>0$, 
\begin{equation}
|\psi(y)|\leq C^{\,-(\sqrt{2}-\sigma)|y|} e^{\,-\sigma \mathrm{d}_\ve(0)}, \label{est psi}
\end{equation}
and by the definition of $\mathrm{d}_\ve(0)$ 
\[
\mathrm{d}_\ve(0)\geq \frac{1}{2}\Big(\frac{\theta}{2{\|f_{\ve, 1}''\|_{{\mathcal C}^0(\R)}}}\Big)\geq Ce^{\,{2\sqrt{2}}|f_{\ve, 1}(0)|},
\]
where the last estimate follows form (\ref{517}). Then we find, taking $\sigma=\frac{3\sqrt{2}}{4}$ above, 
\[
 \int_{-\infty}^{0}\left\{  \frac{1}{2}u_{y}^{2}\left(  0,y\right)
+F\left(  u\left(  0,y\right)  \right)  \right\}  dy= \int_{-\infty}^{0}\left\{  \frac{1}{2}u_{0,y}^{2}\left(y\right)
+F\left(  u_0\left(y\right)  \right)  \right\}  dy+o(e^{\,-3|f_{\ve, 1}(0)|}).
\]
Now we calculate
\begin{align}
&  \int_{-\infty}^{0}\left\{  \frac{1}{2}u_{0,y}^{2}\left(y\right)
+F\left(  u_0\left(y\right)  \right)  \right\}  dy\nonumber\\
&  =\underbrace{\int_{-\infty}^{0}\left\{  \frac{1}{2}\left(  H^{\prime}\left(
y-t\right)  \right)  ^{2}+F\left(  H\left(  y-t\right)  \right)  \right\}
dy}_{I_1}\nonumber\\
&  +\underbrace{\int_{-\infty}^{0}\left\{  H^{\prime}\left(  y-t\right)  \left(
\partial_{y}\phi-H^{\prime}\left(  y+t\right)  \right)  +F^{\prime}\left(
H\left(  y-t\right)  \right)  \left(  \phi-H\left(  y+t\right)  -1\right)
\right\}  dy}_{I_2}\nonumber\\
&  +\underbrace{\frac{1}{2}\int_{-\infty}^{0}\left(  \partial_{y}\phi-H^{\prime}\left(
y+t\right)  \right)  ^{2}+F^{\prime\prime}\left(  H\left(  y-t\right)
\right)  \left(  \phi-H\left(  y+t\right)  -1\right)  ^{2}dy}_{I_3}\nonumber\\
&  +\mathcal{O}\left(  \int_{-\infty}^{0}\left(  \phi-H\left(  y+t\right)  -1\right)
^{3}dy\right)  \label{dis1}%
\end{align}
We note that 
\[
t=f_{\ve, 1}(0)+h_{\ve}(0)=f_{\ve, 1}(0)+{\mathcal O}(e^{\,-2\sqrt{2}|f_{\ve, 1}(0)|(1+o(1))})<0.
\]
The first term on the right hand side of $\left(  \ref{dis1}\right)  $ is
equal to
\begin{align*}
I_1&  =\int_{-\infty}^{-t}\left\{  \frac{1}{2}\left(  H^{\prime}\right)
^{2}+F\left(  H\right)  \right\}  dy\\
&  =\mathbf{e}_{F}-\int_{-t}^{+\infty}\left\{  \frac{1}{2}\left(  H^{\prime
}\right)  ^{2}+F\left(  H\right)  \right\}  dy\\
\  &  =\mathbf{e}_{F}-\int_{-t}^{+\infty}2\mathbf{a}_{F}^{2}e^{-2\sqrt{2}%
y}dy+{\mathcal O}\left(  e^{\,-4\sqrt{2}|t|}\right) \\
&  =\mathbf{e}_{F}-\frac{\sqrt{2}}{2}\mathbf{a}_{F}^{2}e^{\,-2\sqrt{2}|t|}+O\left(
e^{\,-3\sqrt{2}|t|}\right)  .
\end{align*}
We observe that, after an integration by parts,  the second term is equal to%
\[
I_2=H^{\prime}\left(  -t\right)  \left(  \phi\left(0\right)  -H\left(
t\right)  -1\right)  =-\sqrt{2}\mathbf{a}_{F}^{2}e^{\,-2\sqrt{2}|t|}+{\mathcal O}\left(
e^{\,-\frac{5}{2}\sqrt{2}|t|}\right)  .
\]
As to the last term, one has
\begin{align*}
I_3 &  ={\mathcal O}\left(  e^{\,-4|t|}\right)  +\frac{1}{2}\int_{-\infty}^{0}\left(  H^{\prime
}\left(  y+t\right)  \right)  ^{2}+F^{\prime\prime}\left(  H\left(
y-t\right)  \right)  \left(  H\left(  y+t\right)  -1\right)  ^{2}dy\\
&  ={\mathcal O}\left(  e^{\,-4|t|}\right)  +\frac{\sqrt{2}\mathbf{a}_{F}^{2}}{4}e^{\,-2\sqrt
{2}|t|}+\frac{\mathbf{a}_{F}^{2}}{2}\int_{-\infty}^{0}F^{\prime\prime}\left(
H\left(  y-t\right)  \right)  e^{\,-2\sqrt{2}|y+t|}dy\\
&  ={\mathcal O}\left(  e^{\,-4|t|}\right)  +\frac{\sqrt{2}\mathbf{a}_{F}^{2}}{2}e^{\,-2\sqrt
{2}|t|}.
\end{align*}
Therefore, we get that
\[
I_0:=\int_\R \left\{  \frac{1}{2}u_{y}^{2}\left(  0,y\right)
+F\left(  u\left(  0,y\right)  \right)  \right\}  dy=2\mathbf{e}_{F}-2\sqrt{2}\mathbf{a}_{F}^{2}e^{\,-2\sqrt{2}|f_{\ve, 1}\left(
0\right)  +h_\ve\left(  0\right)|  }+{\mathcal O}\left(  e^{\,-3|f_{\ve, 1}\left(  0\right)|
}\right)  .
\]
According to the Hamiltonian identity (\ref{ham 23}),
\begin{align*}
I_0 =2\mathbf{e}_{F}\cos\theta({u}).
\end{align*}
Now, let  $u_{\varepsilon}$ with
$\varepsilon=\tan\theta({u})$ be a solution constructed in \cite{MR2557944} whose nodal line in the lower half plane  is  given by the curve  $y=q_{\ve,1}(x)+r_{\ve,1}(\ve x)$, where $q_{\ve, 1}$ is the solution of the Toda system whose asymptotic angle at $\infty$ is $\ve$, and $r_{\ve,1}(x)$ satisfies, as we stated in section \ref{exists small eps}, with some $\alpha>0$
\[
\|r_{\ve,1}\|_{\mathcal C^{2, \mu}_{\tau } (\mathbb R) \oplus D} 
 \leq C \,  \varepsilon^{\alpha} 
\]
We recall that since we are working in the class of even function $|r_{\ve, 1}(x)|\leq C\ve ^\alpha$, which implies that  $r_{\ve, 1}$ is a bounded, small function. 
 Now, the   Hamiltonian identity (\ref{ham 23}) can be used for $u_\ve$ as well and  by a
computation we get
\[
2\mathbf{e}_{F}\cos\theta({u}_\ve)=2\mathbf{e}_{F}-2\sqrt{2}\mathbf{a}_{F}%
^{2}e^{\,-2\sqrt{2}|q_{\ve, 1}\left(  0\right)  +r_{\ve, 1}(0)|}+\mathcal{O}(e^{\,-3|q_{\ve, 1}\left(  0\right)  +r_{\ve, 1}(0)|})%
\]
where  $r_{\ve, 1}(0)={\mathcal O}\left(  \varepsilon^{\alpha}\right)  .$ Therefore,
\[
I_0=2\mathbf{e}_{F}-2\sqrt{2}\mathbf{a}_{F}^{2}e^{\,-2\sqrt{2}|q_{\ve,1}\left(  0\right)
+r_{\ve,1}(0)|}+\mathcal{O}(e^{\,-3|q_{\ve, 1}\left(  0\right)  +r_{\ve, 1}(0)|}).
\]
That is,
\[
e^{\,-2\sqrt{2}|f_{\ve, 1}\left(
0\right)  +h_\ve\left(  0\right)|  }+{\mathcal O}\left(  e^{\,-3|f_{\ve, 1}|\left(  0\right)|
}\right) =e^{\,-2\sqrt{2}|q_{\ve, 1}\left(  0\right)  +r_{\ve, 1}(0)|}+\mathcal{O}(e^{\,-3|q_{\ve, 1}\left(  0\right)  +r_{\ve, 1}(0)|})
%
\]
This yields%
\[
f_{\ve, 1}\left(  0\right)  +h_\ve \left(  0\right)  +{\mathcal O}\left(  e^{\,-\left(  3-2\sqrt
{2}\right)| f_{\ve, 1}\left(  0\right)  +h_\ve \left(  0\right)|  }\right)
=q_{\ve,1}\left(  0\right)  +{\mathcal O}\left(  \varepsilon^{\alpha}\right)  .
\]
Since $q_{\ve, 1}\left(  0\right)  -\frac{\sqrt{2}}{2}\log\varepsilon={\mathcal O}\left(  1\right)
,$ we get%
\[
f_{\ve, 1}\left(  0\right)  +h_{\ve}\left(  0\right)  =\frac{\sqrt{2}}{2}\log\varepsilon
+{\mathcal O}\left(  1\right)  ,
\]
which leads to
\[
f_{\ve, 1}\left(  0\right)  +h_{\ve}\left(  0\right)  -q_{\ve, 1}\left(  0\right)  ={\mathcal O}\left(
\varepsilon^{\alpha}\right),
\]
as claimed. This ends the proof.
\end{proof}
%
%
%
%

\red{Let us again summarize in the estimates we have obtained so far. Using the above Lemma and Corollary \ref{col holder} we get: for any $\varsigma>0$ there exists a constant $C_\varsigma$ such that 
\begin{equation}
\label{soft est}
\|\phi\|_{\mathcal{C}^{2,\mu}(\R^2)}\leq C_\varsigma\ve^{2-\varsigma}, \quad \|h_{\ve}\|_{\mathcal{C}^{2,\mu}(\R)}\leq C_\varsigma\ve^{2-\varsigma}, \quad  \|f''_{1,\ve}\|_{\mathcal{C}^{0,\mu}(\R)}\leq C_\varsigma \ve^{2-\varsigma}.
\end{equation}}
Now we are in position to prove Proposition \ref{prop gamma}. As we will see the proof of Proposition \ref{estim hat phi}  will be obtained as an intermediate step

\begin{proof}
Our goal is to show  estimate  (\ref{first estimate}) and this will be done in few steps. For brevity let  us denote $p_{\ve, 1}=f_{\ve, 1}+h_{\ve}$, so that $\chi_{\ve, 1}=p_{\ve, 1}-q_{\ve, 1}$. 

\noindent{\bf Step 1.}
We first claim that if $I_a:=\left(  -a,a\right)  $ is an interval where
\begin{equation}
 |p_{\ve, 1}\left(  x\right)|<2\left\vert \log\varepsilon
\right\vert ,\quad | p_{\ve, 1}''\left(  x\right) | <C_\varsigma\varepsilon^{2-\varsigma},\quad x\in I_a
\label{q}%
\end{equation}
(such an interval clearly exists, which can be seen by combining (\ref{517}) and Lemma \ref{est at 0}),
then $p_{\ve,1}$ satisfies a non homogeneous  Toda equation in $I_a,$ i.e.,
\begin{equation}
\label{t23}
\bar c_0p_{\ve, 1}^{\prime\prime}\left(  x\right)  =-2e^{\,2\sqrt
{2}p_{\ve, 1}\left(  x\right)  }+\lambda_{\ve, 1}\left(  x\right)  ,\quad x\in I_a,
\end{equation}
where
\begin{equation}
\label{lambda 1}
\left\Vert \lambda_{\ve, 1}\right\Vert _{C^{0,\mu}\left(  I_a\right)  }\leq
C\varepsilon^{2+\beta_{1}},
\end{equation}
for some constant $\beta_{1}>0.$  Note that the solution $q_{\ve, 1}$ of the Toda system does satisfy $\left(
\ref{q}\right)  $ in the interval $\left(  \frac{-|\log\varepsilon|}{\varepsilon
},\frac{|\log\varepsilon|}{\varepsilon}\right)$, and in fact we will show that $a=\mathcal{O}(\frac{|\log\ve|}{\ve})$.  Thus we loss no generality assuming a priori $a<3\frac{|\log{\ve}|}{\ve}$. \red{If this  is the case then, from (\ref{lambda 1}) it follows that 
\[
\left\Vert \lambda_{\ve, 1}\right\Vert _{C^{0,\mu}_{\ve\tau}\left(  I_a\right)  }\leq
C\varepsilon^{2+\beta_{2}},
\]
with some $\beta_2>0$, if $\tau$ is taken small. We also note that because of symmetry we have:
\[
|p'_{\ve, 1}(x)|\leq C_\varsigma\ve^{2-\varsigma}a, \quad x\in I_a.
\]

}

To begin the proof of the claim we observe that for ${\tt x}=(x,y)$ such that  $\pi_{\ve, 1}^x({\tt x})=x_1\in I_a,$ we have%
\begin{align}
\label{alpha 23}
\begin{aligned}
\left\vert x_{1}-x_{2}\right\vert  &  \leq C\varepsilon^{\alpha},\\
\left\vert y_{1}-y_{2}+2f_{\ve, 1}\left(  x_{1}\right)  \right\vert  &  \leq
C\varepsilon^{\alpha},
\end{aligned}
\end{align}
with some $\alpha>0$, where  $(x_i, y_i)$ denote  the Fermi coordinate of ${\tt x}$  around
$\Gamma_{\ve, i}$ for $x\in I_a$ (whenever these coordinates are defined). 
 Using this,  $\left(  \ref{ep}\right)$, and (\ref{soft est}) we can calculate $\int_{\mathbb{R}%
}\mathbf{x}_{1}^{\ast}\left(  E\left(  \bar{u}\right)  \rho_{1}H_{1}^{\prime
}\right)  dy_{1}$ as Lemma $\ref{fi}$ to get: 
\begin{equation}
\label{t24}
\begin{aligned}
\bar c_0\left(  1+{\mathcal O}_{{\mathcal C}^{0,\mu}(I_a)}\left(  \varepsilon^{\alpha}\right)  \right)  f_{\ve, 1}^{\prime\prime
}\left(  x_1\right)  +\left(  1+{\mathcal O}_{{\mathcal C}^{0,\mu}(I_a)}\left(  \varepsilon^{\alpha}\right)  \right)
h_\ve^{\prime\prime}\left(  x_1\right)  &=-2e^{2\sqrt
{2}\left(  f_{\ve, 1}+h_\ve\right)  \left(  x_1\right)  }\left(  1+{\mathcal O}_{{\mathcal C}^{0,\mu}(I_a)}\left(  \varepsilon
^{\alpha}\right)  \right) \\
&\quad +{\mathcal O}_{{\mathcal C}^{0,\mu}(I_a)}\left(  \varepsilon^{2+\alpha}\right),
\end{aligned}
\end{equation}
with some constant $\alpha>0$.
(For details we refer the reader to \cite{MR2557944}, where similar calculations can be found). This
relation then leads to the claim (\ref{t23}).

\noindent{\bf Step 2.}
Let us set $\hat \chi_{\ve, 1}\left(  x\right)  =p_{\ve, 1}\left(  x\right)  -q_{\ve, 1}\left(  x\right)$. Note that, possibly taking the interval  $I_a$ smaller we may assume that $\hat\chi_{\ve,1}$ and $\hat\chi_{\ve, 1}''$ are small in this interval. This follows by Lemma \ref{est at 0}, and  (\ref{q}), which holds for $q_{\ve, 1}''$ as well.  Now we will show the following local version of (\ref{first estimate}):
\begin{equation}
\label{first estimate A}
\begin{aligned}
\left\Vert\hat \chi_{\ve,1}\right\Vert _{{\mathcal C}^{0,\mu}_{\ve \tau}\left(I_a\right)}  &  \leq C\varepsilon^{\alpha},\\
\left\Vert \hat\chi_{\ve,1}^{\prime}\right\Vert _{{\mathcal C}^{0,\mu}_{\ve\tau}\left(I_a\right)}  &  \leq C\varepsilon^{1+\alpha
},\\
\left\Vert \hat\chi_{\ve,1}^{\prime\prime}\right\Vert _{{\mathcal C}^{0,\mu}_{\ve \tau}\left(I_a\right)}  &  \leq C\varepsilon
^{2+\alpha}.
\end{aligned}
\end{equation}
  
As long as  $\hat\chi_{\ve, 1}\left(  x\right)$ is small   we
get,
\begin{align*}
\bar c_0\hat\chi^{\prime\prime}_{\ve, 1}  &  =-4\sqrt{2}e^{2\sqrt{2}q_{\ve, 1}}%
\hat\chi_{\ve, 1}+O\left(  \chi_{\ve, 1}^{2}\right)  e^{-2\sqrt{2}q_{\ve, 1}}+\lambda_{\ve, 1}\left(  x\right) \\
&  =-4\sqrt{2}e^{2\sqrt{2}q_{\ve, 1}}\hat\chi_{\ve, 1}+\lambda_{\ve, 2}\left(
x\right),
\end{align*}
where
\begin{equation}
\label{lambda 2}
\lambda_{\ve, 2}=\lambda_{\ve, 1}+{\mathcal O}(\chi_{\ve, 1}^2)e^{-2\sqrt{2}q_{\ve, 1}}.
\end{equation}
Let $\{\varsigma_{\ve, i}$, $i=1,2\}$,  be a fundamental set of the linearized Toda equation:
\[
\bar c_0\varsigma_\ve^{\prime\prime}\left(  x\right)  =-4\sqrt{2}%
e^{\,2\sqrt{2}q_{\ve,1}\left(  x\right)  }\varsigma_{\ve}\left(  x\right)
,\]
with $\varsigma_{\ve, 1}$ odd, $\varsigma_{\ve, 2}$ even, $\varsigma_{\ve, 1}\left(
0\right)  =1$ and $\varsigma_{\ve 2}^{\prime}\left(  0\right)  =\varepsilon$ and
$\left\vert \varsigma_{\ve, i}^{\prime}\right\vert \leq C\varepsilon.$ Note that
although $\varsigma_{\ve, 1}$ and $\varsigma_{\ve,2}$ can be explicitly expressed in terms of $q_{\ve, 1}$ and its derivatives
their exact formulas are not needed here. What we should keep in mind is that  $\varsigma_{\ve, 1}$ is bounded and that $\varsigma_{\ve, 2}(x)\sim \ve |x|$. The variation of parameters formula
yields%
\begin{align*}
\hat\chi_{\ve, 1}\left(  x\right)   &  =\frac{\varsigma_{\ve, 2}\left(  x\right)  }{\varepsilon
}\int_{0}^{x}\varsigma_{\ve, 1}\left(  s\right)  \lambda_{\ve, 2}\left(  s\right)
ds-\frac{\varsigma_{\ve, 1}\left(  x\right)  }{\varepsilon}\int_{0}^{x}%
\varsigma_{\ve, 2}\left(  s\right)  \lambda_{\ve, 2}\left(  s\right)  ds\\
&  +\left(  p_{\ve, 1}\left(  0\right)  -q_{\ve, 1}\left(  0\right)  \right)  \varsigma
_{\ve, 1}\left(  x\right)
\end{align*}
and
\begin{align*}
\hat\chi_{\ve, 1}^{\prime}\left(  x\right)   &  =\frac{\varsigma_{\ve, 2}^{\prime}\left(
x\right)  }{\varepsilon}\int_{0}^{x}\varsigma_{\ve, 1}\left(  s\right)  \lambda
_{\ve, 2}\left(  s\right)  ds-\frac{\varsigma_{\ve, 1}^{\prime}\left(  x\right)
}{\varepsilon}\int_{0}^{x}\varsigma_{\ve, 2}\left(  s\right)  \lambda_{\ve, 2}\left(
s\right)  ds\\
&  +\left(  p_{\ve, 1}\left(  0\right)  -q_{\ve, 1}\left(  0\right)  \right)  \varsigma
_{\ve, 1}^{\prime}\left(  x\right)  .
\end{align*}
Then by Lemma \ref{est at 0}, estimate (\ref{lambda 1}),  
and  the estimate for $\lambda_{\ve, 2}$ in (\ref{lambda 2}),  one
has in the interval $I_a$:
\begin{equation}
\label{r1}
\begin{aligned}
\left\vert\hat \chi_{\ve, 1}\left(  x\right)  \right\vert  &  \leq\frac{1}{\varepsilon%
}\left\vert \varsigma_{\ve, 2}\left(  x\right)  \right\vert \int_{0}^{
x}\left\vert \varsigma_{\ve, 1}\left({s}\right)  \lambda
_{\ve, 2}\left({s}\right)  \right\vert ds
 +\frac{1}{\varepsilon}\left\vert \varsigma_{\ve,1}\left(  x\right)
\right\vert \int_{0}^{x}\left\vert \varsigma_{\ve, 2}\left(
{s}\right)  \lambda_{\ve, 2}\left({s}\right)
\right\vert ds\\
&\quad + |p_{\ve, 1}(0)  -q_{\ve, 1}(0)||\varsigma_{\ve, 1}(x)|\\
&  \leq C|a|^2(\|\lambda_{\ve,1}\|_{{\mathcal C}^{0,\mu}(I_a)}+\ve^2{\mathcal O}(\|\chi_{\ve,1}\|^2_{{\mathcal C}^{0,\mu}(I_a)})+C\varepsilon^{\alpha}
\\
&  < C|a|^2\ve^{2+\beta_1}+C\varepsilon^{\alpha_{1}}+C|a|^2\ve^2{\mathcal O}(\|\chi_{\ve,1}\|^2_{{\mathcal C}^{0,\mu}(I_a)}).
\end{aligned}
\end{equation}
From this the ${\mathcal C}^0(I_a)$ estimate for $\hat\chi_{\ve,1}$ follows immediately. It is also evident that we can take $a=\frac{|\log\ve|}{\ve}$.  Additionally, in this same interval, due to the fact that $|\varsigma'_{\ve, i}|\leq C\ve$, $i=1,2$, we 
have,
\begin{equation}
|\hat\chi_{\ve, 1}^{\prime}\left(  x\right)|  \leq 
C|a|^2\ve^{3+\beta_1}+C\varepsilon^{1+\alpha_{1}}+C|a|^2\ve^3{\mathcal O}(\|\chi_{\ve,1}\|^2_{{\mathcal C}^{0,\mu}(I_a)})
 \label{r2}%
\end{equation}
from which we obtain the ${\mathcal C}^1(I_a)$ estimate for $\hat\chi_{\ve, 1}$, with $a=|\log\ve|/\ve$. 
The argument leading finally to the full statement   (\ref{first estimate A}) is clear. 

\noindent{\bf Step 3.}
Next, we will prove that $\mathrm{dist}(\Gamma_{\ve, 1}, \widetilde\Gamma_{\ve, 1})\to 0$ as $\ve \to 0$.\red{(We recall here that $\wt\Gamma_{\ve, 1}=\{y=q_{1,\ve}(x)\}$).} For this we should  consider the function $p_{\ve, 1}$ outside the interval $I_{|\log\ve|/\ve}$ 

%
Since we have proven (\ref{lambda 1}) already, it suffices to show that 
\[
\mathrm{dist}\, \left(  \Gamma_{\ve, 1}\cap \{|x|>\frac{\left\vert \log\varepsilon\right\vert }{\varepsilon}\},\widetilde\Gamma_{1, \ve}\cap \{|x|>\frac{\left\vert \log\varepsilon\right\vert }{\varepsilon}\}\right)  \rightarrow 0.
\]

Let the asymptotic line of $u$ in the fourth quadrant to be $y=-\varepsilon
x-\mathcal{A}_\ve.$ Suppose $\left[  a_{\ve},+\infty\right)  $ is a maximal  subinterval of
$[\frac{\left\vert \log\varepsilon\right\vert }{\varepsilon},+\infty)$ where
$\left\vert f_{\ve, 1}\left(  x\right)  +\left(  \varepsilon x+\mathcal{A}_\ve\right)
\right\vert \leq1.$ This interval is not empty and possibly $a_\ve\geq |\log\varepsilon|/\ve$. We wish to show that in fact $a_\ve= |\log\varepsilon|/\ve$. To argue by contradiction let us assume that there exists a $\delta>0$, independent on $\ve$ and  such that 
\[
\sup_{x>a_{\ve}}|f\left(  x\right)  +\left(  \varepsilon
x+\mathcal{A}_\ve\right) |>\delta.
\]
Next,  we let $x_\ve\in \left[ a_{\ve},+\infty\right)  $ be such that 
\[
\left\vert f\left(  x_\ve\right)  +\left(  \varepsilon
x_\ve+\mathcal{A}_\ve\right)  \right\vert=\sup_{x>a_{\ve}}|f\left(  x\right)  +\left(  \varepsilon
x+\mathcal{A}_\ve\right)| 
\]
since 
\[
| f_{\ve, 1}\left(  x\right)  +\left(  \varepsilon x+\mathcal{A}_\ve\right)| \rightarrow0,\text{ as } x
\rightarrow \infty.
\]
therefore $x_\ve<\infty$ is well defined.

Consider the domain
\[
\Omega_{L}:=\left\{  \left(  x,y\right)\mid y<0,\quad x>x_{\ve},\quad y>\frac{x}{\varepsilon
}-L\right\}  .
\]
Here $L>\ve x_\ve$ is large and indeed we will finally let it go to $+\infty.$ We 
use the balancing formula in this domain and with  the vector field $X:=\left(
f_{\ve, 1}\left(  x_{\ve}\right)  -y,x-x_{\ve}\right)  $. This formula tells us
that
\[
\int_{\partial\Omega_{L}}\left\{  \left(  \frac{1}{2}\left\vert \nabla
u\right\vert ^{2}+F\left(  u\right)  \right)  X-X\left(  u\right)  \nabla
u\right\}  \cdot \nu dS=0.
\]
Let us estimate the relevant boundary integrals. First,
\begin{align*}
&  \int_{\partial\Omega_{L}\cap\left\{  y=0\right\}  }\left\{  \left(
\frac{1}{2}\left\vert \nabla u\right\vert ^{2}+F\left(  u\right)  \right)
X-X\left(  u\right)  \nabla u\right\}  \cdot \nu dS\\
&  =\int_{x_{\ve}}^{L/\ve}\left(  \frac{1}{2}u_{x}^{2}+F\left(  u\right)
\right)  \left(  x-x_{\ve}\right)  dx\to \int_{x_{\ve}}^{\infty}\left(  \frac{1}{2}u_{x}^{2}+F\left(  u\right)
\right)  \left(  x-x_{\ve}\right)  dx, \quad \mbox{as}\ L\to \infty.
\end{align*}
To estimate this integral let us recall that, by symmetry, we have for ${\tt x}=(x,y)$  $y\leq 0$, with some $\kappa>0$
\[
|u({\tt x})^2-1| +|\nabla u({\tt x})|\leq C e^{\,-\kappa \mathrm{dist}\,(\Gamma_{\ve, 1}, {\tt x})}.
\]
Since for $x\geq a_\ve$ the distance between $\Gamma_{\ve, 1}$ and the line $\ell_\ve=\{y=-(\ve x+{\mathcal A}_\ve)\}$ is bounded therefore we have as well
\[
|u({\tt x})^2-1| +|\nabla u({\tt x})|\leq C e^{\,-\kappa \mathrm{dist}\,(\ell_{\ve}, {\tt x})}, \quad {\tt x}=(x,y), \quad x\geq a_\ve.
\]
Now using this, and the fact that
\[
|\varepsilon x_{\ve}+\mathcal{A}_\ve|\geq |f_{\ve, 1}(0)|-1\geq C|\log\ve|,
\]
we deduce that as $\varepsilon\rightarrow 0,$%
\[
\int_{\partial\Omega_{L}\cap\left\{  y=0\right\}  }\left\{  \left(  \frac
{1}{2}\left\vert \nabla u\right\vert ^{2}+F\left(  u\right)  \right)
X-X\left(  u\right)  \nabla u\right\}  \cdot \nu dS\rightarrow 0.
\]
On the other hand, using the asymptotic behavior of $u$ in the lower half plane, we get:
\[
u=\bar u_\ve+o(1)e^{\,-\kappa \mathrm{dist}\,(\Gamma_{\ve, 1}, {\tt x})},\quad ({\mathbf x}_{\ve, 1}^*\bar u_\ve)(x_1,y_1)=H(y_1-h_\ve(x_1))\eta_1(x_1, y_1)+o(1) e^{\,-\sqrt{2} |y_1|}, 
\]
where $(x_1, y_1)$ are the Fermi coordinates of the point ${\tt x}$. Since on the line $\{x=x_\ve\}$ we have $X=(f_{\ve, 1}(x_\ve)-y, 0)$ therefore:

\[
\int_{\partial\Omega_{L}\cap\left\{  x=x_{\ve}\right\}  }\left\{  \left(
\frac{1}{2}\left\vert \nabla u\right\vert ^{2}+F\left(  u\right)  \right)
X-X\left(  u\right)  \nabla u\right\}  \cdot \nu dS=o\left(  1\right).
\]
Finally, we compute
\begin{align*}
  \left\vert \int_{\partial\Omega_{L}\cap\left\{  y=\frac{x}{\varepsilon
}-L\right\}  }\left\{  \left(  \frac{1}{2}\left\vert \nabla u\right\vert
^{2}+F\left(  u\right)  \right)  X-X\left(  u\right)  \nabla u\right\}  \cdot
\nu dS\right\vert 
  =\frac{\left\vert f_{\ve, 1}\left(  x_{\ve}\right)  +\varepsilon x_{\ve}+\mathcal{A}_\ve%
\right\vert }{\sqrt{1+\varepsilon^{2}}}+o\left(  1\right)  .
\end{align*}
Collecting all these estimates, we conclude
\[
\left\vert f_{\ve, 1}\left(  x_{\ve}\right)  +\varepsilon x_{\ve}+\mathcal{A}_\ve\right\vert
=o\left(  1\right)  .
\]
But then we must have $a_\ve =\frac{|\log \ve|}{\ve}$, and consequently, in the interval $[\frac{\left\vert \log\varepsilon\right\vert
}{\varepsilon},+\infty),$%
\[
\left\vert f_{\ve, 1}\left(  x\right)  +\varepsilon x+\mathcal{A}_\ve\right\vert =o\left(
1\right)  .
\]
This implies that outside this interval, $\Gamma_{\ve, 1}$ is close to a straight line, 
which combined with the estimates  (\ref{first estimate})  yields the desired result. Indeed, at $x^*=\frac{|\log\ve|}{\ve}$ we have $q_{\ve, 1}(x^*)=f_{\ve, 1}(x^*)+o(1)$, $q'_{\ve, 1}(x^*)=f'_{\ve, 1}(x^*)+o(1)$ and $q_{\ve, 1}(x)=-\ve|x|-\mathcal{\tilde A}_\ve+o(1)$, $x>x^*$. But then $\mathcal{A}_\ve=\mathcal{\tilde A}_\ve+o(1)$. This ends the proof of Step 3.

\noindent\noindent{\bf Step 4.}
At this point we can use what  we have just proven in Step 2 and Step 3 to get
\[
f_{\ve, 1}(x)=\frac{\sqrt{2}}{2}\log\frac{1}{\ve}+\ve|x|+\mathcal{O}_{\mathcal{C}^0(\R)}(1), \quad |x|\gg 1. 
\]
As a consequence
\begin{equation}
\label{ppp 1}
\|\exp\{-2\sqrt{2}|f_{\ve,1}|(1+\varDelta_\ve)\}\|_{\mathcal{C}_{\ve\tau}^{0,\mu}(\R)}\leq C\ve^2,
\end{equation}
from which we find:
\[
\|\phi\|_{\mathcal{C}_{\ve\tau}^{0,\mu}(\R^2)}\leq C\ve^2, \quad\|f''_{\ve, 1}\|_{\mathcal{C}_{\ve\tau}^{0,\mu}(\R)}\leq C\ve^2, \quad \|h_\ve\|_{\mathcal{C}_{\ve\tau}^{0,\mu}(\R)}\leq C\ve^2.
\]
Using this information to  calculate $\int_{\mathbb{R}%
}\mathbf{x}_{1}^{\ast}\left(  E\left(  \bar{u}\right)  \rho_{1}H_{1}^{\prime
}\right)  dy_{1}$ as Lemma $\ref{fi}$ we obtain: 
\begin{equation}
\label{t24}
\begin{aligned}
\bar c_0\left(  1+{\mathcal O}_{{\mathcal C}^{0,\mu}_{\ve\tau}(\R)}\left(  \varepsilon^{\alpha}\right)  \right)  f_{\ve, 1}^{\prime\prime
}\left(  x_1\right)  +\left(  1+{\mathcal O}_{{\mathcal C}^{0,\mu}_{\ve\tau}(\R)}\left(  \varepsilon^{\alpha}\right)  \right)
h_\ve^{\prime\prime}\left(  x_1\right) & =-2e^{2\sqrt
{2}\left(  f_{\ve, 1}+h_\ve\right)  \left(  x_1\right)  }\left(  1+{\mathcal O}_{{\mathcal C}^{0,\mu}_{\ve\tau}(\R)}\left(  \varepsilon
^{\alpha}\right)  \right) \\
&\quad  +{\mathcal O}_{{\mathcal C}^{0,\mu}_{\ve\tau}(\R)}\left(  \varepsilon^{2+\alpha}\right),
\end{aligned}
\end{equation}
with some constant $\alpha>0$. At this point we repeat the argument in Step 2 to obtain (\ref{first estimate}). This ends the proof of Proposition \ref{prop gamma}. Now, the assertion of Proposition \ref{estim hat phi} is contained in (\ref{ppp 1}).
\end{proof}

\red{
So far we have proven estimates for $f_{\ve, i}$, $h_\ve$ and  $\phi$, respectively   in the weighed $\mathcal{C}^{2,\mu}(\R)$ and  $\mathcal{C}^{2,\mu}(\R^2)$norms. However these functions are a priori more regular and  one can bootstrap the argument argument above differentiating the relevant relations. The details are tedious but standard. As a result one can estimate $3$rd derivatives of these functions. For our purpose it is important to estimate the derivatives in $x_i$.  We summarize this in the following:  
\begin{corollary}\label{higher regularity}
The following estimates are satisfied:
\[
\|f_{\ve,i}'''\|_{\mathcal{C}^0(\R)}+\|h'''_\ve\|_{\mathcal{C}^0(\R)}\leq C\ve^{2+\alpha}.
\]
\end{corollary}
}

\section{Uniqueness of solutions with almost parallel nodal lines}\label{sec teo uniqueness}

\subsection{Parametrization of the family of solutions of (\ref{AC})  by the trajectories of the Toda system} \label{toda aprox} 

Let us consider the  curve $\{y=q_{\ve,i}(x)\}=\widetilde \Gamma_{\ve, i}$. When $i=1$ it is contained in the lower half plane, when $i=2$ is is contained in the upper half plane and we have actually $q_{\ve, 1}(x)=-q_{\ve, 2}(x)$. With this curve we will associate the Fermi coordinates $(\wt {x}_{i}, \wt {y}_{i})$:
\[
{\tt x}=(\wt x_i, q_{\ve, i}(\wt {x_i}))+\wt y_{i}\wt n_{\ve, i}(\wt x_i), \quad \wt n_{\ve, i}(x)=\frac{(-1)^i(q_{\ve, i}'(x), -1)}{\sqrt{1+q_{\ve, i}'(x)^2}}.
\] 
The change of variables $(\widetilde x_i, \widetilde y_i)\mapsto {\tt x }=(x,y)$ is a diffeomorphism in a neighborhood $\widetilde {\mathcal O}_{i}$ of $\widetilde\Gamma_{\ve, i}$.  We denote this  diffeomorphism by $\wt {\tt x}_{\ve, i}$, so that 
\[
\wt{\tt x}_{\ve, i}(\wt x_i, \wt y_i)={\tt x}\in \wt{\mathcal O}_{i}.
\]
For any function $w\colon \wt{\mathcal O}_{i}\to \R$ by $\wt{\tt x}^*_{\ve,i} w$ we denote its pullback by $\wt{\tt x}_{\ve, i}$:
 \[
(\wt{\tt x}^*_{\ve,i} w)(\wt x_i, \wt y_i)=(w\circ \wt{\tt x}_{\ve, i}) (\wt x_i, \wt y_i).
\] 

To describe more precisely the neighborhood $\wt {\mathcal O}_{i}$ we define the projection function $\wt \pi_{\ve, i}\colon \R^2\to \wt\Gamma_{\ve, i}$:
\[
\wt \pi_{\ve,1}({\tt x})=(\wt \pi_{\ve, 1}^x({\tt x}),\wt \pi_{\ve, 1}^y({\tt x)}):=(\wt x_i, q_{\ve, i}(\wt x_i)), \quad \mbox{whenever}\  {\tt x}=(\wt x_i, q_{\ve, i}(\wt x_1))+ \wt y_i n_{\ve, i}(\wt x_i), \quad \wt y_i=\mathrm{dist}\,({\tt x}, \wt \Gamma_{\ve, i}).
\] 

Using basic  properties (linear growth, scaling) of the trajectories of the solutions of the Toda system it is not hard to show that there exits a positive constant $c$ such that  we can define 
\[
 \widetilde{\mathcal O}_{i}:=\{{\tt x}\mid \|\wt \pi_{\ve, i}({\tt x})-{\tt x}\| \leq c\log \ve^{-1}+ c\ve \log (\cosh \wt\pi_{\ve, i}^x({\tt x}))\},
\]
 c.f.  \cite{MR2557944}. In particular, we can chose $c$ large, so that  $\widetilde\Gamma_{\ve, i}\in \widetilde {\mathcal O}_{j}$, $i, j=1,2$.   For future reference we set
 \[
 \wt{\mathrm d}_\ve(x)=c\log \ve^{-1}+ c\ve \log (\cosh {x}).
 \]
 
 \red{Using the results of the previous section we can easily derive relations between the Fermi coordinates of $\wt\Gamma_{\ve, i}$ and those of $\Gamma_{\ve, i}$. Indeed, using the explicit formulas it is not hard to prove that there exist constants $\wt b_\ve=\mathcal{O}(\ve^{1+\alpha})$, $\wt j_\ve=\mathcal{O}(j_\ve)$ such that  we have:
 \begin{equation}
 x_i-\wt x_i=\wt b_\ve+\wt y_i\mathcal{O}_{\mathcal{C}^{2,\mu}_{\ve\tau}(\R)}(\ve^\alpha)+\mathcal{O}_{\mathcal{C}^{2,\mu}_{\ve\tau}(\R)}(\ve^\alpha), \quad  y_i-\wt y_i=\wt j_\ve +\wt y_i\mathcal{O}_{\mathcal{C}^{2,\mu}_{\ve\tau}(\R)}(\ve^\alpha)+\mathcal{O}_{\mathcal{C}^{2,\mu}_{\ve\tau}(\R)}(\ve^\alpha).
 \label{rel var}
 \end{equation}
In fact these relations can be differentiated and the derivates "gain" a power of $\ve$. To see this we define the following diffeomorphism $\Phi_{\ve,i}=\wt{\tt x}_{\ve, i}\circ{\tt x}_{\ve, i}^{-1}$. Then we have:
\begin{equation}
\label{62}
D\Phi_{\ve, i}={\mathrm{Id}}_{2\times 2}+\mathcal{O}_{\mathcal{C}^{1,\mu}(\R)}(\ve^{1+\alpha}), \quad D^2\Phi_{\ve, i}=\mathcal{O}_{\mathcal{C}^{0,\mu}(\R)}(\ve^{2+\alpha}).
\end{equation}
This is proven by directly exploring the relations between the Fermi coordinates of $\Gamma_{\ve, i}$ and of $\wt\Gamma_{\ve, i}$.  Finally, we need to have some information about the relation between $(\wt x_1, \wt y_1)$ and $(\wt x_2, \wt y_2)$.  Again, it is not hard to see that we have:
\begin{equation}
\label{63}
|\wt x_1-\wt x_2|\leq C\ve|\wt y_1|+o(1), \quad |\wt y_1-\wt y_2+2 q_{\ve, 1}(0)|\leq C\ve^2 |\wt y_1|+o(1).
\end{equation}
 }
With these preparations, we would like to write locally any solution  $u$, with $\tan\theta(u)=\ve$ small, in  Fermi
coordinates with respect to $\widetilde\Gamma_{\ve, i}$. To this end we will construct a suitable approximation of $u$ in $\wt {\mathcal O}_{i}$ based on the fact that the true solution is locally close to the heteroclinic. 
By symmetry we may focus on the case $i=1$, namely consider the lower half plane. We chose a solution $u$, whose nodal line $\Gamma_{\ve, 1}$ in the lower half plane is a graph of $y=f_{\ve, 1}(x)$ and chose the solution of the Toda system $q_{\ve, 1}(x)$ such that the assertions of the Proposition \ref{prop gamma} are satisfied. 
We let $\wt \eta$ to be a smooth  cut off function  equal to $1$ in $\wt {\mathcal O}_{1}\cap\{\mathrm{dist}({\tt x}, \partial\wt {\mathcal O}_1)>1\}$ and equal to $0$    in $\R^2\setminus \wt{\mathcal O}_{1}$. A reasonable ansatz  for an approximate solution is built by defining   the function $\wt H_{\ve, 1}$:
\[
(\wt {\tt x}_{\ve, 1}^*\wt  {H}_{\ve, 1})\left( \wt  x_1,\wt y_1\right)  :=(\wt{\tt x}_{\ve, 1}^*\wt{\eta})\left(\wt x_1, \wt y_1\right)
H\left(  \wt y_{1}-\wt g_\ve \left(  \wt x_{1}\right)  \right)  +\left(  1-(\wt{\tt x}_{\ve, 1}^*{\wt\eta})\left(\wt x_1, \wt y_1\right)  \right)  \frac{H\left(  \wt y_{1}-\wt g_\ve\left(
\wt x_{1}\right)  \right)  }{\left\vert H\left(  \wt y_{1}-\wt {g}_\ve\left(
\wt x_{1}\right)  \right)  \right\vert }, 
\]
which is extended to the whole $\R^2$ by $\pm 1$, setting  
 $\wt {H}_{\ve, 2}\left(  x,y\right)  =-\wt {H}_{\ve, 1}\left(  x,-y\right)$,  and finally  defining
\begin{equation}
\wt {{u}}_\ve:=\wt{H}_{\ve, 1}-\wt{H}_{\ve, 2}-1. 
\label{wtueps}
\end{equation}
Note that the function $\wt g_\ve$ has not been specified so far. It turns out that in  order to have a good approximation of $u$ by $\wt u$ we should impose the following orthogonality condition:
\begin{equation}
\int_{\mathbb{R}}\left [\wt {\tt x}_{\ve, i}^{\ast}(  u-\wt{u}_\ve)
\wt \rho_{\ve, i}\wt {H}_{\ve, i}^{\prime}\right](\wt x_i, \wt y_i)\,d\wt y_{i}=0,\quad \forall\, \wt x_i, \quad i=1,2.
\label{wt orto cond}
\end{equation}
and smooth cutoff functions  $\wt\rho_{\ve, i}$ are defined through a smooth cutoff function $\wt\rho$ by:
\[
(\wt{\tt x}_{\ve, i}^* \wt\rho_{\ve, i})(\wt x_i, \wt y_i)=\wt \rho(\wt x_i, \wt y_i-(-1)^{i+1}\wt {g}_\ve(\wt x_i)),
\]
where 
\[
\wt \rho(s,t)=\begin{cases}
1, \quad |t|\leq \frac{1}{2}\wt{\mathrm{d}}_\ve (s),\\
0<\rho<1, \quad \frac{1}{2} \wt{\mathrm{d}}_\ve(s)<t<\frac{3}{4}\wt {\mathrm{d}}_\ve(s),\\
0 \quad\mbox{othewise}.
\end{cases}
\]
%
%
c.f. definition of $\hat\rho_{\ve, i}$ above, while $\wt {H}_{\ve, i}^{\prime}$ is defined by
\begin{align*}
\left(\wt {\tt x}^*_{\ve, i}\wt {H}_{i}^{\prime}\right) \left(  \wt x_{i}, \wt y_{i}\right)  &  =H^{\prime}\left(
\wt y_{i}-\left(  -1\right)  ^{i}\wt {g}_\ve\left(  \wt x_{i}\right)  \right)  .
\end{align*}
Note that because of the definition of the function $\wt {\mathrm {d}}_\ve$ we can assume that $|\nabla \wt \rho_{\ve, i}({\tt x})|=\mathcal{O}(\frac{1}{\ve|{\tt x}|})$, $\ve|\tt x|\gg 1$ with similar estimates for higher order derivatives. 
Changing variables $X_i=\wt x_i$, $Y_i= \wt y_i-(-1)^i\wt g_\ve(\wt x_i)$ one gets the following, equivalent form of (\ref{wt orto cond}):
\begin{equation}
\label{l1}
\int_{\mathbb{R}}\wt\rho\left(X_i, Y_{i}\right)  H^{\prime}\left(  Y_{i}\right)
\left(\wt{\tt x}_{\ve, i}^{\ast}(  u-{\wt{u}}_\ve) \right) \left(
X_{i},Y_{i}+\wt{g}_\ve \left(  X_{i}\right)  \right)  dY_{i}=0,\quad \forall X_i, \quad i=1,2.
\end{equation}
To show  the existence of the function $\wt g_\ve$ one can use the argument similar to the one in Lemma \ref{exists h}. However, since the set $\{y=q_{\ve, i}(x)\}$ does not coincide with the   nodal set of the  solution the function $\wt g_\ve$ does not decay exponentially.   To determine the behavior of the function $\wt g_\ve$ more precisely we need the following:
\red{\begin{lemma}
\label{h1}There exist a ${\tau}>0$ and a constant $\wt {j}_{\ve}$\bigskip\ such that
$\left\vert \wt j_{\ve}\right\vert \leq C\varepsilon^{\alpha}$, and the function
$\wt{h}_\ve\left(  x\right)  :=\wt {g}_\ve\left(  x\right)  -\wt j_{\ve}$
satisfies%
\begin{equation}
\label{2ndest}
\begin{aligned}
\left\Vert \wt {h}_\ve\right\Vert _{C_{\varepsilon{\tau}}^{0,
\mu}\left(
\R\right)  } &  \leq C\varepsilon^{\alpha},\\
\left\Vert \wt {h}_\ve^{\prime}\right\Vert _{C_{\varepsilon{\tau}}%
^{0,\mu}\left(  \R\right)  } &  \leq C\varepsilon^{1+\alpha},\\
\left\Vert \wt{h}_\ve^{\prime\prime}\right\Vert _{C_{\varepsilon{\tau}%
}^{0,\mu}\left(  \R\right)  } &  \leq C\varepsilon^{2+\alpha}.
\end{aligned}
\end{equation}

\end{lemma}

\begin{proof}
As we know,
$\wt{g}_\ve$ is determined by (\ref{wt orto cond}).  By symmetry it suffices to consider the case $i=1$ in (\ref{wt orto cond}). We write:
\begin{equation}
\label{rel xxx}
0=\int_{\mathbb{R}}\left [\wt {\tt x}_{\ve, i}^{\ast}(  u-\wt{u}_\ve)
\wt \rho_{\ve, 1}\wt {H}_{\ve, i}^{\prime}\right]\,d\wt y_{1}=\underbrace{\int_{\mathbb{R}}\left [\wt {\tt x}_{\ve, 1}^{\ast}( \bar u_\ve-\wt{u}_\ve)
\wt \rho_{\ve, i}\wt {H}_{\ve, 1}^{\prime}\right]\,d\wt y_{1}}_{A}+\underbrace{\int_{\mathbb{R}}\left [\wt {\tt x}_{\ve, 1}^{\ast}(  u-\bar{u}_\ve)
\wt \rho_{\ve, 1}\wt {H}_{\ve, 1}^{\prime}\right]\,d\wt y_{1}}_{B}
\end{equation}
With some abuse of notation we will write the first term in the form:
\begin{align*}
A&=\underbrace{\int_\R\wt \rho H'(\wt(y_1)[H(\wt y_1-\wt g_\ve(\wt x_1))-H(y_1-h_\ve(x_1))]\,d \wt y_1}_{A_1}\\
&\quad+\underbrace{\int_\R \wt\rho H'(\wt y_1)[H(y_2+h_\ve(x_2))-H(\wt y_2+\wt g_\ve(\wt x_2))]\,d\wt y_1}_{A_2} + A_3.
\end{align*}
The last term denoted by $A_3$ is negligible. We will now calculate $A_1$ and $A_2$. 

To compute  $A_1$ we introduce a diffeomorphism $\Phi_{\ve,1}={\tt x}_{\ve, 1}\circ\wt {\tt x}_{\ve, 1}$. Then we  can write:
\begin{equation}
\label{A1}
\begin{aligned}
A_1&=\int_\R\wt \rho\big(H'(\wt y_1)\big)^2\underbrace{[\Phi^{-1}_{\ve,1, 2}(\wt x_1, \wt y_1)-\wt y_1-h_\ve\circ\Phi_{\ve, 1,1}^{-1}(\wt x_1, \wt y_1)+\wt g_\ve]}_{\wt \varDelta_{\ve,1}}\,d\wt y_1+\int_\R \wt \rho \mathcal{O}(e^{\,-\sigma|\wt y_1|})\wt\varDelta_{\ve,1}^2(\wt x_1, \wt y_1)\,d\wt y_1\\
&=\wt c (\wt j_\ve-h_\ve(\wt x_1)+\wt g_\ve(\wt x_1))+\mathcal{O}(|\wt j_\ve-h_\ve(\wt x_1)+\wt g_\ve(\wt x_1)|^2)+\mathcal{O}_{\mathcal{C}^{0,\mu}_{\ve\tau}(\R)}(\ve^\alpha).
\end{aligned}
\end{equation}
Similarly we calculate:
\begin{equation}
\label{A2}
\begin{aligned}
A_2&=\int_\R\wt \rho H'(\wt y_1)H'(\wt y_2)\underbrace{[\Phi_{\ve, 2,2}^{-1}(\wt x_2, \wt y_2)-\wt y_2+h_\ve\circ\Phi_{\ve, 2,1}^{-1}(\wt x_2, \wt y_2)-\wt g_\ve]}_{\wt\varDelta_{\ve, 2}}\,d\wt y_1+\int_\R \wt \rho \mathcal{O}(e^{\,-\sigma|\wt y_1|})\wt\varDelta_{\ve,2}^2(\wt x_1, \wt y_1)\,d\wt y_1\\
&=\mathcal{O}_{\mathcal{C}^{0,\mu}_{\ve\tau}(\R)}(\ve^\alpha).
\end{aligned}
\end{equation}
Estimates (\ref{A1})--(\ref{A2}) follow from formulas (\ref{rel var})--(\ref{63}). Finally, it can be shown that the leading order term in $B$ is the following
\begin{equation}
\label{B}
\begin{aligned}
\|\wt B\|_{\mathcal{C}^{0}_{\ve\tau}(\R)}&=\left\|\int_\R \wt \rho H'(\wt y_1)[\wt {\tt x}_{\ve, 1}(u-\bar u_\ve)]\,d\wt y_1+\int_\R \wt \rho H'(\wt y_1)[\wt {\tt x}^*_{\ve, 1}\phi]\,d\wt y_1\right\|_{\mathcal{C}^{0}_{\ve\tau}(\R)}\\
&=\left\|\int_\R\wt\rho\circ\Phi_{\ve, 1}(x_1, y_1)\big(H'\circ \Phi_{\ve, 1,2}(x_1,y_1)\big)[\wt {\tt x}^*_{\ve, 1}\phi]\circ\Phi_{\ve, 1}(x_1, y_1)\,d\Phi_{\ve, 1,2}(x_1, y_1)\,dy_1\right\|_{\mathcal{C}^{0}_{\ve\tau}(\R)}\\
&\leq \left\|\int_\R\wt\rho\circ\Phi_{\ve, 1}(x_1, y_1)\big(H'\circ \Phi_{\ve, 1,2}(x_1,y_1)-\rho(x_1,y_1)H'(y_1)\big)[{\tt x}^*_{\ve, 1}\phi](x_1, y_1)\,d\Phi_{\ve, 1,2}(x_1, y_1)\,dy_1\right\|_{\mathcal{C}^{0}_{\ve\tau}(\R)}\\
&\quad +\left\|\int_\R\rho(x_1,y_1)H'(y_1)\big)[{\tt x}^*_{\ve, 1}\phi](x_1, y_1)\,[d\Phi_{\ve, 1,2}(x_1, y_1)-1]\,dy_1\right\|_{\mathcal{C}^{0}_{\ve\tau}(\R)}
\\
&=\mathcal{O}\big(\ve^\alpha\|\phi\|_{\mathcal{C}^{0,\mu}_{\ve\tau}(\R^2)}\big)\leq C\ve^{2+\alpha}.
\end{aligned}
\end{equation}
From this, combining (\ref{A1})--(\ref{B}) we obtain the first estimate in (\ref{2ndest}). The remaining estimates are obtained similarly using the fact that the relation (\ref{rel xxx}) can be differentiated twice with respect to $\wt x_1$. The details are omitted.
\end{proof}

Given a solution of (\ref{AC}) $u$, such that $\theta(u)=\ve$ we can define an approximate solution $\wt u_\ve$ using the solution of the Toda system with the same asymptotic angle. This solution is unique. Now we can write:
\[
u=\wt u_\ve+\wt \phi.
\]
By definition of the function $\wt g_\ve$ we know that $\wt\phi=u-\wt u_\ve$ satisfies the orthogonality condition (\ref{wt orto cond}). This allows us the control the size of $\wt \phi$ in the weighted norm in terms of the error of the approximation:
\[
E(\wt u_\ve)=\Delta \wt u-F'(\wt u_\ve),
\]
following essentially the same approach as in section \ref{proof {prop gamma}} and in particular relying on a version of Proposition \ref{apriori}. Based in this one can prove:
\begin{equation}
\label{wtphi est}
\|\wt \phi\|_{\mathcal{C}^{2,\mu}_{\ve\tau}(\R^2)}\leq C\ve^2.
\end{equation}

}

\subsection{Conclusion of the proof: the Lipschitz property of solutions}\label{last step}

With the above preparations, we are ready to prove our uniqueness theorem. Based on the results of the previous section we know that any solution with a small asymptotic angle can be written in the following way:
\[
u(\cdot;\wt g_\ve, \wt \phi)=\wt u_\ve(\cdot;\wt g_\ve)+\wt \phi,
\]
where $\wt u_\ve$ is the approximate solution defined in (\ref{wtueps}). Here and below we will indicate the dependence of this solution on the modulation function $\wt g_\ve$ as well as on $\wt \phi$. 
Now, let us consider two solutions $u^{(j)}$, $j=1,2$ with the same asymptotic angle $\theta(u^{(j)})=\ve$. 
Sine the asymptotic angle is the same for both solutions, there is just one solution of the Toda system represented by the functions $q_{\ve,1}=-q_{\ve, 2}$. On the other hand it may happen that $\wt g^{(1)}_\ve\neq \wt g_\ve^{(2)}$, and $\wt\phi^{(1)}\neq \wt \phi^{(2)}$.    In the language of \cite{MR2557944} we have that $\wt g_\ve^{(j)}\in \mathcal{C}^{2,\mu}_{\ve\tau}(\R)\oplus D$ (see also section \ref{exists small eps}) and in the previous section we have shown that
\[
\|\wt g^{(j)}_\ve\|_{\mathcal{C}^{0,\mu}_{\ve\tau}(\R)\oplus D}\leq C \ve^\alpha,
\]
with similar estimates for the higher order derivatives. In addition for the functions $\wt \phi^{(j)}$ we have (\ref{wtphi est}). 
To prove the uniqueness of solutions with small angles is therefore enough to prove "local uniqueness" in the following sense: given two solutions associated to the same solution of the Toda system we have $\wt\phi^{(1)}=\wt \phi^{(2)}$, and $\wt g_\ve^{(1)}=\wt g_\ve^{(2)}$. Our strategy to prove this fact follows in some sense the strategy used to prove the existence of solutions with small angles employed in  \cite{MR2557944}. Namely, we show the Lipschitz property of the map: $\wt g_\ve\mapsto E(\wt u_\ve(\cdot; \wt g_\ve))^\perp$ and then we use the linearized  equation to show that $\wt \phi^{(1)}-\wt \phi^{(2)}$ can be controlled by a small constant times $\wt g_\ve^{(1)}-\wt g_\ve^{(2)}$.   As a final step we show that the function $\wt g_\ve^{(1)}-\wt g_\ve^{(2)}$ satisfies the linearized Toda system with the right hand side again controlled by a small constant times $\wt g_\ve^{(1)}-\wt g_\ve^{(2)}$. This leads us to conclude that $\wt g_\ve^{(1)}-\wt g_\ve^{(2)}=0$, and a s a result we infer the uniqueness.

Now we will present some details of the argument outlined above. Many of the calculations are quite similar to the ones in \cite{MR2557944}.

For future purpose it is convenient to introduce the following projection defined for any function $\psi\colon\R^2\to \R$:
\[
\psi^{\perp_{(j)}}=\psi-\wt {\tt x}_{\ve, i}\circ \sum_{i=1,2}\wt c^{(j)}_i\wt \rho^{(j)}_{\ve, i}\wt {H}_{\ve, i}^{{(j)}\prime}\int_{\mathbb{R}}\left [\wt {\tt x}_{\ve, i}^{*}\psi
\wt \rho_{\ve, i}^{(j)}\wt {H}_{\ve, i}^{{(j)}\prime}\right](\wt x_i, \wt y_i)\,d\wt y_{i}, \quad \wt c^{(j)}_i=\left(\int_\R[\wt{\tt x}_{i,\ve}^{*}(\wt \rho_{\ve, i}^{(j)}\wt {H}_{\ve, i}^{{(j)}\prime})]^2(\wt x_i, \wt y_i)\,d\wt y_i\right)^{-1},
\]
where $j=1,2$.
%
%
%
%

\red{
\begin{lemma}\label{lip error}
The following estimates hold:
\begin{align}
\label{lip error 1}
\big\|\big[E(\wt u^{(1)}(\cdot; \wt g_\ve^{(1)})-E(\wt u^{(2)}(\cdot; \wt g_\ve^{(2)})\big]^{\perp_{(1)}}\big\|_{\mathcal{C}^{0,\mu}_{\ve\tau}(\R^2)} & \leq o(1)\|\wt g_\ve^{(1)}-\wt g_\ve^{(2)}\|_{\mathcal{C}_{\ve\tau}^{2,\mu}(\R)\oplus D}\\
\label{fil}
\|\wt  \phi^{(1)}-\wt \phi^{(2)}\| _{\mathcal{C}^{2,\mu}_{\varepsilon\tau
}(\mathbb{R}^{2})}& \leq o(1)\|\wt g_\ve^{(1)}-\wt g_\ve^{(2)}\|_{\mathcal{C}_{\ve\tau}^{2,\mu}(\R)\oplus D}
\end{align}
\end{lemma}

\begin{remark}
Essentially, up to some minor difference, this Lipschitz property has already
been proved in \cite{MR2557944}. Here we give a sketch of the proof for completeness.
\end{remark}

\begin{proof}
To prove the first of the estimates we use essentially the formula for the error (\ref{eu}), replacing $\bar u_\ve$ by $\wt u_\ve^{(j)}$, $j=1,2$ and then taking  the difference of the resulting terms $E(\wt u^{(j)}(\cdot; \wt g_\ve^{(1)}))$, and finally taking the projection ${}^{\perp_{(1)}}$. Then we write:
\begin{align*}
\big[E(\wt u^{(1)}(\cdot; \wt g_\ve^{(1)})-E(\wt u^{(2)}(\cdot; \wt g_\ve^{(2)})\big]^{\perp_{(1)}}&=\big[E(\wt u^{(1)}(\cdot; \wt g_\ve^{(1)})^{\perp_{(1)}}-E(\wt u^{(2)}(\cdot; \wt g_\ve^{(2)})^{\perp_{(2)}}\big]\\
&\quad +\big[E(\wt u^{(2)}(\cdot; \wt g_\ve^{(2)})^{\perp_{(2)}}-E(\wt u^{(2)}(\cdot; \wt g_\ve^{(2)})^{\perp_{(1)}}\big].
\end{align*}
To estimate this expression we use calculations based on formula (\ref{eu}) and similar to the ones in  Lemma \ref{Eu}. 

To show (\ref{fil}) we  should consider the equation satisfied by the difference $\wt \psi=\wt\phi^{(1)}-\wt\phi^{(2)}$ and use Proposition \ref{apriori}. The slight technical problem is that $\wt \psi$ does not satisfy the orthogonality condition as in  (\ref{wt orto cond}). To overcome this we further define function $\wt\psi^\perp$ by
\[
\wt \psi^\perp=\wt\psi-\wt {\tt x}_{\ve, i}\circ \sum_{i=1,2}\wt c_i\wt \rho^{(1)}_{\ve, i}\wt {H}_{\ve, i}^{{(1)}\prime}\int_{\mathbb{R}}\left [\wt {\tt x}_{\ve, i}^{*}\wt\psi
\wt \rho_{\ve, i}^{(1)}\wt {H}_{\ve, i}^{{(1)}\prime}\right](\wt x_i, \wt y_i)\,d\wt y_{i}, \quad \wt c_i=\left(\int_\R[\wt{\tt x}_{i,\ve}^{*}(\wt \rho_{\ve, i}^{(1)}\wt {H}_{\ve, i}^{{(1)}\prime})]^2(\wt x_i, \wt y_i)\,d\wt y_i\right)^{-1}.
\]
For later use we set $\wt\psi^\parallel=\wt\psi-\wt\psi^\perp$.
It is not hard to show that 
\[
\|\wt\psi^\parallel\|_{\mathcal{C}^{2,\mu}_{\ve\tau}(\R^2)}\leq  C\ve^2\|\wt g_\ve^{(1)}-\wt g_\ve^{(2)}\|_{\mathcal{C}_{\ve\tau}^{2,\mu}(\R)\oplus D},
\]
hence
\[
\|\wt\psi^\perp\|_{\mathcal{C}^{2,\mu}_{\ve\tau}(\R^2)}\geq \|\wt \psi\|_{\mathcal{C}^{2,\mu}_{\ve\tau}(\R^2)}-C\ve^2\|\wt g_\ve^{(1)}-\wt g_\ve^{(2)}\|_{\mathcal{C}_{\ve\tau}^{2,\mu}(\R)\oplus D}.
\]
On the other hand denoting:
\[
L^{(i)}=-\Delta+F^{\prime\prime}(\wt{u}_{\ve}^{(i)}),\quad P(\phi^{(i)})  =F^{\prime}(\wt{u}^{(i)}_\ve+\phi^{(i)})  -F^{\prime
}(\wt{u}^{(i)}_\ve)  -F^{\prime\prime}(\wt{u}^{(i)}_\ve), \quad i=1,2
\]
we get:
\begin{equation}
\label{lonepsi}
L^{(1)}\wt\psi^\perp=\underbrace{E(\wt u_\ve^{(1)})-E(\wt u_\ve^{(2)})+P(\phi^{(1)})-P(\phi^{(2)})+(L^{(1)}-L^{(2)})\phi^{(2)}+L^{(1)}\wt\psi^\parallel}_{\wt f}.
\end{equation}
The $\mathcal{C}^{0,\mu}_{\ve\tau}(\R^2)$ norm of the right hand side of this equation can be estimated by 
\[
\|\wt f\|_{\mathcal{C}^{0,\mu}_{\ve\tau}(\R^2)}\leq o(1)\|\wt \psi\|_{\mathcal{C}^{2,\mu}_{\ve\tau}(\R^2)}+C\ve^2\|\wt g_\ve^{(1)}-\wt g_\ve^{(2)}\|_{\mathcal{C}_{\ve\tau}^{2,\mu}(\R)\oplus D}.
\]
From this the required estimate follows. The proof of the Lemma is complete.
\end{proof}}

\red{Now we are in position to conclude the proof of uniqueness of solutions with small angles. To this end we consider the following quantity (c.f. the proof of Lemma \ref{fi}):
\[
\mathcal{T}=\int_{\mathbb{R}}\big[\wt {\tt x}_{\ve, i}^{*}\big(E(\wt u_\ve^{(1)})\wt \rho_{\ve, i}^{(1)}\wt {H}_{\ve, i}^{{(1)}\prime}\big](\wt x_i, \wt y_i)\,d\wt y_{i}-\int_\R \big[\wt {\tt x}_{\ve, i}^{*}E(\wt u_\ve^{(2)})
\wt \rho_{\ve, i}^{(2)}\wt {H}_{\ve, i}^{{(2)}\prime}\big](\wt x_i, \wt y_i)\,d\wt y_{i}
\]
First we proceed as in Step 1 in the proof of Lemma \ref{fi}. This leads to the following estimate:
\begin{equation}
\label{ppp}
|{\mathcal T}|\leq C\ve^{2+\alpha}\|\wt g^{(1)}_\ve-\wt g^{(2)}_\ve\|_{\mathcal {C}^{2,\mu}_{\ve\tau}(\R)\oplus D}+
C\ve^2\|\wt g^{(1)'}_\ve-\wt g^{(2)'}_\ve\|_{\mathcal {C}^{1,\mu}_{\ve\tau}(\R)}.
\end{equation}
For brevity let us denote $\wt h_\ve=\wt g^{(1)}_\ve-\wt g^{(2)}_\ve$. Now we calculate $\mathcal{T}$ using the explicit expressions for $\wt u_{\ve}^{(i)}$ in a manner similar to Step 2 of Lemma \ref{fi} and  as a result we get a formula which is similar to (\ref{t24}), except the $f_{\ve, i}$ is replaced by $q_{\ve, i}$, a solution to the Toda system. Thus, calculating ${\mathcal T}$ in two ways we get at the end:
\begin{equation}
\label{ppp1}
\bar c^* \wt h_\ve '' + 2\sqrt{2}c_*e^{\,2\sqrt{2}q_{\ve,1}}\wt h_\ve= \mathbb{G}_\ve(\wt h_\ve), \quad \|\mathbb{G}(\wt h_\ve)\|_{\mathcal {C}^{0,\mu}_{\ve\tau}(\R)\oplus D}\leq C\ve^{2+\alpha}\|\wt h_\ve\|_{\mathcal {C}^{2,\mu}_{\ve\tau}(\R)\oplus D}+C\ve^2\|\wt h'_\ve\|_{\mathcal {C}^{1,\mu}_{\ve\tau}(\R)\oplus D}.
\end{equation} 

To proceed, we need to analyze the linearized Toda system. Remember that we
are always working in the space of even functions. Suppose $q$ is an even  solution
of the Toda system:
\[
\xi^{\prime\prime}\left(  t\right)  =-\frac{c^{\ast}}{c_{\ast}}e^{2\sqrt{2}%
q\left(  t\right)  },
\]
and the linearized operator is
\[
\mathcal{P}:\varphi\rightarrow\varphi^{\prime\prime}\left(  t\right)
+2\sqrt{2}\frac{c^{\ast}}{c_{\ast}}e^{2\sqrt{2}q}\varphi\left(  t\right)  .
\]
We want to know the mapping property of this operator. Let $\mathcal{C}_{\tau}^{2,\mu}(  \mathbb{R})_{e}$ be the space of even functions in $\mathcal{C}_{\tau}^{2,\mu}(  \mathbb{R})$, and $D$ be the one dimensional  deficiency space  spanned by the constant function.

\begin{lemma}
The map $\mathcal{P}:$ $C_{\tau}^{2,\mu}(  \mathbb{R})_{e}  \oplus D\rightarrow C_{\tau}^{0,\mu}\left(  \mathbb{R}\right)  $ is an
isomorphism and therefore has a bounded inverse.
\end{lemma}
This result has already been proven in  \cite{MR2557944}. Now we simply need to adopt  the above Lemma to the present context, where the functions involved depend on $\ve$ as well. This can be done by a simple scaling argument if we define:
\[
\wt h(x)=\wt h_\ve\big(\frac{x}{\ve}\big),
\]
and then write the resulting equation for $\wt h$:
\[
c^*\wt h''+2\sqrt{2}c_*e^{\,2\sqrt{2}q_{1}}\wt h = \ve^{-2} \mathbb{G}_\ve(\wt h(\ve^{-1}\cdot)),
\]
where $q=(q_1, q_2)$ is the even solution of the Toda system whose scattering behavior at corresponds to the lines $y=\pm |x|$. Now, using the fact that $\wt h_\ve'=\ve\wt h'$, $\wt h_\ve''=\ve^2\wt h''$ also that 
\[
\|\wt h\|_{\mathcal{C}^{0,\mu}_{\tau}(\R)}=\ve^{-\mu}\|\wt h_\ve\|_{\mathcal{C}^{0,\mu}_{\ve\tau}(\R)},
\]
we  get (c.f. a similar scaling  argument in \cite{MR2557944}) we get:
\begin{equation}
\label{qqq x}
\|\wt h\|_{\mathcal{C}^{2,\mu}_{\tau}(\R)}\leq C\ve^{\alpha-\mu}\|\wt h\|_{\mathcal{C}^{2,\mu}_{\tau}(\R)},
\end{equation}
from which we get $\wt h=0$ provided that $\mu<\alpha$ and $\ve$ is taken small. This in turn implies $\wt g^{(1)}_\ve=\wt g^{(2)}_\ve$ and $\wt \phi^{(1)}=\wt \phi^{(2)}$, hence we get uniqueness.
This ends the proof of Theorem  \ref{teo uniqueness}.

}

\end{document}